\documentclass[12pt,reqno]{amsart}
\usepackage{}
\usepackage{mathrsfs}
\usepackage{amssymb}

\usepackage{fullpage}

\usepackage{amsthm}
\usepackage{mathrsfs}
\usepackage[centertags]{amsmath}
\usepackage{amsfonts}

\usepackage{color}

\usepackage[colorlinks=true, pdfstartview=FitV,linkcolor=blue, citecolor=blue]{hyperref}
\input amssym.def
\input amssym.tex

\date{}

\newtheorem{Theorem}{Theorem}[section]

\newtheorem{Lemma}{Lemma}[section]

\newcommand\R{\mbox{\bf R}}
\newcommand\C{\mbox{\bf C}}
\newcommand\N{\mbox{\bf N}}
\newcommand\Z{\mbox{\bf Z}}
\newcommand\z{\mbox{\bf z}}
\newcommand\T{\mbox{\bf T}}
\newcommand\SR{\mbox{\scriptsize\bf R}}
\newcommand\ST{\mbox{\scriptsize\bf T}}

\newcommand{\definition}{{\lower .5ex
  \hbox{$\>\>\stackrel{\triangle}{=}\>\>$} }}
\newcommand\supp{\mathop{\rm supp}}

\allowdisplaybreaks[4] 

\begin{document}

\baselineskip=22pt
\thispagestyle{empty}
\mbox{}
\bigskip

\begin{center}
{\Large \bf Strichartz estimates for      orthonormal  functions and  probabilistic
 convergence of density functions of  compact operators on manifolds}\\[1ex]

{Wei Yan\footnote{Email:011133@htu.edu.cn}$^a$,
Jinqiao Duan\footnote{Email:Duan@gbu.edu.cn}$^{b}$,
 Jianhua Huang\footnote{Email: jhhuang@nudt.edu.cn}$^{c}$,
 Haoyuan Xu\footnote{Email: hyxu@hust.edu.cn}$^{d}$,
Meihua Yang\footnote{Email: yangmeih@hust.edu.cn}$^{d}$}\\[1ex]

{$^a$School of Mathematics and Information Science, Henan
Normal University,}\\
{Xinxiang, Henan 453007,   China}\\[1ex]
{$^b$ Department of Mathematics, Great Bay University, Dongguan, Guangdong 523000, China}

{$^c$College of Science, National University of Defense  Technology,}\\
{ Changsha, Hunan 410073,  China}\\[1ex]

{$^d$ School of Mathematics and Statistics, Huazhong University
 of Science and Technology, Wuhan, 430074, China}
\end{center}

\noindent{\bf Abstract.}In this paper, we establish
 some  Strichartz estimates  for   orthonormal  functions and their applications as well as
probabilistic convergence of density
 functions related to some compact operators on manifolds.
Firstly,  we  present the suitable bound of
$\int_{a\leq|s|\leq b}e^{isx}s^{-1+i\beta}ds(\beta \in \R,a\geq0,0<b\leq \infty),$
$\int_{a\leq s\leq b}e^{isx}s^{-1+i\beta}ds(\beta \in \R,\beta\neq0,a,b\in \R)$,
which extends the result of Page 204  of
 Vega (Restriction theorems and the Schr\"odinger multiplier
  on the torus.
 Partial differential equations
  with minimal smoothness and applications,  199-211,  1992.).
  Secondly, we prove that  $\left|\beta\int_{a}^{b}e^{isx}s^{-1+i\beta}ds\right|
\leq C(1+e^{-\pi \beta})(1+|\beta|)^{2}(\beta \in \R,a,b\in \R),$
  which extends Lemma 1 of Bez et al.
(Forum of Mathematics, Sigma, 9(2021), 1-52).
  Thirdly, we establish some Strichartz estimates for orthonormal
  functions
 and Schatten bound with space-time norms
related to  elliptic operator and non-elliptic operator  on $\R^{d},$
 which   extend the result of Theorems 8, 9 of
 Frank and Sabin (Amer. J. Math. 139(2017), 1649-1691.);
 we also establish some Strichartz estimates for orthonormal functions
 and Schatten bound with space-time norms on $\T^{d}$ related to
 non-elliptic operator.
 Fourthly, we  establish the Strichartz estimates for
  orthonormal  functions related to Boussinesq operator
 on $\R$.
 Fifthly,  we
 establish  the convergence
 result of some compact operators
 and    nonlinear
 part of the solution to
some operator equations in Schatten norms.
Finally, inspired by  the work of  Hadama  and   Yamamoto
 (Probabilistic Strichartz
estimates in Schatten classes and
 their applications to Hartree equation, arxiv:2311.02713v1.),
 for $\gamma_{0}\in \mathfrak{S}^{2}$ which is  the Hilbert-Schmidt class,
  we establish the probabilistic convergence
of density functions  of some compact operators on
 $\R^{d},\T^{d},\Theta=\left\{x\in \R^{3}:|x|<1\right\}$  with full
 randomization, which improves the result of Corollary 1.2 of
Bez et al. (Selecta Math.  26(2020),  24 pp) with $x\in \R$
in the probabilistic sense.

 \medskip


\medskip
\noindent {\bf Corresponding Author:}Meihua Yang

\medskip
\noindent {\bf Email Address:}yangmeih@hust.edu.cn

\medskip
\noindent {\bf 2020-Mathematics Subject Classification}:
Primary-35Q41, 42B20; Secondary-47B10

\leftskip 0 true cm \rightskip 0 true cm

\newpage

\baselineskip=20pt

\bigskip
\bigskip

\tableofcontents

\section{Introduction}
\medskip

\setcounter{Theorem}{0} \setcounter{Lemma}{0}\setcounter{Definition}{0}\setcounter{Proposition}{1}

\setcounter{section}{1}

We begin by studying the infinite system of some dispersive
equation with Hartree-type nonlinearity in the following form:
\begin{eqnarray}
&&i\frac{du_{j}}{dt}=P(D)u_{j}+(w*\rho) u_{j},j\in \N^{+},x\in \R^{d},t\in I\subset \R,\label{1.01}\\
&&u_{j}(x,0)=f_{j}(x).\label{1.02}
\end{eqnarray}
Here $f_{j}$ forms an orthonormal family in $L^{2}(\R^{d})$, $w(x):\R^{d}\rightarrow \R$,
$\rho=\sum\limits_{j=1}^{+\infty}|u_{j}|^{2}$ and
$P(D)$ is defined as follows:
\begin{eqnarray}
P(D)f(x)=\int_{\SR^{d}}e^{ix\cdot\xi}P(\xi)\mathscr{F}_{x}(\xi)d\xi,D=
\frac{1}{i}(\partial_{x_{1}},\cdot\cdot\cdot,\partial_{x_{d}}),\label{1.03}
\end{eqnarray}
where $x\cdot\xi=\sum\limits_{j=1}^{d}x_{j}\xi_{j}$.
By using a direct computation, we derive the operator valued form
 of (\ref{1.01})
\begin{eqnarray}
&&i\frac{d \gamma(t)}{dt}=[P(D)+w*\rho_{\gamma},\gamma],\label{1.04}\\
&&\gamma(0)=\gamma_{0}.\label{1.05}
\end{eqnarray}
Here $\gamma(t)=\sum\limits_{j=1}^{\infty}|u_{j}\rangle \langle u_{j}|$ is
 bounded and self-adjoint operator
 on $L^{2}(\R^{d})$, $\rho_{\gamma}$
is the density function of $\gamma$.

When $w=0$  in (\ref{1.01}) and (\ref{1.04})
 and  $\gamma_{0}=
\sum\limits_{j=1}^{\infty}\lambda_{j}|f_{j}\rangle\langle f_{j}|$, we have
that the evolved operator
\begin{eqnarray}
\gamma(t):=e^{-itP(D)}\gamma_{0}e^{itP(D)}=
\sum\limits_{j=1}^{\infty}\lambda_{j}
|e^{-itP(D)}f_{j}\rangle\langle e^{-itP(D)}f_{j}|.\label{1.06}
\end{eqnarray}
Here we have used Dirac's notation  for the rank-one operator
  $|u\rangle\langle v|$
$f\longrightarrow\langle v,f\rangle u=u\int_{I}\overline{v}fdx$,
 where $I\subset \R^{d}.$
Because the $f_{j}(j\in \N^{+})$ form an orthonormal system,
then $\lambda_{j}$ are precisely the
eigenvalues of the operator $\gamma$.
By using a direct computation, we derive the kernel of $\gamma (t)$ as
 follows:
\begin{eqnarray}
K(x,y,t)=\sum\limits_{j=1}^{\infty}\lambda_{j} (e^{-itP(D)}f_{j})(x,t)
(\overline{e^{-itP(D)}f_{j}})(y,t).\label{1.07}
\end{eqnarray}
Thus, the density function of $\gamma(t)$ is $\rho_{\gamma(t)}=K(x,x,t)=
\sum\limits_{j=1}^{\infty}\lambda_{j} |e^{-itP(D)}f_{j}|^{2}$.
Thus, we have
\begin{eqnarray}
&&Trace(\gamma(t))=\int_{\SR^{d}}K(x,x,t)dx=
\sum\limits_{j=1}^{\infty}\lambda_{j}\int_{\SR^{d}}|e^{-itP(D)}f_{j}|^{2}dx\nonumber\\&&
=\sum\limits_{j=1}^{\infty}\lambda_{j}\|f_{j}\|_{L^{2}}^{2}=
\sum\limits_{j=1}^{\infty}\lambda_{j}.\label{1.08}
\end{eqnarray}
One  of our goal in  this paper is to establish the Strichartz estimates
 for density function of $\gamma(t)$:
\begin{eqnarray}
&&\left\|\sum\limits_{j=1}^{\infty}\lambda_{j}
|e^{-itP(D)}f_{j}(x,t)|^{2}\right\|_{L_{t\in I_{1}}^{p}L_{x\in I_{2}}^{q}}
=\left\|\rho_{\gamma(t)}\right\|_{L_{t\in I_{1}}^{p}L_{x\in I_{2}}^{q}}\nonumber\\
&&=\left[\int_{I_1}
\left(\int_{I_{2}}\left|\rho_{\gamma(t)}\right|^{q}dx\right)^{p/q}dt\right]^{\frac{1}{p}}\nonumber\\
&&
\leq C\left[\sum\limits_{j=1}^{\infty}
|\lambda_{j}|^{\alpha}\right]^{\frac{1}{\alpha}}\leq
C\left\|\gamma(t)\right\|_{\mathfrak{S}^{\alpha}}(1\leq \alpha<\infty)\label{1.09},
\end{eqnarray}
which is the duality form of
\begin{eqnarray*}
\left\|We^{-itP(D)}(e^{-itP(D)})^{*}\overline{W}\right\|_{\mathfrak{S}^{\alpha^{\prime}}}\leq C
\|W\|_{L_{t\in I_{1}}^{2p^{\prime}}L_{x\in I_{2}}^{2q^{\prime}}}^{2}.
\end{eqnarray*}
Here, $\lambda_{j}(j\in \N^{+})$ are eigenvalues of $\gamma(t)$
 and we use the fact that
\begin{eqnarray}
\left[Trace(|\gamma(t)|^{\alpha})\right]^{\frac{1}{\alpha}}=
\left[\sum\limits_{j=1}^{\infty}|\lambda_{j}|^{\alpha}\right]^{\frac{1}{\alpha}}
=\left\|\gamma(t)\right\|_{\mathfrak{S}^{\alpha}}(1\leq \alpha<\infty)\label{1.010},
\end{eqnarray}
which results from the spectral mapping Theorem of Page 227 in \cite{Y}.
Here $\mathfrak{S}^{\alpha}=\mathfrak{S}^{\alpha}(L^{2}(\R^{d})$ denotes
the Schatten space based on $L^{2}(\R^{d})$, more precisely,
$\mathfrak{S}^{\alpha}$ is the space of all compact operator $A$ on
$L^{2}(\R^{d})$ such that
$Trace|A|^{\alpha}<\infty.$
We define $\|A\|_{\mathfrak{S}^{\infty}}=\|A\|_{L^{2}\longrightarrow L^{2}}$.
$\mathfrak{S}^{2}$ is the Hilbert-Schmidt class and
the $\mathfrak{S}^{2}$ norm is given by $\|A\|_{\mathfrak{S}^{2}}=
\|K_{A}\|_{L^{2}(\SR^{d}\times \SR^{d})},$ where $K_{A}$ is the integral
kernel of $A.$
For the elementary properties of $\mathfrak{S}^{\alpha},$
we refer the readers to \cite{S}.
Obviously, (\ref{1.09}) is considered as the extended version of the
 following Strichartz estimates
\begin{eqnarray}
\left\|e^{-itP(D)}f\right\|_{L_{t\in I_{1}}^{2p}L_{x\in I_{2}}^{2q}}:=\left[\int_{I_1}
\left(\int_{I_{2}}\left|e^{-itP(D)}f\right|^{2q}dx\right)^{p/q}dt\right]^{\frac{1}{2p}}\leq C\label{1.011}
\end{eqnarray}
if $\|e^{-itP(D)}f\|_{L_{t\in I_{1}}^{2p}L_{x\in I_{2}}^{2q}}\leq C
\|f\|_{L^{2}}=C$ with $\|f\|_{L^{2}}=1$.

When $P(D)=-\Delta$  and $1\leq j \leq N$, (\ref{1.04})-(\ref{1.05}) reduces to
\begin{equation}
\begin{cases}
i\partial_{t}u_{1}=(-\bigtriangleup+w(x)\ast \rho)u_{1},&~\text{$ u_{1}|_{t=0}=f_{1}$},\\
&~\vdots\\
i\partial_{t}u_{N}=(-\bigtriangleup+w(x)\ast \rho)u_{N},&~\text{$ u_{N}|_{t=0}=f_{N}$}.
\end{cases}\label{1.012}
\end{equation}
The above system means that
in quantum  mechanics, the $N$ couple  of  nonlinear Hartree
 equations describes
 the dynamics of the system of  $N$ fermions interacting with
 a potential $w(x):\R^{d}\rightarrow \R$.
Here $f_{j}(1\leq j\leq N)$ is orthonormal in $L^{2}(\R^{d})$
 and $\rho$ is a density
 function defined by $\rho=\sum\limits_{j=1}^{N}|u_{j}(x,t)|^{2}.$
When $N\longrightarrow +\infty,$ we derive the operator valued
 equivalent form of
(\ref{1.012}) as follows:
\begin{eqnarray}
i\frac{d \gamma(t)}{dt}=[-\Delta+w*\rho_{\gamma},\gamma]\label{1.013}.
\end{eqnarray}
Here $\gamma(t)=\sum\limits_{j=1}^{+\infty}|u_{j}\rangle \langle u_{j}|$.
For the detail, we refer the readers to \cite{LS2015}.
Some authors
\cite{CHP2017,CHP2018,LS2015} studied the dynamics of
(\ref{1.013}) around the stationary solution. Recently,
by using the idea of
\cite{FLLS,FS2017} and the technique of  \cite{Vega1992},
Nakamura \cite{N} studied  the orthonormal
 Strichartz inequality on the  torus and applied the
   orthonormal
 Strichartz inequality on the  torus to show the local
  well-posedness
 of periodic Hartree equation
describing the infinitely many quantum particles
 interacting with
 the power type potential.

Extending functional inequalities related to  a
single function to
 systems of orthonormal functions can date back to
 Lieb-Thirring inequality \cite{LT} involving finite orthonormal
functions in $L^{2}(\R^{d})$.
 Frank et al. \cite{FLLS} extended the Strichartz inequality
\begin{eqnarray}
\|e^{it\Delta}f\|_{L_{t}^{p}L_{x}^{q}(\SR\times \SR^{d})}\leq C\|f\|_{L^{2}},
d\geq 1, p,q\geq2,\frac{2}{p}+\frac{d}{q}=\frac{d}{2},(d,p,q)\neq(2,\infty,2)\label{1.014}
\end{eqnarray}
to
\begin{eqnarray}
\left\|\sum\limits_{j=1}^{+\infty}\lambda_{j}
|e^{it\Delta}f_{j}|^{2}\right\|_{L_{t}^{p}L_{x}^{q}(\SR\times \SR^{d})}\leq C
\left(\sum\limits_{j=1}^{\infty}|\lambda_{j}|^{\alpha}\right)^{\frac{1}{\alpha}},
1\leq \alpha \leq \frac{2q}{q+1},q\leq 1+\frac{2}{d},\label{1.015}
\end{eqnarray}
where $\lambda_{j}$ are complex numbers and   $f_{j}(1\leq j\leq \infty)$
 is orthonormal
 in $L^{2}(\R^{d})$.  They also presented the
 necessity $\alpha\leq \frac{2q}{q+1}$ for  (\ref{1.015}). (\ref{1.014}) is found in \cite{C,KPV}.
(\ref{1.015}) was exploited in \cite{CHP2017,CHP2018,LS2014,LS2015} to
 investigate the
 well-posedness and scattering of
  the nonlinear evolution of quantum systems with an infinite number
 of particles.
By using the  duality principle related to Schatten norms,
Frank and  Sabin \cite{FS2017} gave the
 full range of Strichartz estimate (\ref{1.015}) except the endpoints.
Bez  et al. \cite{BHLNS} considered the case of orthonormal
 system $(f_{j})_{j}$ in the homogeneous
 Sobolev spaces $\dot{H}^{s}(s\in (0,\frac{d}{2}))$ and established
 the sharp value of $\alpha$
 as a function of $p,q,s$ except possibly endpoint in certain cases.
 They also investigated the frequency
   localised  estiamtes for  orthonormal system $(f_{j})_{j}$ in $L^{2}$.
By combining  the estimates for weighted oscillatory integral estimates
 and the idea of \cite{FS2017},
   Bez et al. \cite{BLN} established some new Strichartz estimates for
 orthonormal
 families of initial data.
By using these spectral cluster bounds  which  follow from
Schatten-type bounds on oscillatory integral
operators,  Frank and  Sabin \cite{FS} generalized the $L^p$ spectral
 cluster bounds of Sogge for the
Laplace-Beltrami operator on compact Riemannian manifolds
to systems of orthonormal functions and discussed  the optimality of these
new bounds. Recently,  Bez et al. \cite{BLN} studied the  maximal estimates for the
 Schr\"odinger equation with orthonormal initial data and
proved that
\begin{eqnarray}
\lim\limits_{t\longrightarrow0}\rho_{\gamma(t)}=\rho_{\gamma_{0}}.\label{1.016}
\end{eqnarray}
Here $\gamma(t)$  is the solution
to the following operator valued equation
\begin{eqnarray}
&&i\partial_{t}\gamma(t)=[-\partial_{x}^{2},\gamma(t)],\label{1.017}\\
&&\gamma(0)=\gamma_{0}\label{1.018}
\end{eqnarray}
and $\rho_{\gamma(t)}$ and $\rho_{\gamma_{0}}$ are the density functions of
 $\gamma(t),\gamma_{0}$, respectively.

Our motivation to write this paper is  sevenfold. Firstly,
in last line of page 205 of
 \cite{Vega1992}, Vega presented the following result
\begin{eqnarray}
\left|\int_{\epsilon<|s|<M}e^{isx}s^{-1+i\beta}ds\right|\leq Ce^{-x\beta}
(\beta \in \R,\epsilon>0,M>0),\label{1.019}
\end{eqnarray}
where $C$ is a constant independent of $x,\epsilon,M,\beta$. However,
Vega does not give the strict proof. In \cite{BHLNS,BLN,N,MS},
(\ref{1.019}) is applied to   establishing some Strichartz estimates for
     orthonormal  functions.
Thus, it is natural to ask whether it is possible to give a suitable
bound of  the following complex integrals
\begin{eqnarray}
\int_{\epsilon<|s|<M}e^{isx}s^{-1+i\beta}ds(\beta \in \R,\epsilon >0,M>0)\label{1.020}.
\end{eqnarray}
Secondly,
by using the fact that
\begin{eqnarray}
 G_{z}(\tau,\xi)|_{z=-1}=\frac{1}{\Gamma(z+1)}(\tau-|\xi|^{2})^{z}|_{z=-1}
=\delta(\tau-|\xi|^{2}),
\forall (\tau,\xi)\in \R\times \R^{d}\label{1.021}
 \end{eqnarray}
 and
 \begin{eqnarray}
 \int_{\SR^{d+1}}e^{i x\cdot\xi+i t\tau}\delta(\tau-|\xi|^{2})d\tau d\xi
=\int_{\SR^{d}}e^{i x\cdot\xi+i t|\xi|^{2}}d\xi\label{1.022}
 \end{eqnarray}
as well  as  the complex interpolation in Schatten spaces which is just
 Theorem 2.9 of \cite{S},  in Theorem 9  of
\cite{FS2017}, Frank and Sabin  established the following Schatten bound
\begin{eqnarray}
\left\|W_{1}T_{S}W_{2}\right\|_{\mathfrak{S}^{q}}\leq C\prod\limits_{j=1}^{2}
\|W_{j}\|_{L_{t}^{p}L_{x}^{q}(\SR\times \SR^{d})}\label{1.023}.
\end{eqnarray}
Here $\frac{2}{p}+\frac{d}{q}=1,q>d+1$, $C$ is  independent of $W_{1},W_{2}$,
$S:=\left\{(\tau,\xi)\in \R\times \R^{d},\tau=|\xi|^{2}\right\}.$
(\ref{1.023}) is just Theorem 9 of \cite{FS2017}.
By combining duality principle which is Lemma 3 \cite{FS2017} with (\ref{1.023}),
 Frank and Sabin  established
\begin{eqnarray}
\hspace{-0.5cm}\left\|\sum\limits_{j=1}^{+\infty}\lambda_{j}
|e^{it\Delta}f_{j}|^{2}\right\|_{L_{t}^{p}L_{x}^{q}(\SR\times \SR^{d})}\leq C
\left(\sum\limits_{j=1}^{\infty}|\lambda_{j}|^{\frac{2q}{q+1}}\right)^{\frac{q+1}{2q}},
1\leq q<1+\frac{2}{d-1},\frac{2}{p}+\frac{d}{q}=d.\label{1.024}
\end{eqnarray}
(\ref{1.024}) is just Theorem 8 of \cite{FS2017}.
Thus, the method of Frank and Sabin \cite{FS2017}  is only  applicable for
$(x,t)\in \R^{d}\times\R$ since  (\ref{1.021})-(\ref{1.022})   play the crucial role
 in establishing (\ref{1.023}).
By using  (\ref{1.019}),  the  duality principle which is Lemma 3 \cite{FS2017}
and   the technique of \cite{Vega1992} instead of  the method of   Frank and Sabin
\cite{FS2017}, for $(x,t)\in \mathbb{T}^{d+1},$    Nakamura \cite{N} established
\begin{eqnarray}
\left\|\sum\limits_{j=1}^{+\infty}\lambda_{j}|e^{it\Delta}
P_{\leq N}f_{j}|^{2}\right\|_{L_{t}^{p}L_{x}^{q}(\mathbb{T}^{d+1})}\leq CN^{\frac{1}{p}}
\left(\sum\limits_{j=1}^{\infty}|\lambda_{j}|^{\alpha}\right)^{\frac{1}{\alpha}},
(\frac{1}{q},\frac{1}{p})\in (A,B],d\geq 1,\label{1.025}
\end{eqnarray}
where $\alpha\leq\frac{2q}{q+1},A=(\frac{d-1}{d+1},\frac{d}{d+1}),B=(1,0)$
and $(A,B]$ represents the open line segment combining $A,B$.
Thus, it is natural to ask whether it is possible to present an
 alternative proof of (\ref{1.024}) with the aid of the technique of
  \cite{N,Vega1992} and other techniques.
Thirdly,  in Lemma 1 of \cite{BLN2021}, Bez et al. proved that
\begin{eqnarray*}
\left|\beta\int_{\epsilon}^{R}e^{\pm ix\rho }\rho^{-1+i\beta}d\rho\right|
\leq C(1+|\beta|)^{2}, \epsilon>0,R>0.
\end{eqnarray*}
Here $C$ is independent of $\beta,\epsilon,R,x$.
Thus, it is natural to ask whether it is possible to prove that
$\left|\beta\int_{a}^{b}e^{isx}s^{-1+i\beta}ds\right|
\leq C(1+|\beta|)^{2}(\beta \in \R,a,b\in \R),$
  where $C$ is independent of $\beta,a, b$.
Fourthly,
 in this paper,   by using   Van der Corput's Lemma which is just
  Lemma 2.11 in this paper,
  suitable variable
substitution  and  the    frequency  decomposition slightly different
 from the technique of Liu \cite{L} as well as TT* method,
we  also prove that
\begin{align}
\left\|e^{it\mp\partial_{x}\sqrt{1-\partial_{x}^{2}}}u_{0}\right\|_{L_{t}^{p}L_{x}^{q}}
\leq C\|u_{0}\|_{L^{2}}\label{1.026}
\end{align}
is valid. It is natural to ask whether it is possible to establish
\begin{eqnarray}
\hspace{-0.5cm}\left\|\sum\limits_{j=1}^{+\infty}\lambda_{j}
|e^{it\mp\partial_{x}\sqrt{1-\partial_{x}^{2}}}f_{j}|^{2}\right\|_{L_{t}^{p}L_{x}^{q}(\SR\times \SR)}\leq C
\left(\sum\limits_{j=1}^{\infty}|\lambda_{j}|^{\frac{2q}{q+1}}\right)^{\frac{q+1}{2q}},
1\leq q \leq \infty,\frac{2}{p}+\frac{1}{q}=1\label{1.027}
\end{eqnarray}
and
\begin{eqnarray}
\left\|\sum\limits_{j=1}^{+\infty}\lambda_{j}|P_{\leq N}e^{it\mp\partial_{x}
\sqrt{1-\partial_{x}^{2}}}f_{j}|^{2}\right\|_{L_{t}^{p}L_{x}^{q}(\mathbf{T}^{2})}
\leq CN^{\frac{1}{p}}
\left(\sum\limits_{j=1}^{\infty}|\lambda_{j}|^{\alpha}\right)^{\frac{1}{\alpha}},
(\frac{1}{q},\frac{1}{p})\in (A,B],\label{1.028}
\end{eqnarray}
where $\alpha\leq\frac{2q}{q+1},A=(0,\frac{1}{2}),B=(1,0)$ and $(A,B]$
represents the open line segment combining $A,B$ and $\T=[0,2\pi)$.
Fifthly, from Lemmas 2.9, 2.10, Theorems 4.1, 4.3, 4.9 established in this paper, we know that
\begin{eqnarray*}
&&\lim\limits_{t\longrightarrow0}\left\|W_{1}\right\|_{L_{t}^{2p^{\prime}}
L_{x}^{2q^{\prime}}([0,t]\times\mathbf{R}^{d})}=0,\\
&&\left\|W_{1}AA^{*}W_{2}\right\|_{\mathfrak{S}^{2q^{\prime}}}\leq
C\left\|W_{1}\right\|_{L_{t}^{2p^{\prime}}L_{x}^{2q^{\prime}}([0,t]\times\mathbf{R}^{d})}
\left\|W_{2}\right\|_{L_{t}^{2p^{\prime}}L_{x}^{2q^{\prime}}([0,t]\times\mathbf{R}^{d})}.
\end{eqnarray*}
Here $A=U(t),U_{\pm}(t),U_{1}(t).$
From Theorem 14 of \cite{FS2017}, we have that
 \begin{eqnarray*}
\left\|\gamma(t)-e^{it\Delta}\gamma_{0}e^{-it\Delta}
\right\|_{C_{t}^{0}([0,\>T],\>\mathfrak{S}^{\frac{2q}{q+1}})}\leq 8T^{1/p^{\prime}}
\|w\|_{L_{x}^{q^{\prime}}}{\rm max}(1,C_{stri}^{2})R^{2}.
\end{eqnarray*}
It is natural to ask whether it is possible to prove that
\begin{eqnarray*}
\lim\limits_{t\longrightarrow0}\left\|W_{1}AA^{*}(t)W_{2}\right\|_{\mathfrak{S}^{2q^{\prime}}}=0,A=U(t),U_{\pm}(t),U_{1}(t)
\end{eqnarray*}
and
\begin{eqnarray*}
\lim\limits_{T\longrightarrow0}\left\|\gamma(t)-e^{it\Delta}\gamma_{0}e^{-it\Delta}
\right\|_{C_{t}^{0}([0,\>T],\>\mathfrak{S}^{\frac{2q}{q+1}})}=0.
\end{eqnarray*}
Sixthly, for $\gamma_{0}\in \mathfrak{S}^{\beta}(\beta<2)$,
Bez et al. \cite{BLN}
\begin{eqnarray}
\lim\limits_{t\longrightarrow0}\rho_{\gamma(t)}(x)=\lim\limits_{t\longrightarrow0}
\sum\limits_{j=1}^{\infty}\lambda_{j}|e^{it\partial_{x}^{2}}f_{j}|^{2}
=\rho_{\gamma_{0}}(x)=\sum\limits_{j=1}^{\infty}\lambda_{j}|f_{j}|^{2},\qquad  a.e. x\in \R,\label{1.029}
\end{eqnarray}
where $\rho_{\gamma(t)}$ is the density function of
 $\gamma(t)$ and
$\gamma(t)=\sum\limits_{j=1}^{+\infty}
\lambda_{j}|e^{it\partial_{x}^{2}}f_{j}\rangle\langle e^{it\partial_{x}^{2}}f_{j}|$
 is the solution to
the operator-valued equation
to
$\left\{
        \begin{array}{ll}
         i\frac{d\gamma(t)}{dt}=[-\partial_{x}^{2},\gamma(t)] \\
          \gamma(0)=\sum\limits_{j=1}^{\infty}\lambda_{j}|f_{j}\rangle\langle f_{j}|.
        \end{array}
      \right.$
Here $(f_{j})_{j\in \mathbf{N}^{+}}$ are orthonormal functions in $\dot{H}^{1/4}(\R)$.
Obviously,
(\ref{1.029}) is equivalent to the following statement: $\forall \alpha>0,$
\begin{eqnarray}
&&mes\left(\left\{x\in \R:\lim\limits_{t\longrightarrow0}
\left|\rho_{\gamma(t)}(x)-\rho_{\gamma_{0}}(x)\right|>\alpha\right\}\right)\nonumber\\
&&=mes\left(\left\{x\in \R:\lim\limits_{t\longrightarrow0}
\left|\sum\limits_{j=1}^{\infty}\lambda_{j}|e^{it\partial_{x}^{2}}f_{j}|^{2}-
\sum\limits_{j=1}^{\infty}\lambda_{j}|f_{j}|^{2}\right|>\alpha\right\}\right)=0.\label{1.030}
\end{eqnarray}
In fact, when $\lambda_{1}=1,\lambda_{j}=0(2\leq j\leq +\infty)$, (\ref{1.029})
 reduces to the classical pointwise convergence
problem of one dimensional linear Schr\"odinger equation
\begin{eqnarray}
\lim\limits_{t\longrightarrow0}
e^{it\partial_{x}^{2}}f_{1}=
f_{1},\qquad  a.e. x\in \R,\label{1.031}
\end{eqnarray}
which was initially studied by Carleson \cite{Carleson} and
 then studied by some authors \cite{Bourgain1995,Bourgain2016,CLV,Cowling,Du,
  DG,DGL, DGZ,DK,DZ,Lee,LR2015,LR2017,MV2008,Shao,Vega1988,Zhang} in higher dimension and
   also studied by Compaan et al. \cite{CLS} with random data on  $\R^{d},$
  $\T^{d}.$ Recently, Wang et al. \cite{WYY} and Yan et al. \cite{YZY}
   studied (\ref{1.031}) in higher dimension on
$\R^{d}$ and
  $\T^{d}$ and $\Theta=\left\{x\in \R^{3}
  |\sum\limits_{j=1}^{3}x_{j}^{2}<1\right\}$ with rough data and random data. Moreover,
the result of \cite{WYY} and  \cite{YZY} improves the result of \cite{CLS}
 on  $\R^{d},$
  $\T^{d}$ and the result of \cite{WZ} on  $\Theta=
  \left\{x\in \R^{3}|\sum\limits_{j=1}^{3}x_{j}^{2}<1\right\}$.
Finally, when  $(f_{j})_{j\in \mathbf{N}^{+}}$ are orthonormal functions
 in $L^{2}(\R^{d})$ or $L^{2}(\T^{d})$ or $L^{2}(\Theta)$,  for the compact operator
  $\gamma_{0}=\sum\limits_{j=1}^{+\infty}
\lambda_{j}|f_{j}\rangle\langle f_{j}|\in \mathfrak{S}^{2},$
it is natural to ask whether it is possible to perform suitable
 randomization
 on the compact operator $\gamma_{0}$
 which is defined as $\gamma_{0}^{\omega,\>\widetilde{\omega}}$
  (to be determined) such that
  for $\alpha>0$ (to be determined)
\begin{eqnarray}
\lim\limits_{t\longrightarrow0}(\mathbb{P}\times\widetilde{\mathbb P})
\left(\left\{(\omega, \widetilde{\omega})\in (\Omega\times \widetilde{\Omega})|
|\rho_{e^{it\Delta}\gamma_{0}^{\omega,\>\widetilde{\omega}}e^{-it\Delta}}-
\rho_{\gamma_{0}^{\omega,\>\widetilde{\omega}}}|>\alpha\right\}\right)=0\label{1.032}
\end{eqnarray}
with $x\in \R^{d}$ or $\T^{d}$ and
\begin{eqnarray*}
&&\hspace{-1cm}\lim\limits_{t\longrightarrow0}(\mathbb{P}\times\widetilde{\mathbb P})
\left(\left\{(\omega_{1}, \widetilde{\omega})\in (\Omega_{1}\times \widetilde{\Omega})|
|\rho_{\gamma_{0}^{\omega_{1},\>\widetilde{\omega}}}-
\rho_{e^{it\Delta}\gamma_{0}^{\omega_{1},\>\widetilde{\omega}}e^{-it\Delta}}|>\alpha\right\}\right)=0
\end{eqnarray*}
with  $x\in \Theta=\left\{x\in \R^{3}:|x|<1\right\}$.

The main content of our paper is eightfold.
Firstly, by using the  Cauchy's principal value,
the Fourier transformation of generalized functions,  Gamma function
 of Weierstrass definition,
Euler product formula as well as the Dirichlet test,  we  present
 the suitable bound of
$\int_{a\leq|s|\leq b}e^{isx}s^{-1+i\beta}ds(\beta \in \R,a\geq0,b>0),$
$\int_{a\leq s\leq b}e^{isx}s^{-1+i\beta}ds(\beta \in \R,\beta\neq0,a,b\in \R)$
and   $\int_{0\leq|s|< +\infty}e^{isx}s^{-1+i\beta}ds(\beta \in \R)$,
which extends the result of Page 204  of
 Vega  \cite{Vega1992}, for the details, we refer the readers to
  Lemmas 3.1-3.17, 3.20-3.21 and Remark 4 in this paper. In fact,
 in page 204 of \cite{Vega1992}, Vega  studied the complex integral
\begin{eqnarray*}
\lim\limits_{\epsilon\longrightarrow0,\>M\longrightarrow+\infty}
\int_{\epsilon\leq |\rho|\leq M}e^{-it\rho}
\rho^{-1+ib}d\rho=P.V.\int_{\SR}e^{-it\rho}\rho^{-1+ib}d\rho
\end{eqnarray*}
 and claim the similar conclusion,
 however, he did not give the strict proof.
The conclusion of  \cite{Vega1992} has been applied to \cite{BHLNS,BLN,MS,N}.
In  Lemma 3.8, we prove that
\begin{eqnarray*}
&&\left|P.V.\int_{\SR}e^{-it\rho}\rho^{-1+ib}d\rho\right|=
\left|\int_{\SR}e^{-it\rho}\rho^{-1+ib}d\rho\right|\nonumber\\&&\leq
 C\left[1+\left(\frac{1-e^{-2\pi b}}{b}\right)^{1/2}\right]\nonumber\\&&\leq C(1+e^{-\pi b})
 \left(\sqrt{|b|}+\frac{1}{\sqrt{|b|}}\right)^{2}
\end{eqnarray*}
and $|P.V.\int_{\SR}e^{-it\rho}\rho^{-1+ib}d\rho|\leq \pi(b=0)$.
 In particular, in Remark 4, we prove that
\begin{eqnarray*}
&&\left|\int_{0}^{\infty}e^{-it\rho}\rho^{-1+bi}d\rho\right|\leq
C\left(\frac{1}{|b|}+1\right)(1+e^{\pi b})(b\neq0),\\&&
\left|\int_{-\infty}^{0}e^{-it\rho}\rho^{-1+bi}d\rho\right|\leq
C\left(\frac{1}{|b|}+1\right)(1+e^{-\pi b})(b\neq0).
\end{eqnarray*}
Secondly, by using Lemma 1 of \cite{BLN2021}  and  Lemma 3.18 in
    this paper as well as variable substitution,
  we prove that  $\left|\beta\int_{a}^{b}e^{isx}s^{-1+i\beta}ds\right|
\leq C(1+e^{-\pi \beta})(1+|\beta|)^{2}(\beta \in \R,a,b\in \R),$
  where $C$ is independent of $\beta,a, b$, in particular, when $a>0,b>0$, we have that
  $\left|\beta\int_{a}^{b}e^{isx}s^{-1+i\beta}ds\right|
\leq C(1+|\beta|)^{2}$.
Thus, our result extends Lemma 1 of Bez et al. of \cite{BLN2021}, for the details,
 we refer the readers to Lemma 3.19 in this paper.
 Thirdly, by using duality principle in the mixed Lebesgue spaces
  which  is just Lemma 2.10 and   Lemmas 3.8-3.10, 3.12-3.17, 3.20-3.21
 and Remark 4,
  we establish some Strichartz estimates for orthonormal functions
 and Schatten bound with space-time norms
related to  elliptic operator and non-elliptic operator  on $\R^{d},\T^{d}$, for the details,
 we refer the readers to Theorems 4.1-4.7 established in this paper
 which   extend the result of Theorems 8, 9 of
 \cite{FS2017} and Proposition 3.3 and Theorem 1.5 of \cite{N},
 respectively, notice that the method in proving Theorems 4.1-4.7
 is different from Theorems 8, 9 of
 \cite{FS2017}.
Fourthly, we  establish the Strichartz estimates  for
  orthonormal  functions related to Boussinesq operator
 on $\R$ with the
  aid of Lemma 3.20 in this paper, for the details,
 we refer the readers to Theorems 4.8-4.11.
Fifthly, by using the continuity of the mixed Lebesgue spaces
 which is just Lemma 2.9 in this paper
and duality principle in the mixed Lebesgue spaces which is just
 Lemma 2.10 in this paper as well as Theorems 4.1, 4.3, 4.8, we prove that
\begin{eqnarray*}
\lim\limits_{t\longrightarrow0}
\left\|W_{1}AA^{*}(t)W_{2}\right\|_{\mathfrak{S}^{2q^{\prime}}}=0,
\end{eqnarray*}
where $A=U(t),$ or $U_{\pm}(t)$,  $U_{1}(t)$
 are defined in Section 4, for the details, we refer the
 readers to Theorems 5.1-5.3.
By using Theorem 14 of \cite{FS2017},  we prove that
\begin{eqnarray*}
\lim\limits_{T\longrightarrow0}\left\|\gamma(t)-e^{it\Delta}\gamma_{0}e^{-it\Delta}
\right\|_{C_{t}^{0}([0,T],\mathfrak{S}^{\frac{2q}{q+1}})}=0,
\end{eqnarray*}
for the details, we refer the readers to Theorem 5.4.
Sixthly, combining   the idea of
 \cite{HY} with the idea of \cite{YZY,YZDY},
  we establish the probabilistic convergence
of density functions with full
 randomization  related to  some compact operators on $\R^{d}$.
 More precisely,
for any $(\omega,\widetilde{\omega})\in\Omega \times \widetilde{\Omega}$,
we define the full randomization of some compact operators
 $\gamma_{0}=\sum\limits_{j=1}^{+\infty}
 \lambda_{j}|f_{j}\rangle\langle f_{j}|\in \mathfrak{S}^{2}$
by
$\gamma_{0}^{\omega,\>\widetilde{\omega}}:=
\sum\limits_{j=1}^{\infty}\lambda_{j}g^{(2)}_{j}
(\widetilde{\omega})|f_{j}^{\omega}
\rangle\langle f_{j}^{\omega}|.$
Here, for the definition of $f_{j}^{\omega}$, we refer the readers to (\ref{6.05});
for the definition of $\gamma_{0}^{\omega,\>\widetilde{\omega}}$,
we refer the readers to (\ref{6.07}).
 Obviously,
\begin{eqnarray*}
\rho_{\gamma_{0}^{\omega,\>\widetilde{\omega}}}=
\sum_{j=1}^{\infty}\lambda_{j}g^{(2)}_{j}(\widetilde{\omega})|f_{j}^{\omega}|^{2},
\rho_{e^{it\Delta}\gamma_{0}^{\omega,\>\widetilde{\omega}}e^{-it\Delta}}
=\sum_{j=1}^{\infty}\lambda_{j}g^{(2)}_{j}(\widetilde{\omega})|e^{it\Delta}f_{j}^{\omega}|^{2}.
\end{eqnarray*}
Here, $\rho_{\gamma_{0}^{\omega,\>\widetilde{\omega}}}, \rho_{e^{it\Delta}
   \gamma_{0}^{\omega,\>\widetilde{\omega}}e^{-it\Delta}}$ denote the
 density of $\gamma_{0}^{\omega,\>\widetilde{\omega}}$ and $e^{it\Delta}
   \gamma_{0}^{\omega,\>\widetilde{\omega}}e^{-it\Delta}$, respectively.
\noindent For
   $F(t,\omega,\>\widetilde{\omega})=\rho_{e^{it\Delta}
   \gamma_{0}^{\omega,\>\widetilde{\omega}}e^{-it\Delta}}-
\rho_{\gamma_{0}^{\omega,\>\widetilde{\omega}}}$,   we prove
\begin{eqnarray}
&&\hspace{-1cm}\lim\limits_{t\longrightarrow0}(\mathbb{P}
\times\widetilde{\mathbb P})
\left(\left\{(\omega, \widetilde{\omega})\in
(\Omega\times \widetilde{\Omega})|
|F(t,\omega,\>\widetilde{\omega})|>Ce
\left[\left\|\gamma_{0}\right\|_{\mathfrak{S}^{2}}+1\right]
\epsilon^{1/2}\left(\epsilon ln
\frac{1}{\epsilon}\right)^{d+\frac{1}{2}}\right\}\right)
\nonumber\\&&=0\label{1.033}
\end{eqnarray}
with  $x\in \R^{d}$. For the detail of proof of (\ref{1.033}),
 we refer the readers to Theorem 6.1.
Here $(f_{j})_{j\in \mathbf{N}^{+}}$
 are orthonormal functions in $L^{2}(\R^{d})$. Thus,
we improve the result of Corollary 1.2 of
 \cite{BLN2021} in the probabilistic sense.
Seventhly, combining the idea of \cite{HY,CLS} with the idea of \cite{WYY},
 we establish the probabilistic convergence
of density functions with full
 randomization  related to some compact operators on $\T^{d}$.
More precisely,
for any $(\omega,\widetilde{\omega})\in\Omega \times \widetilde{\Omega}$,
we define the full randomization of some compact operators
 $\gamma_{0}=\sum\limits_{j=1}^{+\infty}\lambda_{j}|f_{j}\rangle\langle f_{j}|\in \mathfrak{S}^{2}$
by
$\gamma_{0}^{\omega,\>\widetilde{\omega}}:=
\sum\limits_{j=1}^{\infty}\lambda_{j}g^{(2)}_{j}(\widetilde{\omega})|f_{j}^{\omega}
\rangle\langle f_{j}^{\omega}|.$
Here, for the definition of $f_{j}^{\omega}$, we refer the readers to (\ref{7.01});
for the definition of $\gamma_{0}^{\omega,\>\widetilde{\omega}}$,
we refer the readers to (\ref{7.03}).
 Obviously,
\begin{eqnarray*}
\rho_{\gamma_{0}^{\omega,\>\widetilde{\omega}}}=
\sum_{j=1}^{\infty}\lambda_{j}g^{(2)}_{j}(\widetilde{\omega})|f_{j}^{\omega}|^{2},
\rho_{e^{it\Delta}\gamma_{0}^{\omega,\>\widetilde{\omega}}e^{-it\Delta}}
=\sum_{j=1}^{\infty}\lambda_{j}g^{(2)}_{j}(\widetilde{\omega})|e^{it\Delta}f_{j}^{\omega}|^{2}.
\end{eqnarray*}
Here, $\rho_{\gamma_{0}^{\omega,\>\widetilde{\omega}}}, \rho_{e^{it\Delta}
   \gamma_{0}^{\omega,\>\widetilde{\omega}}e^{-it\Delta}}$ denote the
 density function of $\gamma_{0}^{\omega,\>\widetilde{\omega}}$ and $e^{it\Delta}
   \gamma_{0}^{\omega,\>\widetilde{\omega}}e^{-it\Delta}$, respectively.
\noindent For
   $F(t,\omega,\>\widetilde{\omega})=\rho_{e^{it\Delta}
   \gamma_{0}^{\omega,\>\widetilde{\omega}}e^{-it\Delta}}-
\rho_{\gamma_{0}^{\omega,\>\widetilde{\omega}}}$,   we prove
\begin{eqnarray}
&&\hspace{-1cm}\lim\limits_{t\longrightarrow0}(\mathbb{P}
\times\widetilde{\mathbb P})
\left(\left\{(\omega, \widetilde{\omega})\in
(\Omega\times \widetilde{\Omega})|
|F(t,\omega,\>\widetilde{\omega})|>Ce
\left[\left\|\gamma_{0}\right\|_{\mathfrak{S}^{2}}+1\right]
\epsilon^{1/2}\left(\epsilon ln
\frac{1}{\epsilon}\right)^{d+\frac{1}{2}}\right\}\right)
\nonumber\\&&=0\label{1.034}
\end{eqnarray}
with  $x\in \T^{d}$. For the detail of proof of (\ref{1.034}),
we refer the readers to Theorem 7.1.
Finally, combining the idea of \cite{HY} with the idea of \cite{WYY},
we establish the probabilistic convergence
of density functions with full
 randomization  related to some compact operators on $\Theta=\left\{x\in \R^{3}:|x|<1\right\}$.
For any $(\omega_{1},\widetilde{\omega})\in\Omega_{1} \times \widetilde{\Omega}$,
we define the full randomization of
 $\gamma_{0}=\sum\limits_{j=1}^{+\infty}\lambda_{j}|f_{j}\rangle\langle f_{j}|$
by
$\gamma_{0}^{\omega_{1},\>\widetilde{\omega}}:=
\sum\limits_{j=1}^{\infty}\lambda_{j}g^{(2)}_{j}(\widetilde{\omega})|f_{j}^{\omega_{1}}
\rangle\langle f_{j}^{\omega_{1}}|.$ Here, for the definition of $f_{j}^{\omega_{1}}$, we refer the readers to (\ref{8.01});
for the definition of $\gamma_{0}^{\omega_{1},\>\widetilde{\omega}}$,
we refer the readers to (\ref{8.03}).
Obviously,
\begin{eqnarray*}
\rho_{\gamma_{0}^{\omega_{1},\>\widetilde{\omega}}}=
\sum_{j=1}^{\infty}\lambda_{j}g^{(2)}_{j}(\widetilde{\omega})|f_{j}^{\omega_{1}}|^{2},
\rho_{e^{it\Delta}\gamma_{0}^{\omega_{1},\>\widetilde{\omega}}e^{-it\Delta}}
=\sum_{j=1}^{\infty}\lambda_{j}g^{(2)}_{j}(\widetilde{\omega})|e^{it\Delta}f_{j}^{\omega_{1}}|^{2}.
\end{eqnarray*}
Here, $\rho_{\gamma_{0}^{\omega_{1},\>\widetilde{\omega}}}, \rho_{e^{it\Delta}
   \gamma_{0}^{\omega_{1},\>\widetilde{\omega}}e^{-it\Delta}}$ denote the
 density function of $\gamma_{0}^{\omega_{1},\>\widetilde{\omega}}, e^{it\Delta}
   \gamma_{0}^{\omega_{1},\>\widetilde{\omega}}e^{-it\Delta}$, respectively.
For $F(t,\omega_{1},\widetilde{\omega})=\rho_{\gamma_{0}^{\omega_{1},\>\widetilde{\omega}}}-
\rho_{e^{it\Delta}\gamma_{0}^{\omega_{1},\>\widetilde{\omega}}e^{-it\Delta}},$ we have
\begin{eqnarray}
&&\hspace{-1cm}\lim\limits_{t\longrightarrow0}(\mathbb{P}\times\widetilde{\mathbb P})
\left(\left\{(\omega_{1}, \widetilde{\omega})\in (\Omega_{1}\times \widetilde{\Omega})|
|F(t,\omega_{1},\widetilde{\omega})|>Ce
\left[\left\|\gamma_{0}\right\|_{\mathfrak{S}^{2}}+1\right]\epsilon^{1/2}\left(\epsilon ln
\frac{1}{\epsilon}\right)^{3/2}\right\}\right)\nonumber\\&&=0\label{1.035}
\end{eqnarray}
with  $x\in \Theta$.
Here $(f_{j})_{j\in \mathbf{N}^{+}}$ are orthonormal functions in $L^{2}(\Theta)$.
For the detail of proof of (\ref{1.035}), we refer the readers to Theorem 8.1.

The remainder of this paper is organized as follows:

$\bullet$ In Section 2, we present some preliminaries which contain some
properties of the generalized functions,
some integrals,   some inequalities, the
 continuity in the mixed Lebesgue spaces,
the duality principle in the mixed Lebesgue spaces and
Van der Corput's Lemma.

$\bullet$ In Section 3, we present the bound of some complex integrals, some
 Strichartz estimates related to Boussinesq operator. More precisely,
 by  using   Cauchy's principal value,
the Fourier transformation of generalized functions,  Gamma function
 of Weierstrass definition,
Euler product formula as well as the Dirichlet test, for $b\in \R$,
  $a\geq0,c\geq0,d_{0}>0$,
we prove that the following four types of complex integrals
\begin{eqnarray*}
&&P.V.\int_{\SR}e^{-it\rho}\rho^{-1+ib}d\rho,P.V.\int_{|\rho|\geq a}
e^{-it\rho}\rho^{-1+ib}d\rho,\nonumber\\
&&
P.V.\int_{|\rho|\leq a}e^{-it\rho}\rho^{-1+ib}d\rho, P.V.
\int_{c\leq |\rho|\leq d_{0}}e^{-it\rho}\rho^{-1+ib}d\rho
\end{eqnarray*}
 converge and present a suitable bound. As a byproduct, for $b\neq 0$,
  $a_{j}\in \R(j=1,2,3,4)$ and $\epsilon>0,M>0$, we prove that
 \begin{eqnarray*}
&&\int_{\SR}e^{-it\rho}\rho^{-1+ib}d\rho,\int_{a_{1}}^{+\infty}e^{-it\rho}\rho^{-1+ib}d\rho,
\int_{-\infty}^{a_{2}}e^{-it\rho}\rho^{-1+ib}d\rho,\nonumber\\&&
\int_{a_{3}\leq \rho\leq a_{4}}e^{-it\rho}\rho^{-1+ib}d\rho,
\int_{\epsilon\leq |\rho|\leq M}e^{-it\rho}\rho^{-1+ib}d\rho
\end{eqnarray*}
converge and present a suitable bound. These results generalize the result
 of the last line of page 205 of
 \cite{Vega1992}.
For $b,c,d_{0}\in \R,$    we also prove that
\begin{eqnarray*}
\left|b\int_{c}^{d_{0}}e^{-it\rho}\rho^{-1+ib}d\rho\right|\leq C(1+e^{-\pi b})(1+|b|)^{2}.
\end{eqnarray*}
This extends Lemma 1 of \cite{BLN2021}.
By using  Van der Corput's  Lemma which is just Lemma 2.11,  we also prove that
\begin{eqnarray*}
&&\left|\int_{\mathbb{R}}e^{ix\xi\pm it\xi\sqrt{1+\xi^{2}}}d\xi\right|\leq Ct^{-\frac{1}{2}},0<t\leq1,\\
&&\left|\int_{\mathbb{R}}e^{ix\xi\pm it\xi\sqrt{1+\xi^{2}}}d\xi\right|\leq Ct^{-\frac{1}{3}},t\geq1.
\end{eqnarray*}
which are Lemmas 3.25, 3.28, respectively.

$\bullet$ In Section 4, inspired by \cite{Vega1992,N, S}, by using
Lemmas 3.1-3.21, 3.25, 3.28  established in this paper,
we present the proof of some Strichartz estimates for
 orthonormal functions
  and Schatten bound with space-time norms related to elliptic
operator, non-elliptic
operator on $\R^{d},\T^{d}$ and Boussinesq operator on $\R$.

$\bullet$  In Section 5, by using the continuity of the mixed Lebesgue
spaces which is just Lemma 2.9 in this paper
and duality principle in the mixed Lebesgue spaces which is just Lemma
2.10 in this paper as well as Theorems 4.1, 4.3, 4.8, we  prove
\begin{eqnarray*}
\lim\limits_{t\longrightarrow0}\left\|W_{1}AA^{*}(t)W_{2}\right\|_{\mathfrak{S}^{2q^{\prime}}}=0,
\end{eqnarray*}
where $A=U(t),$ or $U_{\pm}(t)$,  $U_{1}(t)$ are defined in Section 4.
 By using Theorem 14 of \cite{FS2017},
we  establish   the convergence  result related to nonlinear part
 of the solution to
some operator equations in Schatten norms.

$\bullet$ In Section 6, we present the probabilistic convergence
 of density functions
   of some compact operators with full randomization on $\R^{d}$ with
    the aid of stochastic
continuity at zero related to Schatten norms.

$\bullet$ In Section 7, we present the probabilistic convergence of density functions
   of some compact operators with full randomization on $\T^{d}$ with the aid of stochastic
continuity at zero related to Schatten norms.

$\bullet$ In Section 8, we present the probabilistic convergence of density functions
   of some compact operators with full randomization on
$\Theta=\left\{x\in \R^{3}|\sum\limits_{j=1}^{3}x_{j}^{2}<1\right\}$
 with the aid of stochastic
continuity at zero related to Schatten norms.

\bigskip
\bigskip

\section{Preliminaries}
\medskip
\setcounter{equation}{0}
\setcounter{Theorem}{0} \setcounter{Lemma}{0}\setcounter{Definition}{0}\setcounter{Proposition}{2}

\setcounter{section}{2}
As explained in the introduction, in this section, we present
 some preliminaries which contain some
properties of the generalized functions,
some integrals,   some inequalities, the
 continuity in the mixed Lebesgue spaces and
the duality principle in the mixed Lebesgue spaces and
Van der Corput's Lemma.
\begin{Lemma}\label{2.1}
Let $\lambda\in \C.$ Then, we have
\begin{eqnarray*}
(x+i0)^{\lambda}=\lim\limits_{y\longrightarrow 0^{+}}(x^{2}+y^{2})^{\frac{\lambda}{2}}
exp\left[i\lambda arg(x+iy)\right]
=\left\{
        \begin{array}{ll}
         e^{i\lambda \pi}|x|^{\lambda}, & \hbox{$x<0$,} \\
          x^{\lambda}, & \hbox{$x>0$,}
        \end{array}
      \right.
\end{eqnarray*}
and
\begin{eqnarray*}
(x-i0)^{\lambda}=\lim\limits_{y\longrightarrow 0^{-}}(x^{2}+y^{2})^{\frac{\lambda}{2}}
exp\left[i\lambda arg(x+iy)\right]
=\left\{
        \begin{array}{ll}
         e^{-i\lambda \pi}|x|^{\lambda}, & \hbox{$x<0$,} \\
          x^{\lambda}, & \hbox{$x>0$.}
        \end{array}
      \right.
\end{eqnarray*}

\end{Lemma}
\noindent{\bf Proof.}For the proof of Lemma 2.1, we  refer the readers
 to page 59 of \cite{GS}.

This completes the proof of Lemma 2.1.

\begin{Lemma}\label{2.2}
Let $b\in \R.$ Then, we have
\begin{eqnarray*}
(-t+i0)^{-ib}-(-1)^{ib}(t+i0)^{-ib}=\left\{
        \begin{array}{ll}
         t^{-ib}(e^{b\pi}-e^{-b\pi}), & \hbox{$t>0$,} \\
          0, & \hbox{$t<0$.}
        \end{array}
      \right.
\end{eqnarray*}

\end{Lemma}
\noindent{\bf Proof.} A direct application of Lemma 2.1 yields
 that Lemma 2.2 is valid.

This completes the proof of Lemma 2.2.

\begin{Lemma}\label{2.3}
Let $0<a_{1}<b_{1}<+\infty.$ Then, we have
\begin{align}
\left|\int_{a_{1}}^{b_{1}}\frac{sinx}{x}dx\right|\leq4 .\label{2.01}
\end{align}
\end{Lemma}
\noindent{\bf Proof.} We consider $b_{1}\leq1, a_{1}\leq 1\leq b_{1}, a_{1}\geq1$, respectively.

\noindent When $b_{1}\leq1$, we have
\begin{eqnarray}
\left|\int_{a_{1}}^{b_{1}}\frac{sinx}{x}dx\right|\leq \int_{a_{1}}^{b_{1}}\left|\frac{sinx}{x}\right|dx
\leq \int_{a_{1}}^{b_{1}}dx=b_{1}-a_{1}\leq 1.\label{2.02}
\end{eqnarray}
\noindent When $a_{1}\leq1\leq b_{1}$, we have
\begin{eqnarray}
\int_{a_{1}}^{b_{1}}\frac{sinx}{x}dx=\int_{a_{1}}^{1}\frac{sinx}{x}dx+
\int_{1}^{b_{1}}\frac{sinx}{x}dx.\label{2.03}
\end{eqnarray}
Obviously, we have
\begin{eqnarray}
\left|\int_{a_{1}}^{1}\frac{sinx}{x}dx\right|\leq \int_{a_{1}}^{1}
\left|\frac{sinx}{x}\right|dx\leq \int_{a_{1}}^{1}dx=1-a_{1}\leq1.\label{2.04}
\end{eqnarray}
By using the integration by parts, we have
\begin{eqnarray}
&&\left|\int_{1}^{b_{1}}\frac{sinx}{x}dx\right|=\left|-\int_{1}^{b_{1}}\frac{1}{x}dcosx\right|=
\left|\frac{cos1}{1}-\frac{cosb_{1}}{b_{1}}-\int_{1}^{b_{1}}\frac{cosx}{x^{2}}dx\right|\nonumber\\
&&\leq 1+1+\int_{1}^{b_{1}}\frac{1}{x^{2}}dx=3-\frac{1}{b_{1}}\leq 3.\label{2.05}
\end{eqnarray}
By using (\ref{2.03})-(\ref{2.05}), it follows that
\begin{align}
\left|\int_{a_{1}}^{b_{1}}\frac{sinx}{x}dx\right|\leq4 .\label{2.06}
\end{align}
 When $a_{1}\geq1$, by using the integration by parts, we have
 \begin{eqnarray}
&&\left|\int_{a_{1}}^{b_{1}}\frac{sinx}{x}dx\right|=\left|-\int_{a_{1}}^{b_{1}}\frac{1}{x}dcosx\right|
=\left|\frac{cosa_{1}}{a_{1}}-\frac{cosb_{1}}{b_{1}}-\int_{a_{1}}^{b_{1}}\frac{cosx}{x^{2}}dx\right|\nonumber\\
&&\leq 1+1+\int_{a_{1}}^{b_{1}}\frac{1}{x^{2}}dx\leq 3.\label{2.07}
\end{eqnarray}

This completes the proof of Lemma 2.3.

\noindent {\bf Remark 1.} The conclusion of Lemma 2.3 can be found in page 263
 of \cite{G}, however, the strict proof is not presented.

\begin{Lemma}\label{2.4}
Let $a>0,t\neq0.$ Then, we have
\begin{align}
\left|\int_{a}^{+\infty}\frac{sintx}{x}dx\right|\leq4,
\left|\int_{-\infty}^{-a}\frac{sintx}{x}dx\right|\leq4.\label{2.08}
\end{align}
\end{Lemma}
\noindent{\bf Proof.} Since
\begin{eqnarray}
&&\int_{0}^{+\infty}\frac{sinx}{x}dx=\int_{0}^{+\infty}sinx \left(\int_{0}^{+\infty}e^{-xy}dy\right)dx\nonumber\\
&&=\int_{0}^{+\infty}\left(\int_{0}^{+\infty}e^{-xy}sinx dx\right)dy=\int_{0}^{+\infty}\frac{1}{y^{2}+1}dy=
\frac{\pi}{2}, \label{2.09}
\end{eqnarray}
for $\forall \epsilon>0,$    there exists $M>0$ such that
\begin{eqnarray}
\left|\int_{M}^{+\infty}\frac{sinx}{x}dx\right|\leq \epsilon\label{2.010}.
\end{eqnarray}
\noindent When $t>0$, from Lemma 2.3, we have
\begin{eqnarray}
&&\left|\int_{a}^{\infty}\frac{sintx}{x}dx\right|\ \ \underline{\underline{xt=\rho}}\ \
\left|\int_{at}^{\infty}\frac{sin\rho}{\rho}d\rho\right|\nonumber\\&&=
\left|\int_{at}^{M}\frac{sin\rho}{\rho}d\rho+\int_{M}^{+\infty}\frac{sin\rho}{\rho}d\rho\right|
\leq \left|\int_{at}^{M}\frac{sin\rho}{\rho}d\rho\right|+\left|\int_{M}^{+\infty}\frac{sin\rho}{\rho}d\rho\right|\nonumber\\
&&\leq 4+\epsilon\label{2.011}.
\end{eqnarray}
\noindent When $t<0$, from Lemma 2.3, we have
\begin{eqnarray}
&&\left|\int_{a}^{\infty}\frac{sintx}{x}dx\right|=\left|\int_{a}^{\infty}\frac{sin(-tx)}{-x}dx\right|
\ \ \underline{\underline{-xt=\rho}}\ \ \left|\int_{-at}^{\infty}\frac{sin\rho}{\rho}d\rho\right|\nonumber\\
&&=
\left|\int_{-at}^{M}\frac{sin\rho}{\rho}d\rho+\int_{M}^{+\infty}\frac{sin\rho}{\rho}d\rho\right|\leq
\left|\int_{-at}^{M}\frac{sin\rho}{\rho}d\rho\right|+\left|\int_{M}^{+\infty}\frac{sin\rho}{\rho}d\rho\right|\nonumber\\
&&\leq 4+\epsilon\label{2.012}.
\end{eqnarray}
Since $\epsilon>0$ is arbitrary, from (\ref{2.011}) and (\ref{2.012}),
 we have that
\begin{eqnarray}
\left|\int_{a}^{\infty}\frac{sintx}{x}dx\right|\leq 4\label{2.013}.
\end{eqnarray}
By using a proof similar to (\ref{2.013}), we have that
\begin{eqnarray}
\left|\int_{-\infty}^{-a}\frac{sintx}{x}dx\right|\leq 4\label{2.014}.
\end{eqnarray}

This completes the proof of Lemma 2.4.

\begin{Lemma}\label{2.5}
Let $x\neq0.$ Then, we have
\begin{align}
\left|\frac{2xe^{-x}}{e^{x}-e^{-x}}\right|\leq2|x|+1.\label{2.015}
\end{align}
\end{Lemma}
\noindent{\bf Proof.} When $x>0,$ since  $e^{2x}-1\geq 2x,$ we have
\begin{eqnarray}
\frac{2xe^{-x}}{e^{x}-e^{-x}}=\frac{2x}{e^{2x}-1}\leq 1.\label{2.016}
\end{eqnarray}
When $x<0,$ we have
\begin{eqnarray}
&&\left|\frac{2xe^{-x}}{e^{x}-e^{-x}}\right|
=\frac{-2xe^{-x}}{e^{-x}-e^{x}}\ \ \underline{\underline{t=-x}}
\ \ \frac{2te^{t}}{e^{t}-e^{-t}}=\frac{2te^{2t}}{e^{2t}-1}\nonumber\\
&&=2t\left(1+\frac{1}{e^{2t}-1}\right)\leq 2t+\frac{2t}{e^{2t}-1}\leq 2t+1=2|x|+1.\label{2.017}
\end{eqnarray}
Here, we use the inequality $e^{2t}-1\geq 2t.$

This completes the proof of Lemma 2.5.

\noindent {\bf Remark 2.} Obviously,
\begin{eqnarray}
\lim\limits_{x\longrightarrow0} \frac{2xe^{-x}}{e^{x}-e^{-x}}=  \lim\limits_{x\longrightarrow0}
\frac{2x}{e^{2x}-1} = \lim\limits_{x\longrightarrow0} \frac{1}{e^{2x}} =1.\label{2.018}
\end{eqnarray}

\begin{Lemma}\label{2.6}
Let $b\neq0.$ Then, we have
\begin{align}
\left|\frac{2\pi be^{-\pi b}}{e^{\pi b}-e^{-\pi b}}\right|\leq2\pi|b|+1.\label{2.019}
\end{align}
\end{Lemma}
\noindent{\bf Proof.} Take $x=\pi b$ in (\ref{2.015}), we derive that (\ref{2.019}) is valid.

This completes the proof of Lemma 2.6.

\begin{Lemma}\label{lem2.7}(The density Theorem in the mixed Lebesgue spaces)
Suppose that  $0< p,q<\infty$, $x=(x_{1},\cdot\cdot\cdot,x_{d-1},x_{d})$
 and
\begin{eqnarray}
\left\|f\right\|_{L_{t\in I}^{q}L_{x}^{p}}:=\left[\int_{I}
\left(\int_{\SR^{d}}|f|^{p}dx\right)^{\frac{q}{p}}dt\right]^{1/q}<\infty,I\subset \R.\label{2.020}
\end{eqnarray}
Then, the functions of the form
\begin{eqnarray*}
&&\sum\limits_{i=1}^{r}a_{i}(x)\chi_{E_{i}}(t)
\end{eqnarray*}
are dense in $L_{t\in I}^{q}L_{x}^{p}$. Here
$a_{i}(x)\in L^{p}$ and $E_{i}$ are pairwise disjoint measurable sets in $I$ with finite measures
and $r\geq1$ is an integer.
\end{Lemma}

For the proof of Lemma 2.7, see Proposition 5.5.6 of \cite{G}.

\begin{Lemma}\label{lem2.8}(The triangle inequality in the mixed Lebesgue spaces)
Suppose that  $1\leq  p,q<\infty$, $x=(x_{1},\cdot\cdot\cdot,x_{d-1},x_{d})$,
$\left\|f\right\|_{L_{t\in I}^{q}L_{x}^{p}}<\infty$ and  $f,g\in L_{t\in I}^{q}L_{x}^{p}(\R^{d})$. Then, we have
\begin{eqnarray}
\left\|f+g\right\|_{L_{t\in I}^{q}L_{x}^{p}}\leq \left\|f\right\|_{L_{t\in I}^{q}L_{x}^{p}}
+\left\|g\right\|_{L_{t\in I}^{q}L_{x}^{p}}.\label{2.021}
\end{eqnarray}

\end{Lemma}

For the proof of Lemma 2.8, we refer the reader to Lemma A.1  of \cite{Frank}.

\begin{Lemma}\label{lem2.9}(The continuity in the mixed Lebesgue spaces)
Let $1\leq p,q<\infty$, $x=(x_{1},\cdot\cdot\cdot,x_{d-1},x_{d})$
and $\left\|f\right\|_{L_{t\in I}^{q}L_{x}^{p}}<\infty$. Then, we have
\begin{eqnarray}
&&\lim\limits_{t_{2}\longrightarrow t_{1}}\left(\int_{t_{1}}^{t_{2}}
\left(\int_{\SR^{d}}|f|^{p}dx\right)^{\frac{q}{p}}dt\right)^{\frac{1}{q}}=0\label{2.022}.
\end{eqnarray}
\end{Lemma}
\noindent{\bf Proof.} For $f\in L_{t\in I}^{q}L_{x}^{p}$ and
$\forall \epsilon>0$,  it follows from Lemma 2.8 that
 there exists $r\geq1$ such that  the  functions of the form
\begin{eqnarray*}
&&\sum\limits_{i=1}^{r}a_{i}(x)\chi_{E_{i}}(t)
\end{eqnarray*}
are dense in $L_{t}^{q}L_{x}^{p}$. Here
$a_{i}(x)\in L^{p}$ and $E_{i}$ are pairwise disjoint
 measurable sets in $I$ with finite measures
and $r\geq1$ is an integer.
It follows from  Lemma 2.8 that
\begin{eqnarray}
&&\lim\limits_{t_{2}\longrightarrow t_{1}}\left(\int_{t_{1}}^{t_{2}}
\left(\int_{\SR^{d}}|f|^{p}dx\right)^{\frac{q}{p}}dt\right)^{\frac{1}{q}}\nonumber\\&&
\leq \lim\limits_{t_{2}\longrightarrow t_{1}}\left(\int_{t_{1}}^{t_{2}}
\left(\int_{\SR^{d}}|f-\sum\limits_{i=1}^{r}a_{i}(x)
\chi_{E_{i}}(t)|^{p}dx\right)^{\frac{q}{p}}dt\right)^{\frac{1}{q}}
\nonumber\\&&\qquad\qquad
+\lim\limits_{t_{2}\longrightarrow t_{1}}\left(\int_{t_{1}}^{t_{2}}
\left(\int_{\SR^{d}}|\sum\limits_{i=1}^{r}
a_{i}(x)\chi_{E_{i}}(t)|^{p}dx\right)^{\frac{q}{p}}dt\right)^{\frac{1}{q}}\nonumber\\
&&\leq \epsilon+\lim\limits_{t_{2}\longrightarrow t_{1}}\sum\limits_{i=1}^{r}
\left(\int_{\SR^{d}}|a_{i}(x)|^{p}dx\right)^{\frac{1}{p}}
{\rm mes}(E_{i})^{\frac{1}{q}}\nonumber\\&&\leq \epsilon+\lim\limits_{t_{2}
\longrightarrow t_{1}}\sum\limits_{i=1}^{r}
\left(\int_{\SR^{d}}|a_{i}(x)|^{p}dx\right)^{\frac{1}{p}}
(t_{2}-t_{1})^{\frac{1}{q}}\leq\epsilon.\label{2.023}
\end{eqnarray}

This completes the proof of Lemma 2.9.

\begin{Lemma}(Duality principle in the mixed Lebesgue spaces)
\label{2.10} We assume that
  $A$ is a bounded operator from $L^{2}$ to
$L_{t\in I_{1}}^{2p}L_{x\in I_{2}}^{2q}$ with
$I_{1}\subset \R, I_{2}\subset \R^{d}$
and $\alpha\geq1$. Then, the following two statements are equivalent.
\begin{eqnarray}
&&(i)\left\|WAA^{*}\overline{W}\right\|_{\mathfrak{S}^{\alpha}}\leq C
\|W\|_{L_{t\in I_{1}}^{2p^{\prime}}L_{x\in I_{2}}^{2q^{\prime}}}^{2},\label{2.024}\\
&&(ii)\left\|\sum\limits_{j=1}^{+\infty}\lambda_{j}|Af_{j}|^{2}\right\|_{L_{t\in I_{1}}^{p}L_{x\in I_{2}}^{q}}\leq
C\left(\sum\limits_{j=1}^{+\infty}|\lambda_{j}|^{\alpha^{\prime}}\right)^{1/\alpha^{\prime}}
=C\left\|\gamma(t)\right\|_{\mathfrak{S}^{\alpha^{\prime}}}.\label{2.025}
\end{eqnarray}
Here, $\lambda_{j}(j\in\mathbf{N}^{+})$ are eigenvalues of $\gamma(t).$
\end{Lemma}
 \noindent{\bf Proof.} By using the idea of  Lemma 3 of \cite{FS2017}
  and the fact that $WA(WA)^{*}$ and $(WA)^{*}WA$ possess
   the same singular values,  we derive that Lemma 2.10 is valid.

 This completes the proof of Lemma 2.10.

\noindent{\bf Remark 3.}
Thus, according to spectral mapping Theorem of Page 227 in \cite{Y},  we have
\begin{eqnarray}
\left[Trace(|\gamma(t)|^{\alpha})\right]^{\frac{1}{\alpha}}=
\left[\sum\limits_{j=1}^{\infty}|\lambda_{j}|^{\alpha}\right]^{\frac{1}{\alpha}}
=\left\|\gamma(t)\right\|_{\mathfrak{S}^{\alpha}}.\label{2.026}
\end{eqnarray}
Here, $\lambda_{j}(j\in \mathbf{N}^{+})$ are eigenvalues of $\gamma(t).$

\begin{Lemma}(Van der Corput's Lemma) \label{2.11} We assume that
  $k\geq1$ and $a,b\in \R$ with $a<b$. Suppose that $\phi:[a,b]\rightarrow\R$
  satisfies
\begin{eqnarray}
\left|\phi^{(k)}\right|\geq1\label{2.027}
\end{eqnarray}
for all $x\in [a,b].$ Suppose that $f:[a,b]\rightarrow\R$ is
differentiable and $f^{\prime}$ is integrable on $[a,b]$.
 Then, there exists  a constant $C_{k}$ depending only on $k$
 such that
\begin{eqnarray}
\left|\int_{a}^{b}e^{i\lambda \phi}fdx\right|\leq C_{k}\left(|f(b)|
+\int_{a}^{b}|f^{\prime}|dx\right)\lambda^{-1/k}.\label{2.028}
\end{eqnarray}

\end{Lemma}

For the proof of Lemma 2.11, we refer the readers to
Page 72, P. 334 of  \cite{Stein}.

\bigskip
\bigskip

\section{The bound of some complex integrals}
\medskip
\setcounter{equation}{0}
\setcounter{Theorem}{0} \setcounter{Lemma}{0}\setcounter{Definition}{0}\setcounter{Proposition}{3}

\setcounter{section}{3}

In this section, for any $b\in \R$ and $a\geq0$, by  using
Cauchy's principal value,
the Fourier transformation of generalized functions,  Gamma
function of Weierstrass definition,
Euler product formula as well as the Dirichlet test, for
$b\in \R$,  $a\geq0,c\geq0,d_{0}>0$
we prove that the following four types of complex integrals
\begin{eqnarray*}
&&P.V.\int_{\SR}e^{-it\rho}\rho^{-1+ib}d\rho,
P.V.\int_{|\rho|\geq a}e^{-it\rho}\rho^{-1+ib}d\rho,\nonumber\\
&&
P.V.\int_{|\rho|\leq a}e^{-it\rho}\rho^{-1+ib}d\rho,
P.V.\int_{c\leq |\rho|\leq d_{0}}e^{-it\rho}\rho^{-1+ib}d\rho
\end{eqnarray*}
 converge. As a byproduct, for $b\neq 0$,  $a_{j}\in \R(j=1,2,3,4)$
 and $\epsilon>0,M>0$, we prove that
 \begin{eqnarray*}
&&\int_{\SR}e^{-it\rho}\rho^{-1+ib}d\rho,\int_{a_{1}}^{+\infty}e^{-it\rho}\rho^{-1+ib}d\rho,
\int_{-\infty}^{a_{2}}e^{-it\rho}\rho^{-1+ib}d\rho,\nonumber\\&&
\int_{a_{3}\leq \rho\leq a_{4}}e^{-it\rho}\rho^{-1+ib}d\rho,
\int_{\epsilon\leq |\rho|\leq M}e^{-it\rho}\rho^{-1+ib}d\rho
\end{eqnarray*}
converge. These results generalized the result of the last line
 of page 205 of
 \cite{Vega1992}.

For $b,c,d_{0}\in \R,$     we also prove that
\begin{eqnarray*}
\left|b\int_{c}^{d_{0}}e^{-it\rho}\rho^{-1+ib}d\rho\right|\leq C(1+e^{-\pi b})(1+|b|)^{2}.
\end{eqnarray*}
This extends Lemma 1 of \cite{BLN2021}.

\begin{Lemma}\label{3.1}
Let $b\in \R.$ Then, we have
\begin{align}
\int_{\SR}\left[\int_{0}^{+\infty}\rho^{-1+ib}d\rho\right]\phi(b) db &
=2\pi\phi(0)=2\pi\int_{\SR}\delta (b)\phi(b)db.\label{3.01}
\end{align}
Thus,  we have
\begin{eqnarray*}
\int_{0}^{+\infty}\rho^{-1+ib}d\rho=2\pi\delta (b)
\end{eqnarray*}
in the sense of distribution.
Here $\delta(b)$ is the Dirac function and $\phi\in C_{0}^{\infty}(\R)$.
\end{Lemma}
\noindent{\bf Proof.} We present two methods to prove Lemma 3.1.
We present the first method to prove Lemma 3.1.
For $b\in \R,$ we have
\begin{eqnarray}
&&\int_{\SR}\left[\int_{0}^{+\infty}\rho^{-1+ib}d\rho\right]\phi(b)db\nonumber\\&&=
\int_{\SR} \left[\int_{0}^{+\infty}\rho^{ib}d\ln\rho\ \right] \phi(b)db
 \quad\underline{\underline{t=\ln\rho}}\ \
  \nonumber\\&&=\int_{\SR}\left[\int_{-\infty}^{+\infty}e^{ibt}dt\right]\phi(b)db=
\int_{\SR}\left[\int_{-\infty}^{+\infty}e^{ibt}dt\right]\phi(b)db\nonumber\\&&
=\int_{\SR}\left[\lim\limits_{\alpha\longrightarrow 0^{+}}
\int_{-\infty}^{+\infty}e^{ibt}e^{-\alpha|t|}dt\right]\phi(b)db\nonumber\\&&=\int_{\SR}\left[
\lim\limits_{\alpha\longrightarrow 0^{+}}\int_{0}^{+\infty}e^{ibt}e^{-\alpha t}dt+
\lim\limits_{\alpha\longrightarrow 0^{+}}\int_{-\infty}^{0}e^{ibt}e^{\alpha t}dt\right]\phi(b)db\nonumber\\
&&=\int_{\SR}\left[\lim\limits_{\alpha\longrightarrow 0^{+}}\frac{2\alpha}{\alpha^{2}+b^{2}}\right]\phi(b)db
=2\pi\phi(0).\label{3.02}
\end{eqnarray}
Thus, we have
\begin{eqnarray}
\int_{0}^{+\infty}\rho^{-1+ib}d\rho=
\lim\limits_{\alpha\longrightarrow 0^{+}}
\frac{2\alpha}{\alpha^{2}+b^{2}}=2\pi\delta(b)\label{3.03}
\end{eqnarray}
in the sense of distribution.
Now we present the second method to prove Lemma 3.1.
For any  $\phi\in C_{0}^{\infty}(\R)$, we have
\begin{eqnarray}
&&\int_{\SR}\left(\int_{0}^{+\infty}\rho^{-1+ib}d\rho\right)\phi db\ \
\underline{\underline{t=\ln\rho}}\ \
\int_{\SR}\left(\int_{\SR}e^{ibt}dt\right)\phi db\nonumber\\
&&=\int_{\SR}\int_{\SR}e^{ibt}\phi dbdt=
2\pi\int_{\SR}(\mathscr{F}\phi)(-t)dt \ \underline{\underline{t=-s}}
\ \ \nonumber\\&&=2\pi\int_{\SR}(\mathscr{F}\phi)(s)ds=2\pi\phi(0).\label{3.04}
\end{eqnarray}
From (\ref{3.04}),   we have
\begin{eqnarray*}
\int_{0}^{+\infty}\rho^{-1+ib}d\rho=2\pi\delta (b)
\end{eqnarray*}
in the sense of distribution.

This completes the proof of Lemma 3.1.

\begin{Lemma}\label{3.2}
For  $b\in \R,$ we have
\begin{align}
\int_{\SR}\rho^{-1+ib}d\rho & =2\pi(1-e^{-\pi b})\delta (b)\label{3.05}
\end{align}
in the sense of distribution.
Here $\delta(b)$ is the Dirac function.
\end{Lemma}
\noindent{\bf Proof.}For $b\in \R,$ we have
\begin{eqnarray}
\int_{-\infty}^{0}\rho^{-1+ib}d\rho\ \underline{\underline{\rho=-s}}\ \
-e^{-\pi b}\int_{0}^{+\infty}s^{-1+ib}ds=-2\pi e^{-\pi b}\delta(b)\label{3.06}.
\end{eqnarray}
From (\ref{3.06})  and Lemma 3.1,   we have
\begin{eqnarray}
&&\int_{\SR}\rho^{-1+ib}d\rho=\int_{0}^{\infty}\rho^{-1+ib}d\rho+\int_{-\infty}^{0}\rho^{-1+ib}d\rho\nonumber\\&&
=(1-e^{-\pi b})\int_{0}^{+\infty}\rho^{-1+ib}d\rho=2\pi(1-e^{-\pi b})\delta (b)\label{3.07}.
\end{eqnarray}

This completes the proof of Lemma 3.2.

\begin{Lemma}\label{3.3}
Let $b\in \R,b\neq0.$ Then,   we have
\begin{align}
\int_{0}^{1}\rho^{-1+ib}d\rho & =\pi\delta (b)-\frac{i}{b}=-\frac{i}{b}.
\end{align}\label{3.08}

Here $\delta(b)$ is the Dirac function.
\end{Lemma}
\noindent{\bf Proof.}
For $b\in \R,b\neq0$,  we have
\begin{eqnarray}
&&\int_{0}^{1}\rho^{-1+ib}d\rho=
\int_{0}^{1}\rho^{ib}d\ln\rho\ \ \underline{\underline{t=\ln\rho}}\ \
  \nonumber\\&&=\int_{-\infty}^{0}e^{ibt}dt=
\int_{-\infty}^{0}e^{ibt}dt=\lim\limits_{\alpha\longrightarrow 0^{+}}\int_{-\infty}^{0}e^{ibt}e^{\alpha t}dt\nonumber\\&&=
\lim\limits_{\alpha\longrightarrow 0^{+}}\int_{-\infty}^{0}e^{ibt}e^{\alpha t}dt\nonumber\\
&&=\lim\limits_{\alpha\longrightarrow 0^{+}}\frac{\alpha-ib}{\alpha^{2}+b^{2}}
=\pi\delta (b)-\frac{i}{b}=-\frac{i}{b}\label{3.09}
\end{eqnarray}
in the sense of distribution.
In the last equality, we use $\lim\limits_{\alpha\longrightarrow 0^{+}}\frac{\alpha}{\alpha^{2}+b^{2}}=\pi \delta(b)$.

This completes the proof of Lemma 3.3.

\begin{Lemma}\label{3.4}
Let $b\in \R,b\neq0.$ Then,  we have
\begin{align}
\int_{-1}^{0} \rho^{-1+ib}d\rho=\frac{ie^{-\pi b}}{b}.\label{3.010}
\end{align}
\end{Lemma}
\noindent{\bf Proof.}
For $b\in \R,$ we have
\begin{eqnarray}
\int_{-1}^{0}\rho^{-1+ib}d\rho\ \underline{\underline{\rho=-s}}
\ \ -e^{-\pi b}\int_{0}^{1}s^{-1+ib}ds\label{3.011}.
\end{eqnarray}
From (\ref{3.011}) and Lemma 3.3, we have
\begin{eqnarray}
&&\int_{-1}^{0} \rho^{-1+ib}d\rho=\frac{ie^{-\pi b}}{b}\label{3.012}
\end{eqnarray}
in the sense of distribution.

This completes the proof of Lemma 3.4.

\begin{Lemma}\label{3.5}
Let $b\in \R,b\neq0.$ Then, for a.e. $b\in \R,$ we have
\begin{align}
\int_{1}^{+\infty}\rho^{-1+ib}d\rho  =\pi\delta (b)+\frac{i}{b}=
\frac{i}{b}.\label{3.013}
\end{align}
Here $\delta(b)$ is the Dirac function.
\end{Lemma}
\noindent{\bf Proof.} Combining Lemma 3.1 with Lemma 3.3,
we have that Lemma 3.5 is valid.

This completes the proof of Lemma 3.5.

\begin{Lemma}\label{3.6}
Let $b\in \R,b\neq0.$ Then, for a.e. $b\in \R,$ we have
\begin{align}
& \int_{-\infty}^{-1}\rho^{-1+ib}d\rho=-\frac{ie^{-\pi b}}{b}.\label{3.014}
\end{align}

\end{Lemma}
\noindent{\bf Proof.}From Lemma 3.1,  (\ref{3.06}) and (\ref{3.010}),
we have that Lemma 3.6 is valid.

This completes the proof of Lemma 3.6.

\begin{Lemma}\label{3.7}
Let $b\in \R,b\neq0$. Then,  we have
\begin{eqnarray}
\left|\Gamma(ib)\right|=\left(\frac{\pi}{bsinh\pi b}\right)^{1/2}.\label{3.015}
\end{eqnarray}
Here  $sinh\pi b=\frac{e^{\pi b}-e^{-\pi b}}{2}$.
\end{Lemma}
\noindent{\bf Proof.} From A.7 of \cite{G}, we have
\begin{eqnarray}
\frac{1}{\Gamma (z)}=ze^{\gamma z}
\prod\limits_{n=1}^{\infty}\left(1+\frac{z}{n}\right)e^{-z/n},\gamma=
\lim\limits_{n\longrightarrow +\infty}\left[\sum\limits_{k=1}^{n}
\frac{1}{k}-ln n\right].\label{3.016}
\end{eqnarray}
(\ref{3.016}) is called as Gamma function of Weierstrass
 definition and $\gamma$ is called as Euler-Mascheroni constant.
From (\ref{3.016}), we have
\begin{eqnarray}
\left|\frac{1}{\Gamma (ib)}\right|=
\left|ibe^{\gamma ib}\prod\limits_{n=1}^{\infty}
\left(1+\frac{ib}{n}\right)e^{-ib/n}\right|
=|b|\left[\prod\limits_{n=1}^{\infty}
\left(1+\frac{b^{2}}{n^{2}}\right)\right]^{1/2}.\label{3.017}
\end{eqnarray}
From A.8 of \cite{G}, we have
\begin{eqnarray}
\frac{sin \pi z}{\pi z}=\prod\limits_{n=1}^{\infty}
\left(1-\frac{z^{2}}{n^{2}}\right),\label{3.018}
\end{eqnarray}
which is called as {\bf Euler product formula}.
In particular, take $z=ib$ in (\ref{3.018}), we have
\begin{eqnarray}
\frac{sin i\pi b}{i\pi b}=\prod\limits_{n=1}^{\infty}
\left(1+\frac{b^{2}}{n^{2}}\right).\label{3.019}
\end{eqnarray}
Since
\begin{eqnarray}
sin i\pi b=\frac{e^{-\pi b}-e^{\pi b}}{2i}\label{3.020},
\end{eqnarray}
by using (\ref{3.019}), we have
\begin{eqnarray}
\prod\limits_{n=1}^{\infty}\left(1+\frac{b^{2}}{n^{2}}\right)
=\frac{sin i\pi b}{\pi ib}=\frac{sin h\pi b}{\pi b}.\label{3.021}
\end{eqnarray}
Combining (\ref{3.017}) with (\ref{3.021}), we have
\begin{eqnarray}
\left|\Gamma (ib)\right|=\left[\frac{\pi }{bsin h\pi b}\right]^{1/2}
=\left[\frac{2\pi }{b(e^{\pi b}-e^{-\pi b})}\right]^{1/2}.\label{3.022}
\end{eqnarray}

This completes the proof of Lemma 3.7.

\begin{Lemma}\label{3.8}
Let $b\in \R.$ Then, we have
\begin{align}
P.V.\int_{\SR}e^{-it\rho}\rho^{-1+ib}d\rho=C_{0}(b) .\label{3.023}
\end{align}
Here \bigskip

$\qquad C_{0}(b)=\left\{
        \begin{array}{ll}
         0, & \hbox{$t=0,b\in \R$;} \\
          i\pi sgn t , & \hbox{$t\neq0,b=0$;} \\
          t^{-ib}\Gamma(ib)(e^{\frac{b\pi}{2}}-e^{\frac{-3b\pi}{2}}),
           & \hbox{$t>0,b\neq0$};\\
            0, & \hbox{$t<0,b\neq0$.}
        \end{array}
      \right.$

\bigskip
In particular, we have
\begin{align}
\left|P.V.\int_{\SR}e^{-it\rho}\rho^{-1+ib}d\rho\right|=\left|C_{0}(b)\right| .\nonumber
\end{align}
Here \bigskip

$\qquad \left|C_{0}(b)\right|=\left\{
        \begin{array}{ll}
         0, & \hbox{$t=0,b\in \R$;} \\
          \pi, & \hbox{$t\neq0,b=0$;} \\
         \left[\frac{2\pi (1-e^{-2b\pi})}{b}\right]^{1/2},
           & \hbox{$t>0,b\neq0$};\\
            0, & \hbox{$t<0,b\neq0$.}
        \end{array}
      \right.$

\end{Lemma}
\noindent{\bf Proof.}When $t=0$ and $b=0$, we have
\begin{align}
P.V.\int_{\SR}\rho^{-1}d\rho=0.\label{3.024}
\end{align}
When $t=0$ and $b\neq0$, from Lemma 3.2,  we have
\begin{align}
P. V.\int_{\SR}\rho^{-1+ib}d\rho =\int_{\SR}\rho^{-1+ib}d\rho
=2\pi(1-e^{-\pi b})\delta (b)=0.\label{3.025}
\end{align}
When $t\neq0$ and $b=0$, we have
\begin{align}
P.V.\int_{\SR}e^{-it\rho}\rho^{-1}d\rho=i\pi sgn t .\label{3.026}
\end{align}
When $t\neq0$ and $b\neq0$, from Lemma 2.2 and  page 172 of \cite{GS}
 which presents $ \mathscr{F}_{x}
\left[x_{+}^{\lambda}\right]=ie^{i\lambda \frac{\pi}{2}}
 \Gamma(\lambda+1)(t+i0)^{-\lambda-1},$    we have
\begin{eqnarray}
&&P.V.\int_{\SR}e^{-it\rho}\rho^{-1+ib}d\rho=\int_{0}^{+\infty}
e^{-it\rho}\rho^{-1+ib}d\rho-(-1)^{ib}
\int_{0}^{+\infty}e^{it\rho}\rho^{-1+ib}d\rho\nonumber\\
&&=e^{-\frac{b\pi}{2}}\Gamma(bi)
\left[(-t+i0)^{-bi}-(-1)^{ib}(t+i0)^{-bi}\right]\nonumber\\
&&=\left\{
        \begin{array}{ll}
         t^{-ib}\Gamma(ib)(e^{\frac{b\pi}{2}}-e^{\frac{-3b\pi}{2}}), &
         \hbox{$t>0$;} \label{3.027}\\
          0, & \hbox{$t<0$.}
        \end{array}
      \right.
\end{eqnarray}
Combining Lemma 3.7 with (\ref{3.024})-(\ref{3.027}),
we have that Lemma 3.8 is valid.

This completes the proof of Lemma 3.8.

\noindent {\bf Remark 4.} In page 204 of \cite{Vega1992}, Vega
studied the complex integral
\begin{eqnarray*}
\lim\limits_{\epsilon\longrightarrow0,\>M\longrightarrow+\infty}
\int_{\epsilon\leq |\rho|\leq M}e^{-it\rho}
\rho^{-1+ib}d\rho=P.V.\int_{\SR}e^{-it\rho}\rho^{-1+ib}d\rho
\end{eqnarray*}
 and claim the similar conclusion, however, he did not
  give the strict proof.
The conclusion of  \cite{Vega1992} has been applied to
 \cite{BHLNS,BLN,MS,N}.
From Lemma 3.8, we have that $|C_{0}(b)|\leq C
\left[1+\left(\frac{1-e^{-2\pi b}}{b}\right)^{1/2}\right]
\leq C(1+e^{-\pi b})\left(\sqrt{|b|}+\frac{1}{\sqrt{|b|}}\right)^{2}
\leq C(e^{\pi b}+e^{-\pi b})\left(\sqrt{|b|}+\frac{1}{\sqrt{|b|}}\right)^{2}(b\neq0)$
and $|C_{0}(b)|\leq \pi(b=0)$. In particular, from page 172 of \cite{GS}
 and  Lemma 2.1 which presents $ \mathscr{F}_{x}
\left[x_{+}^{\lambda}\right]=ie^{i\lambda \frac{\pi}{2}}
\Gamma(\lambda+1)(t+i0)^{-\lambda-1},$
we know that
\begin{eqnarray*}
\left|\int_{0}^{\infty}e^{-it\rho}\rho^{-1+bi}d\rho\right|
=e^{-\frac{b\pi}{2}}\left|\Gamma(bi)(-t+i0)^{-bi}\right|
\end{eqnarray*}
$\qquad =\left\{
        \begin{array}{ll}
         e^{\frac{\pi b}{2}}|\Gamma(bi)|, & \hbox{$t>0,b\neq0$,} \\
            e^{-\frac{\pi b}{2}}|\Gamma(bi)|, & \hbox{$t<0,b\neq0$.}
        \end{array}
      \right.$
From Lemmas 2.6, 3.7 and the above equality, we know
\begin{eqnarray*}
\left|\int_{0}^{\infty}e^{-it\rho}\rho^{-1+bi}d\rho\right|\leq
 C\left(\frac{1}{|b|}+1\right)\left(1+e^{\pi b}\right).
\end{eqnarray*}
By using a similar proof, we have
\begin{eqnarray*}
\left|\int_{-\infty}^{0}e^{-it\rho}\rho^{-1+bi}d\rho\right|\leq
 C\left(\frac{1}{|b|}+1\right)\left(1+e^{-\pi b}\right).
\end{eqnarray*}
However,
\begin{eqnarray*}
\left|\int_{0}^{\infty}e^{-it\rho}\rho^{-1}d\rho\right|\geq
 \left|\int_{0}^{\infty}\frac{cost\rho}{\rho}d\rho\right|
 =\infty,\left|\int_{-\infty}^{0}e^{-it\rho}\rho^{-1}d\rho\right|\geq
 \left|\int_{-\infty}^{0}\frac{cost\rho}{\rho}d\rho\right|=\infty,
\end{eqnarray*}
which imply that the imaginary part $b$ in
\begin{eqnarray*}
\left|\int_{0}^{\infty}e^{-it\rho}\rho^{-1+bi}d\rho\right|,\left|\int_{-\infty}^{0}e^{-it\rho}\rho^{-1+bi}d\rho\right|
\end{eqnarray*}
guarantee the convergence of
\begin{eqnarray*}
\left|\int_{0}^{\infty}e^{-it\rho}\rho^{-1+bi}d\rho\right|,\left|\int_{-\infty}^{0}e^{-it\rho}\rho^{-1+bi}d\rho\right|.
\end{eqnarray*}

\begin{Lemma}\label{3.9}
Let $a>0$. Then,  we have
\begin{eqnarray*}
P.V. \int_{|\rho|\geq a}\frac{1}{\rho}d\rho=0.
\end{eqnarray*}
\end{Lemma}
\noindent{\bf Proof.}
\begin{eqnarray*}
P.V. \int_{|\rho|\geq a}\frac{1}{\rho}d\rho=P.V.
\int_{\SR}\frac{1}{\rho}d\rho-P.V. \int_{|\rho|\leq a}
\frac{1}{\rho}d\rho=0.
\end{eqnarray*}

This completes the proof of Lemma 3.9.
\begin{Lemma}\label{3.10}
Let $b\neq0$ and $a\geq0$. Then,  we have
\begin{eqnarray*}
\left|\int_{a}^{+\infty}\rho^{-1+bi}d\rho\right|\leq
\frac{3}{|b|},\left|\int_{-\infty}^{-a}\rho^{-1+bi}d\rho\right|\leq
\frac{3e^{-\pi b}}{|b|}.
\end{eqnarray*}
\end{Lemma}
\noindent{\bf Proof.}From Lemma 1 of \cite{BLN2021} and Lemmas 3.3, 3.5, we have
\begin{eqnarray}
&&\left|\int_{a}^{+\infty}\rho^{-1+bi}d\rho\right|
\leq \left|\int_{a}^{1}\rho^{-1+bi}d\rho\right|+\left|\int_{1}^{\infty}
\rho^{-1+bi}d\rho\right|\nonumber\\&&\leq \frac{2}{|b|}+ \frac{1}{|b|}
=\frac{3}{|b|}.\label{3.028}
\end{eqnarray}
From (\ref{3.028}), we have
\begin{eqnarray*}
\left|\int_{-\infty}^{-a}\rho^{-1+bi}d\rho\right|\ \underline{\underline{\rho=-s}}\ \ e^{-\pi b}
\left|\int_{a}^{+\infty}s^{-1+bi}ds\right|\leq \frac{3e^{-\pi b}}{|b|}.  \label{3.029}
\end{eqnarray*}

This completes the proof of Lemma 3.10.

\begin{Lemma}\label{3.11}
Let $b\in \R$. Then, for $\forall \epsilon>0$,  there exists $M_{0}\geq 2|b|+2$ such that
\begin{align}
&&\left|\int_{M_{0}}^{+\infty}\frac{cos (y-blny)}{y}dy\right|\leq \epsilon,
\left| \int_{M_{0}}^{+\infty}\frac{sin(y-blny)}{y}dy\right|\leq \epsilon,
\left|\int_{M_{0}}^{+\infty}e^{\pm iy}y^{-1+bi}dy\right|\leq 2\epsilon.  \label{3.029}
\end{align}

\end{Lemma}
\noindent{\bf Proof.} We  only prove that
\begin{eqnarray}
\int_{1}^{+\infty}y^{-1}cos (y-blny)dy \label{3.030}
\end{eqnarray}
converges since other cases can be similarly proved.
$\forall \epsilon>0$, $\exists M={\rm max}\left\{\frac{1}{\epsilon}+1, 2|b|+2\right\}$
such that $\frac{1}{M}<\epsilon$.
We define
\begin{eqnarray}
F(A)=\int_{1}^{A}\frac{cos (y-blny)}{y}dy\label{3.031}.
\end{eqnarray}
For $A_{1}, A_{2}\in [M,+\infty)$,  by using the second mean
value theorem for integrals,
we have
\begin{eqnarray}
&&F(A_1)-F(A_2)=\int_{A_1}^{A_2}y^{-1}cos (y-blny)dy\nonumber\\
&&=\frac{1}{A_{1}}
\int_{A_1}^{\xi}cos (y-blny)dy+\frac{1}{A_{2}}\int_{\xi}^{A_2}
cos (y-blny)dy.\label{3.032}
\end{eqnarray}
By using integration by parts, since $A_{2}\geq\xi\geq A_{1}\geq
 M\geq 2|b|+2$, it follows that
\begin{eqnarray}
&&\hspace{-1cm}\left|\int_{A_1}^{\xi}cos (y-blny)dy\right|\nonumber\\&&
=\left|\frac{\xi}{\xi-b}sin(\xi-bln\xi)-\frac{A_{1}}{A_{1}-b}
sin(A_{1}-blnA_{1})+\int_{A_{1}}^{\xi}
\frac{bsin (y-blny)}{(y-b)^{2}}dy\right|\nonumber\\&&\leq
\left|\frac{\xi}{\xi-b}sin(\xi-bln\xi)\right|+\left|\frac{A_{1}}{A_{1}-b}
 sin(A_{1}-blnA_{1})\right|
+\left|\int_{A_{1}}^{\xi}\frac{bsin (y-blny)}{(y-b)^{2}}dy\right|
\nonumber\\
&&\leq \left|\frac{\xi}{\xi-b}\right|+\left|\frac{A_{1}}{A_{1}-b} \right|
+|b|\left|\int_{A_{1}}^{\xi}\frac{1}{(y-b)^{2}}dy\right|\nonumber\\
&&\leq 2+2+|b|\left|\frac {1}{A_{1}-b}-\frac{1}{\xi-b}\right|\nonumber\\
&&\leq 4+\frac{|b|}{A_{1}-b}+\frac{|b|}{\xi-b}\leq 4+4=8
\label{3.033}
\end{eqnarray}
and
\begin{eqnarray}
&&\hspace{-1cm}\left|\int_{\xi}^{A_2}cos (y-blny)dy\right|\nonumber\\&&
=\left|\frac{A_{2}}{A_{2}-b} sin(A_{2}-blnA_{2})-\frac{\xi}{\xi-b}sin(\xi-bln\xi)+
\int_{\xi}^{A_{2}}\frac{bsin (y-blny)}{(y-b)^{2}}dy\right|\nonumber\\&&
\leq
\left|\frac{\xi}{\xi-b}sin(\xi-bln\xi)\right|+\left|\frac{A_{2}}{A_{2}-b}
sin(A_{2}-blnA_{2})\right|
+\left|\int_{\xi}^{A_{2}}\frac{bsin (y-blny)}{(y-b)^{2}}dy\right|
\nonumber\\
&&\leq \left|\frac{\xi}{\xi-b}\right|+\left|\frac{A_{2}}{A_{2}-b} \right|
+|b|\left|\int_{\xi}^{A_{2}}\frac{1}{(y-b)^{2}}dy\right|\nonumber\\
&&\leq 2+2+|b|\left|\frac {1}{A_{2}-b}-\frac{1}{\xi-b}\right|\nonumber\\
&&\leq 4+\frac{|b|}{A_{2}-b}+\frac{|b|}{\xi-b}\leq 4+4=8.
\label{3.034}
\end{eqnarray}
By using (\ref{3.032})-(\ref{3.034}), we have
\begin{eqnarray}
\left|F(A_1)-F(A_2)\right|\leq \frac{1}{A_{1}}+\frac{1}{A_{2}}\leq
 \frac{2}{M}\leq 2\epsilon.\label{3.035}
\end{eqnarray}
By using   (\ref{3.035}), the Cauchy
 criterion for infinite integral convergence, it follows that
 (\ref{3.030}) converges.
Thus, for $\forall \epsilon>0,$  there exists
$M_{0}\geq M\geq 2|b|+2$ such that
\begin{align}
&&\left|\int_{M_{0}}^{+\infty}\frac{cos (y-blny)}{y}dy\right|
\leq \epsilon.  \label{3.036}
\end{align}
By using a proof similar to (\ref{3.036}), we have that
\begin{align}
&& \left| \int_{M_{0}}^{+\infty}\frac{sin(y-blny)}{y}dy\right|
\leq \epsilon.  \label{3.037}
\end{align}
Combining (\ref{3.036}) with (\ref{3.037}), we have
\begin{eqnarray}
&&\left|\int_{M_{0}}^{+\infty}e^{-iy}y^{-1+bi}dy\right|\nonumber\\&&\leq
\left|\int_{M_{0}}^{+\infty}\frac{cos (y-blny)}{y}dy\right|+
\left| \int_{M_{0}}^{+\infty}\frac{sin(y-blny)}{y}dy\right|\leq
2\epsilon.  \label{3.038}
\end{eqnarray}

This completes the proof of Lemma 3.11.

\begin{Lemma}\label{3.12}
Let $t\in \R,b\neq 0$ and $a>0.$ Then,  we have
\begin{align}
&\left|\int_{a}^{+\infty}e^{-it\rho}\rho^{-1+bi}d\rho\right|
\leq C(1+e^{-\pi b})\left(\sqrt{|b|}+\frac{1}{\sqrt{|b|}}\right)^{2},
\label{3.039}\\&
\left|\int_{-\infty}^{-a}e^{-it\rho}\rho^{-1+bi}d\rho\right|
\leq C(1+e^{-\pi b})\left(\sqrt{|b|}+\frac{1}{\sqrt{|b|}}\right)^{2},
\label{3.040}\\
& \left|\int_{|\rho|\geq a}e^{-it\rho}\rho^{-1+bi}d\rho\right|
\leq C(1+e^{-\pi b})\left(\sqrt{|b|}+\frac{1}{\sqrt{|b|}}\right)^{2}.
 \label{3.041}
\end{align}
Here $C$ is independent of $a.$
\end{Lemma}
\noindent{\bf Proof.}When $t>0,$ for $\forall \epsilon>0,$ by using
 Lemma 1 of \cite{BLN2021} and  Lemma 3.11,  we have
\begin{eqnarray}
&&\left|\int_{a}^{+\infty}e^{-it\rho}\rho^{-1+bi}d\rho\right|\
\underline{\underline{t\rho=s}}\ \
 \left|\int_{at}^{+\infty}e^{-is}s^{-1+bi}ds\right|\nonumber\\
&&\leq \left|\int_{at}^{M_{0}}e^{-is}s^{-1+bi}ds\right|+
\left|\int_{M_{0}}^{+\infty}e^{-is}s^{-1+bi}ds\right|
\leq C\left(\sqrt{|b|}+\frac{1}{\sqrt{|b|}}\right)^{2}+\epsilon.\label{3.042}
\end{eqnarray}
Since $\epsilon>0$ is arbitrary, we have that (\ref{3.039}) is valid.
When $t<0,$  for $\forall \epsilon>0,$ by using Lemma 1 of
\cite{BLN2021} and  Lemma 3.11,  we have
\begin{eqnarray}
&&\left|\int_{a}^{+\infty}e^{-it\rho}\rho^{-1+bi}d\rho\right|
\ \underline{\underline{-t\rho=s}}\ \ e^{-\pi b}
\left|\int_{-at}^{+\infty}e^{is}s^{-1+bi}ds\right|\nonumber\\
&&\leq e^{-\pi b}\left|\int_{-at}^{M_{0}}e^{is}s^{-1+bi}ds\right|
+e^{-\pi b}\left|\int_{M_{0}}^{+\infty}e^{is}s^{-1+bi}ds\right|\nonumber\\&&
\leq Ce^{-\pi b}\left(\sqrt{|b|}+\frac{1}{\sqrt{|b|}}\right)^{2}
+e^{-\pi b}\epsilon.\label{3.043}
\end{eqnarray}
Since $\epsilon>0$ is arbitrary, we have that (\ref{3.039}) is valid.
When $t=0,$ from Lemma 3.10, we have that
\begin{align}
&&\left|\int_{a}^{+\infty}\rho^{-1+bi}dy\right|\leq \frac{3}{|b|}.\label{3.044}
\end{align}
By using a proof similar to (\ref{3.039}), we have that (\ref{3.040}) is valid.
Combining (\ref{3.039}) with (\ref{3.040}), we have that (\ref{3.041}) is valid.

This completes the proof of Lemma 3.12.

\begin{Lemma}\label{3.13}
Let $t\in \R$ and $a>0.$ Then,  we have
\begin{align}
&\left|P.V.\left(\int_{|\rho|\geq a}e^{-it\rho}\rho^{-1+bi}
d\rho\right)\right|\nonumber\\
&=\left|P.V.\left(\int_{a}^{+\infty}e^{-it\rho}\rho^{-1+bi}
d\rho+\int_{-\infty}^{-a}e^{-it\rho}
\rho^{-1+bi}d\rho\right)\right|\leq CC_{1}(b).\label{3.045}
\end{align}
Here $C$ is independent of $a.$
Here \bigskip

$\qquad  C_{1}(b)=\left\{
        \begin{array}{ll}
         0, & \hbox{$t=b=0$,} \\
          8, & \hbox{$t\neq0,b=0$,} \\
         (1+e^{-\pi b})\left(\sqrt{|b|}+\frac{1}{\sqrt{|b|}}\right)^{2},
           & \hbox{$b\neq0$}.
        \end{array}
      \right.$
\end{Lemma}
\noindent {\bf Proof.}When $t=b=0,$ we have
\begin{eqnarray}
\left|P.V.\left(\int_{|\rho|\geq a}\rho^{-1}d\rho\right)\right|=0.
\label{3.046}
\end{eqnarray}
When $t\neq0,b=0,$ from Lemma 2.4,  we have
\begin{align}
&\left|P.V.\left(\int_{|\rho|\geq a}e^{-it\rho}\rho^{-1}d\rho\right)\right|\nonumber\\&
=\left|P.V.\left(\int_{|\rho|\geq a}\rho^{-1}cos t\rho d\rho\right)+iP.V.\left(\int_{|\rho|\geq a}
\rho^{-1}sint \rho d\rho\right)\right|\nonumber\\
&=\left|P.V.\left(\int_{|\rho|\geq a}\rho^{-1}sint \rho d\rho\right)\right|\nonumber\\&
\leq \left|P.V.\left(\int_{a}^{+\infty}\rho^{-1}sint \rho d\rho\right)\right|+
\left|P.V.\left(\int_{-\infty}^{-a}\rho^{-1}sint \rho d\rho\right)\right|\leq 8.\label{3.047}
\end{align}
When $b\neq0,$ from Lemma 3.12, we have
\begin{align}
&\left|P.V.\left(\int_{|\rho|\geq a}e^{-it\rho}\rho^{-1+bi}d\rho\right)\right|\leq
C(1+e^{-\pi b})\left(\sqrt{|b|}+\frac{1}{\sqrt{|b|}}\right)^{2}.\label{3.048}
\end{align}

This completes the proof of Lemma 3.13.

\begin{Lemma}\label{3.14}
Let $b\in \R,a>0.$ Then, we have
\begin{align}
\left|P.V.\int_{|\rho|\leq a}e^{-it\rho}\rho^{-1+ib}d\rho\right|
\leq |C_{0}(b)|+ CC_{1}(b).\label{3.049}
\end{align}
Here $|C_{0}(b)|,C_{1}(b)$ are defined as in Lemmas 3.8, 3.13, respectively.
\end{Lemma}
\noindent {\bf Proof.} Combining Lemmas 3.8, 3.13 with
\begin{eqnarray}
&&\left|P.V.\int_{|\rho|\leq a}e^{-it\rho}\rho^{-1+ib}d\rho\right|\nonumber\\&&=
\left|P.V.\int_{\SR}e^{-it\rho}\rho^{-1+ib}d\rho-P.V.
\int_{|\rho|\geq a}e^{-it\rho}\rho^{-1+ib}d\rho\right|\nonumber\\
&&\leq \left|P.V.\int_{\SR}e^{-it\rho}\rho^{-1+ib}d\rho\right|+
\left|P.V.\int_{|\rho|\geq a}e^{-it\rho}\rho^{-1+ib}d\rho\right|\nonumber\\
&&\leq|C_{0}(b)| + CC_{1}(b).\label{3.050}
\end{eqnarray}

This completes the proof of Lemma 3.14.

\begin{Lemma}\label{3.15}
Let $b\in \R,\epsilon, M>0.$ Then, we have
\begin{align}
\left|P.V.\int_{\epsilon \leq |\rho|\leq M}e^{-it\rho}\rho^{-1+ib}d\rho\right|
\leq 2|C_{0}(b)|+ CC_{1}(b).\label{3.051}
\end{align}
Here $|C_{0}(b)|,C_{1}(b)$ are defined as in Lemmas 3.8, 3.13, respectively.
\end{Lemma}
\noindent {\bf Proof.} From Lemma 3.14, we have
\begin{eqnarray}
&&\left|P.V.\int_{\epsilon\leq |\rho|\leq M}e^{-it\rho}\rho^{-1+ib}d\rho\right|\nonumber\\&&
=\left|P.V.\int_{|\rho|\leq M}e^{-it\rho}\rho^{-1+ib}d\rho-
P.V.\int_{|\rho|\leq \epsilon}e^{-it\rho}\rho^{-1+ib}d\rho\right|\nonumber\\
&&\leq \left|P.V.\int_{|\xi|\leq M}e^{-it\rho}\rho^{-1+ib}d\rho\right|+
\left|P.V.\int_{|\rho|\leq \epsilon}e^{-it\rho}\rho^{-1+ib}d\rho\right|\nonumber\\
&&\leq 2|C_{0}(b)|+ CC_{1}(b).\label{3.052}
\end{eqnarray}

This completes the proof of Lemma 3.15.

\begin{Lemma}\label{3.16}
Let $b\in \R,b\neq0.$ Then, we have
\begin{align}
\int_{\SR}e^{-it\rho}\rho^{-1+ib}d\rho=C_{2}(b) .\label{3.053}
\end{align}
Here \bigskip

$\qquad C_{2}(b)=\left\{
        \begin{array}{ll}
         0, & \hbox{$t=0$,} \\

          t^{-ib}\Gamma(ib)(e^{\frac{b\pi}{2}}-e^{\frac{-3b\pi}{2}}),
           & \hbox{$t>0$},\\
            0, & \hbox{$t<0$.}
        \end{array}
      \right.$
\bigskip

In particular, we have
$\qquad |C_{2}(b)|=\left\{
        \begin{array}{ll}
         0, & \hbox{$t=0$,} \\

        \left[\frac{2\pi (1-e^{-2\pi b})}{b}\right]^{1/2},
           & \hbox{$t>0$},\\
            0, & \hbox{$t<0$.}
        \end{array}
      \right.$

\end{Lemma}
\noindent {\bf Proof.} From Lemma 3.2, (\ref{3.027}), we derive that
(\ref{3.053}) is valid.

This completes the proof of Lemma 3.16.

\begin{Lemma}\label{3.17}
Let $b\in \R,b\neq0,a>0.$ Then, we have
\begin{align}
\left|\int_{|\rho|\leq a}e^{-it\rho}\rho^{-1+ib}d\rho\right|\leq
 C(1+e^{-\pi b})\left(\sqrt{|b|}+\frac{1}{\sqrt{|b|}}\right)^{2}.\label{3.054}
\end{align}
Here $C$ is independent of $a.$
\end{Lemma}

\noindent {\bf Proof.} Combining  Lemma 3.16 with  (\ref{3.041}), we derive
\begin{eqnarray*}
&&\left|\int_{|\rho|\leq a}e^{-it\rho}\rho^{-1+ib}d\rho\right|=
\left|\int_{\SR}e^{-it\rho}\rho^{-1+ib}d\rho-\int_{|\rho|\geq a}
e^{-it\rho}\rho^{-1+ib}d\rho\right|\nonumber\\
&&\leq\left|\int_{\SR}e^{-it\rho}\rho^{-1+ib}d\rho\right|+
\left|\int_{|\rho|\geq a}e^{-it\rho}\rho^{-1+ib}d\rho\right|\nonumber\\
&&\leq  \left[\frac{2\pi (1-e^{-2\pi b})}{b}\right]^{1/2}+C(1+e^{-\pi b})
\left(\sqrt{|b|}+\frac{1}{\sqrt{|b|}}\right)^{2}\nonumber\\
&&\leq C(1+e^{-\pi b})\left(\sqrt{|b|}+\frac{1}{\sqrt{|b|}}\right)^{2}
\end{eqnarray*}
since
\begin{eqnarray*}
\left[\frac{2\pi (1-e^{-2\pi b})}{b}\right]^{1/2}\leq 2\pi(1+e^{-\pi b})
\left(\sqrt{|b|}+\frac{1}{\sqrt{|b|}}\right)^{2}.
\end{eqnarray*}
Here $C$ may vary from line to line.

This completes the proof of Lemma 3.17.

\begin{Lemma}\label{3.18}
Let $b\in \R$  and $d_{1}>0,d_{2}<0$. Then, we have
\begin{eqnarray}
&&\left|b\int_{0}^{d_{1}}e^{-it\rho}\rho^{-1+ib}d\rho\right|\leq
C(1+|b|)^{2},\label{3.055}\\
&&\left|b\int_{0}^{d_{2}}e^{-it\rho}\rho^{-1+ib}d\rho\right|\leq
 Ce^{-\pi b}(1+|b|)^{2}.\label{3.056}
\end{eqnarray}
Here $C$ is independent of $b,t,d_{1},d_{2}$.
\end{Lemma}
\noindent {\bf Proof.}When $b=0$, obviously, (\ref{3.055}) is valid. When $b\neq0$,
 we consider $t=0,t>0,t<0,$ respectively.
When $t=0,$ from Lemma 3.3 and Lemma 1 of  \cite{BLN2021},  we have
\begin{eqnarray}
&&\left|b\int_{0}^{d_{1}}\rho^{-1+ib}d\rho\right|
=\left|b\int_{0}^{1}\rho^{-1+ib}d\rho+b\int_{1}^{d_{1}}\rho^{-1+ib}d\rho\right|\nonumber\\
&&\leq \left|b\int_{0}^{1}\rho^{-1+ib}d\rho\right|+\left|b\int_{1}^{d_{1}}
\rho^{-1+ib}d\rho\right|\nonumber\\&&\leq C(1+|b|)^{2}.\label{3.057}
\end{eqnarray}
When $t>0,$  $\forall \epsilon>0$,  by using (2) of page 171 in \cite{GS} and Lemma 1
 of \cite{BLN2021} as well as Lemma 2.6,  we have
\begin{eqnarray}
&&\left|b\int_{0}^{d_{1}}e^{-it\rho}\rho^{-1+ib}d\rho\right|\ \underline{\underline{t\rho=s}}
\ \ \left|t^{-ib}b\int_{0}^{td_{1}}e^{-is}s^{-1+ib}ds\right|\nonumber\\&&
=\left|b\int_{0}^{+\infty}e^{-is}s^{-1+ib}ds-b\int_{td_{1}}^{+\infty}e^{-is}s^{-1+ib}ds\right|\nonumber\\
&&=\left|be^{-\frac{b\pi}{2}}\Gamma(b i)(-1+i0)^{-b i}-
b\int_{td_{1}}^{+\infty}e^{-is}s^{-1+ib}ds\right|\nonumber\\
&&=\left|be^{-\frac{b\pi}{2}}\Gamma(b i)(-1+i0)^{-b i}-
b\int_{td_{1}}^{M_{0}}e^{-is}s^{-1+ib}ds-b\int_{M_{0}}^{+\infty}e^{-is}s^{-1+ib}ds\right|\nonumber\\
&&\leq  \left|be^{-\frac{b\pi}{2}}\Gamma(b i)(-1+i0)^{-b i} \right|+
\left|b\int_{td_{1}}^{M_{0}}e^{-is}s^{-1+ib}ds\right|  +\left|b\int_{M_{0}}^{+\infty}
e^{-is}s^{-1+ib}ds\right| \nonumber\\&&
\leq \left(\frac{2\pi be^{-b \pi}}{e^{\pi b}-e^{-\pi b}}\right)^{\frac{1}{2}}+C(1+|b|)^{2}
+|b|\epsilon\nonumber\\&&\leq (2\pi|b|+1)^{1/2}+C(1+|b|)^{2}+|b|\epsilon\leq C(1+|b|)^{2}. \label{3.058}
\end{eqnarray}
Here $C$ may vary from line to line. When $t<0,$
by using  a  proof  similar to  case $t>0$, it follows that
\begin{eqnarray}
\left|b\int_{0}^{d_{1}}e^{-it\rho}\rho^{-1+ib}d\rho\right|\leq C(1+|b|)^{2}\label{3.059}.
\end{eqnarray}
Since
\begin{eqnarray}
&&\left|b\int_{0}^{d_{2}}e^{-it\rho}\rho^{-1+ib}d\rho\right|=\ \underline{\underline{\rho=-s}}
\ \ e^{-\pi b}\left|b\int_{0}^{-d_{2}}e^{its}s^{-1+ib}ds\right|
,\label{3.060}
\end{eqnarray}
by using a proof similar to (\ref{3.059}), we have
\begin{eqnarray}
\left|b\int_{0}^{d_{2}}e^{-it\rho}\rho^{-1+ib}d\rho\right|\leq Ce^{-\pi b}(1+|b|)^{2}.\label{3.061}
\label{3.071}
\end{eqnarray}

This completes the proof of Lemma 3.18.

\begin{Lemma}\label{3.19}
Let $c,d_{3},b \in \R.$ Then, we have
\begin{eqnarray}
\left|b\int_{c}^{d_{3}}e^{-it\rho}\rho^{-1+ib}d\rho\right|\leq CC_{3}(b).\label{3.062}
\end{eqnarray}
Here \bigskip

$\qquad C_{3}(b)=\left\{
        \begin{array}{ll}
         0, & \hbox{$b=0,c,d_{3}\in \R$,} \\
          (1+|b|)^{2}, & \hbox{$b\neq0,c>0,d_{3}>0$,} \\
          e^{-\pi b}(1+|b|)^{2},
           & \hbox{$b\neq0,c<0,d_{3}<0$},\\
            (1+e^{-\pi b})(1+|b|)^{2}, & \hbox{$b\neq0,cd_{3}<0$},\\
            (1+e^{-\pi b})(1+|b|)^{2}, & \hbox{$b\neq0,cd_{3}=0$.}
        \end{array}
      \right.$

\bigskip

Here $C$ is independent of $c,b,d_{3},t$.
\end{Lemma}
\noindent{\bf Proof.}When $b=0$, obviously, (\ref{3.062}) is valid.

\noindent When $b\neq0$, we consider
\begin{eqnarray*}
(I) c>0,d_{3}>0;(II)c<0,d_{3}<0;(III)c>0,d_{3}<0;(IV)c<0,d_{3}>0;(V)cd_{3}=0.
\end{eqnarray*}
Case $(I) c>0,d_{3}>0.$ By using Lemma 1  of \cite{BLN2021},  we have
\begin{eqnarray}
\left|b\int_{c}^{d_{3}}e^{it\rho}\rho^{-1+ib}d\rho\right|\leq C(1+|b|)^{2}.\label{3.063}
\end{eqnarray}
Case $(II)c<0,d_{3}<0.$  By using Lemma 1  of \cite{BLN2021},
we have
\begin{eqnarray}
&&\left|b\int_{c}^{d_{3}}e^{-it\rho}\rho^{-1+ib}d\rho\right|\ \underline{\underline{\rho=-s}}\ \
e^{-\pi b}\left|b\int_{-c}^{-d_{3}}e^{its}s^{-1+ib}ds\right|\nonumber\\&&
\leq Ce^{-\pi b}(1+|b|)^{2}.\label{3.064}
\end{eqnarray}
Case $(III)c>0,d_{3}<0.$ By using Lemma 3.18,
we have
\begin{eqnarray}
&&\left|b\int_{c}^{d_{3}}e^{-it\rho}\rho^{-1+ib}d\rho\right|=
\left|b\int_{0}^{d_{3}}e^{-it\rho}\rho^{-1+ib}d\rho-b\int_{0}^{c}e^{-it\rho}\rho^{-1+ib}d\rho\right|\nonumber\\
&&\leq \left|b\int_{0}^{d_{3}}e^{-it\rho}\rho^{-1+ib}d\rho\right|+\left|b\int_{0}^{c}e^{-it\rho}\rho^{-1+ib}d\rho\right|\nonumber\\
&&\leq C(1+e^{-\pi b})(1+|b|)^{2}.\label{3.065}
\end{eqnarray}
Case $(IV)c<0,d_{3}>0.$ By using a proof similar to  Case $(III)c>0,d_{3}<0,$ we have
\begin{eqnarray}
&&\left|b\int_{c}^{d_{3}}e^{-it\rho}\rho^{-1+ib}d\rho\right|
\leq C(1+e^{-\pi b})(1+|b|)^{2}.\label{3.066}
\end{eqnarray}
Case $(V)cd_{3}=0.$ When $c=0,d_{3}=0,$ the conclusion is valid. When $c=0,d_{3}\neq0$ or
$c\neq0,d_{3}=0$, by using Lemma 3.18,  we
have
\begin{eqnarray*}
&&\left|b\int_{c}^{d_{3}}e^{-it\rho}\rho^{-1+ib}d\rho\right|
\leq C(1+e^{-\pi b})(1+|b|)^{2}.
\end{eqnarray*}

This completes the proof of Lemma 3.19.

\noindent {\bf Remark 5.}In  Lemma 1 of  \cite{BLN2021},    Bez et al.  proved that
(\ref{3.063}) is valid for $c>0,d_{3}>0.$ Thus, Lemma 3.16 extends the result of Lemma 1 of \cite{BLN2021}.

\begin{Lemma}\label{3.20}
Let $c,d_{3} \in \R,b\neq0.$ Then, we have
\begin{eqnarray*}
\left|\int_{c}^{d_{3}}e^{-it\rho}\rho^{-1+ib}d\rho\right|\leq CC_{4}(b).
\end{eqnarray*}
Here \bigskip

$\qquad C_{4}(b)=\left\{
        \begin{array}{ll}
        \left(\sqrt{|b|}+\frac{1}{\sqrt{|b|}}\right)^{2}, & \hbox{$b\neq0,c>0,d_{3}>0$;} \\
          e^{-\pi b}\left(\sqrt{|b|}+\frac{1}{\sqrt{|b|}}\right)^{2},
           & \hbox{$b\neq0,c<0,d_{3}<0$},\\
            (1+e^{-\pi b})\left(\sqrt{|b|}+\frac{1}{\sqrt{|b|}}\right)^{2}, & \hbox{$b\neq0,cd_{3}<0$,}\\
            (1+e^{-\pi b})\left(\sqrt{|b|}+\frac{1}{\sqrt{|b|}}\right)^{2}, & \hbox{$b\neq0,cd_{3}=0$.}
        \end{array}
      \right.$

\bigskip

Here $C$ is independent of $c,b,d_{3},t$.
\end{Lemma}
\noindent{\bf Proof.}From the process of proof of Lemma 3.19, we know that Lemma 3.20 is valid.

This completes the proof of Lemma 3.20.

\begin{Lemma}\label{3.21}
Let $b\neq 0$ and $\epsilon>0,M>0$.Then, we have
\begin{eqnarray*}
\left|\int_{\epsilon}^{M}e^{-it\rho}\rho^{-1+ib}d\rho+
\int_{-M}^{-\epsilon}e^{-it\rho}\rho^{-1+ib}d\rho\right|
\leq C(1+e^{-\pi b})\left(\sqrt{|b|}+\frac{1}{\sqrt{|b|}}\right)^{2}.
\end{eqnarray*}

Here $C$ is independent of $b,t$.
\end{Lemma}
\noindent{\bf Proof.} By using the triangle inequality and
  Lemma 3.20, we have
\begin{eqnarray*}
&&\left|\int_{\epsilon}^{M}e^{-it\rho}\rho^{-1+ib}d\rho+
\int_{-M}^{-\epsilon}e^{-it\rho}\rho^{-1+ib}d\rho\right|\nonumber\\
&&\leq \left|\int_{\epsilon}^{M}e^{-it\rho}\rho^{-1+ib}d\rho\right|+
\left|\int_{-M}^{-\epsilon}e^{-it\rho}\rho^{-1+ib}d\rho\right|\nonumber\\
&&\leq C(1+e^{-\pi b})\left(\sqrt{|b|}+\frac{1}{\sqrt{|b|}}\right)^{2}.
\end{eqnarray*}

This completes the proof of Lemma 3.21.

\begin{Lemma}\label{3.22}
Let $\psi(\eta,t)=\frac{(2\eta^{2}+3t)\eta}{(\eta^{2}+t)^{\frac{3}{2}}}$,
$0<t\leq \delta,\delta \leq 1$. Then, we have
\begin{align}
\left|\int_{|\eta|\geq 8t^{\frac{1}{2}}}e^{iA\eta+i\eta\sqrt{t+\eta^{2}}}
\left|\psi(\eta,t)\right|^{\frac{1}{2}+i\beta}d\eta\right|\leq C.\label{3.067}
\end{align}
Here $C$ is a constant which is independent of $\eta, t$.
\end{Lemma}

\begin{proof} If $|A|\leq 2$, then, we have
\begin{align}
&\int_{|\eta|\geq 8t^{\frac{1}{2}}}e^{iA\eta+i\eta\sqrt{t+\eta^{2}}}
\left|\psi(\eta,t)\right|^{\frac{1}{2}+i\beta}d\eta\notag\\
&=\int_{8t^{\frac{1}{2}}\leq|\eta|\leq 8}e^{iA\eta+i\eta\sqrt{t+\eta^{2}}}
\left|\psi(\eta,t)\right|^{\frac{1}{2}+i\beta}d\eta+\int_{|\eta|\geq 8}e^{iA\eta+i\eta\sqrt{t+\eta^{2}}}
\left|\psi(\eta,t)\right|^{\frac{1}{2}+i\beta}d\eta\notag\\
&=I_{1}+I_{2}.\label{3.068}
\end{align}
Obviously, $|I_{1}|\leq C$. By using the integration by parts, we derive that $|I_{2}|\leq C$.

If $|A|> 2$, then, we have
\begin{align}
&\int_{|\eta|\geq 8t^{\frac{1}{2}}}e^{iA\eta+i\eta\sqrt{t+\eta^{2}}}
\left|\psi(\eta,t)\right|^{\frac{1}{2}+i\beta}d\eta\notag\\
&=\int_{8t^{\frac{1}{2}}\leq|\eta|\leq 1}e^{iA\eta+i\eta\sqrt{t+\eta^{2}}}
\left|\psi(\eta,t)\right|^{\frac{1}{2}+i\beta}d\eta+\int_{1\leq|\eta|\leq
\frac{\sqrt{2}}{4}|A|}e^{iA\eta+i\eta\sqrt{t+\eta^{2}}}\left|\psi(\eta,t)\right|^{\frac{1}{2}+i\beta}d\eta\notag\\
&+\int_{\frac{\sqrt{2}}{4}|A|\leq|\eta|\leq \frac{3}{4}|A|}
e^{iA\eta+i\eta\sqrt{t+\eta^{2}}}\left|\psi(\eta,t)\right|^{\frac{1}{2}+i\beta}d\eta+
\int_{|\eta|\geq \frac{3}{4}|A|}e^{iA\eta+i\eta\sqrt{t+\eta^{2}}}\left|\psi(\eta,t)\right|^{\frac{1}{2}+i\beta}d\eta
\notag\\
&=\sum_{k=1}^{4}J_{k}.\label{3.069}
\end{align}
Obviously, $|J_{1}|\leq C$. By using the integration by parts, we derive
 that $|J_{2}|\leq C$ and $|J_{4}|\leq C$.
Since
\begin{align}
&\left|(A\eta+\eta\sqrt{t+\eta^{2}})^{\prime\prime}\right|=
\left|(A+\sqrt{t+\eta^{2}}
+\frac{\eta^{2}}{\sqrt{t+\eta^{2}}})^{\prime}\right|\notag\\
&=\left|\frac{\eta}{\sqrt{t+\eta^{2}}}+
\frac{2\eta(t+\eta^{2})-\eta^{3}}{(t+\eta^{2})^{\frac{3}{2}}}\right|
=\left|\frac{2\eta^{3}+3t\eta}{(t+\eta^{2})^{\frac{3}{2}}}\right|
\sim C,\label{3.070}
\end{align}
by using  Lemma 2.11, we have that $|J_{3}|\leq C$.
\end{proof}

This completes the proof of Lemma 3.22.

\begin{Lemma} \label{3.23} Let $0<t\leq\delta,\delta \leq 1$. Then, we have that
\begin{align}
\left|\int_{|\xi|\geq 8}e^{ix\xi+it\xi\sqrt{1+\xi^{2}}}
\left|\phi^{\prime\prime}(\xi)\right|^{\frac{1}{2}+i\beta}d\xi\right|
\leq Ct^{-\frac{1}{2}}.\label{3.071}
\end{align}
Here $\phi(\xi)=\xi\sqrt{1+\xi^{2}}$.
\end{Lemma}
\noindent{\bf Proof.}
 Let $t^{\frac{1}{2}}\xi=\eta$, then, we have
\begin{align}
\int_{|\xi|\geq 8}e^{ix\xi+it\xi\sqrt{1+\xi^{2}}}
\left|\phi^{\prime\prime}(\xi)\right|^{\frac{1}{2}+i\beta}d\xi=
\left(\int_{|\eta|\geq 8t^{\frac{1}{2}}}
e^{iA\eta+i\eta\sqrt{t+\eta^{2}}}\left|\psi(\eta,t)\right|^{\frac{1}{2}+i\beta}d\eta\right)
t^{-\frac{1}{2}}.\label{3.072}
\end{align}
Here $A=xt^{-1/2}$ and $\psi(\eta,t)$ is defined as in Lemma 3.22.
It follows from the above identity and  Lemma 3.22  that
\begin{align}
\left|\int_{|\xi|\geq 8}e^{ix\xi+it\xi\sqrt{1+\xi^{2}}}
\left|\phi^{\prime\prime}(\xi)\right|^{\frac{1}{2}+i\beta}d\xi\right|
\leq Ct^{-\frac{1}{2}}.\label{3.073}
\end{align}

This completes the proof of Lemma 3.23.

\begin{Lemma} \label{3.24} Let $0<t\leq\delta,\delta \leq 1$. Then, we have that
\begin{align}
\left|\int_{\SR}e^{ix\xi\pm it\xi\sqrt{1+\xi^{2}}}
\left|\phi^{\prime\prime}(\xi)\right|^{\frac{1}{2}+i\beta}d\xi\right|
\leq Ct^{-\frac{1}{2}}.\label{3.074}
\end{align}
Here $\phi(\xi)=\xi\sqrt{1+\xi^{2}}$.
\end{Lemma}
\noindent{\bf Proof.}  Now we bound
 $$\int_{\SR}e^{ix\xi+ it\xi\sqrt{1+\xi^{2}}}
\left|\phi^{\prime\prime}(\xi)\right|^{\frac{1}{2}+i\beta}d\xi.$$
Obviously,
\begin{align}
&\int_{\SR}e^{ix\xi+ it\xi\sqrt{1+\xi^{2}}}
\left|\phi^{\prime\prime}(\xi)\right|^{\frac{1}{2}+i\beta}d\xi\notag\\
&=\int_{|\xi|\leq8}e^{ix\xi+ it\xi\sqrt{1+\xi^{2}}}
\left|\phi^{\prime\prime}(\xi)\right|^{\frac{1}{2}+i\beta}d\xi
+\int_{|\xi|\geq8}e^{ix\xi+ it\xi\sqrt{1+\xi^{2}}}
\left|\phi^{\prime\prime}(\xi)\right|^{\frac{1}{2}+i\beta}d\xi.\label{3.075}
\end{align}
Obviously,  $\left|\int_{|\xi|\leq8}e^{ix\xi+ it\xi\sqrt{1+\xi^{2}}}
\left|\phi^{\prime\prime}(\xi)\right|^{\frac{1}{2}+i\beta}d\xi\right|\leq C$.
 From Lemma 3.23, we have
\begin{align}
\left|\int_{|\xi|\geq 8}e^{ix\xi+it\xi\sqrt{1+\xi^{2}}}
\left|\phi^{\prime\prime}(\xi)\right|^{\frac{1}{2}+i\beta}
d\xi\right|\leq Ct^{-\frac{1}{2}}.\label{3.076}
\end{align}
Thus, we have that $\left|\int_{\SR}e^{ix\xi+it\xi\sqrt{1+\xi^{2}}}
\left|\phi^{\prime\prime}(\xi)\right|^{\frac{1}{2}+i\beta}d\xi\right|\leq Ct^{-\frac{1}{2}}.$
By using a proof  similar to $$\left|\int_{\SR}e^{ix\xi+it\xi\sqrt{1+\xi^{2}}}
\left|\phi^{\prime\prime}(\xi)\right|^{\frac{1}{2}+i\beta}d\xi\right|\leq Ct^{-\frac{1}{2}},$$
we derive that
$$\left|\int_{\SR}e^{ix\xi-it\xi\sqrt{1+\xi^{2}}}
\left|\phi^{\prime\prime}(\xi)\right|^{\frac{1}{2}+i\beta}d\xi\right|\leq Ct^{-\frac{1}{2}}.$$

This completes the proof of Lemma 3.24.

\begin{Lemma} \label{3.25} Let $0<t\leq\delta,\delta \leq 1$.
Then, we have that
\begin{align}
\left|\int_{\SR}e^{ix\xi\pm it\xi\sqrt{1+\xi^{2}}}d\xi\right|
\leq Ct^{-\frac{1}{2}}.\label{3.077}
\end{align}
\end{Lemma}
\noindent {\bf Proof.} By using a proof similar to Lemma 3.24,
 we have that Lemma 3.25 is valid.

This completes the proof of Lemma 3.25.

\begin{Lemma} \label{3.26} Let $0<t\leq\delta,\delta \leq 1$
 and $\frac{2}{p}+\frac{1}{q}
=\frac{1}{2}$. Then, we have that
\begin{align}
\left\|e^{it\mp\partial_{x}\sqrt{1-\partial_{x}^{2}}}u_{0}
\right\|_{L_{t}^{p}L_{x}^{q}}\leq C\|u_{0}\|_{L^{2}}.\label{3.078}
\end{align}
\end{Lemma}
\noindent {\bf Proof.}Combining $TT^*$ method \cite{C} with (\ref{3.077}),
 we have that (\ref{3.078}) is valid.

This completes the proof of Lemma 3.26.

\begin{Lemma} \label{3.27} Let $t>1$ and $A\in \R$. Then, we have that
\begin{align}
\left|\int_{\SR}e^{it(A\xi+\xi\sqrt{1+\xi^{2}})}
\left|\phi^{\prime\prime}(\xi)\right|^{\frac{1}{2}+i\beta}d\xi\right|
\leq Ct^{-\frac{1}{3}}.\label{3.079}
\end{align}
Here $\phi(\xi)=\xi\sqrt{1+\xi^{2}}$.
\end{Lemma}
\noindent{\bf Proof.}We consider $A\geq0,A<0$, respectively.  Firstly,
 we consider $A\geq0.$
Now we bound
\begin{eqnarray}
\int_{\SR}e^{it(A\xi+\xi\sqrt{1+\xi^{2}})}
\left|\phi^{\prime\prime}(\xi)\right|^{\frac{1}{2}+i\beta}d\xi.\label{3.080}
\end{eqnarray}
By using  Lemma 2.11, since
$(A\xi+\xi\sqrt{1+\xi^{2}})^{\prime\prime\prime}=
\frac{3}{(1+\xi^{2})^{5/2}}\geq \frac{3}{4^{10}}(|\xi|\leq8),$
we have
\begin{eqnarray}
\left|\int_{|\xi|\leq8}e^{it(A\xi+ \xi\sqrt{1+\xi^{2}}}
\left|\phi^{\prime\prime}(\xi)\right|^{\frac{1}{2}+i\beta}d\xi\right|
\leq Ct^{-1/3}.\label{3.081}
\end{eqnarray}
By using integration by parts, since $t>1$ and  $A\geq0,$
   we have
\begin{eqnarray}
\left|\int_{|\xi|>8}e^{it(A\xi+ \xi\sqrt{1+\xi^{2}})}
\left|\phi^{\prime\prime}(\xi)\right|^{\frac{1}{2}+i\beta}d\xi\right|
\leq Ct^{-1}\leq Ct^{-1/3}.\label{3.082}
\end{eqnarray}
Now we consider $A<0.$
Obviously, $\xi_{0}=\pm\sqrt{\frac{A^{2}-4+\sqrt{A^{4}+8A^{2}}}{8}}$
 is the solution to
$A+\frac{2\xi^{2}+1}{\sqrt{\xi^{2}+1}}=0$.
When $|\xi_{0}|\leq2,$ by using  Lemma 2.11, we have
\begin{eqnarray}
\left|\int_{|\xi|\leq8}e^{it(A\xi+ \xi\sqrt{1+\xi^{2}})}
\left|\phi^{\prime\prime}(\xi)\right|^{\frac{1}{2}+i\beta}d\xi\right|
\leq Ct^{-1/3}.\label{3.083}
\end{eqnarray}
By using integration by parts, since $t>1$,  we have
\begin{eqnarray}
\left|\int_{|\xi|>8}e^{it(A\xi+ \xi\sqrt{1+\xi^{2}})}
\left|\phi^{\prime\prime}(\xi)\right|^{\frac{1}{2}+i\beta}d\xi\right|
\leq Ct^{-1}\leq Ct^{-1/3}.\label{3.084}
\end{eqnarray}
When $|\xi_{0}|\geq2,$ we have
\begin{eqnarray}
\int_{\SR}e^{it(A\xi+\xi\sqrt{1+\xi^{2}})}
\left|\phi^{\prime\prime}(\xi)\right|^{\frac{1}{2}+i\beta}d\xi=
\sum\limits_{j=1}^{3}I_{j},\label{3.085}
\end{eqnarray}
where
\begin{eqnarray*}
&&
I_{1}=\int_{ |\xi|\leq\frac{|\xi_{0}|}{2}}e^{it(A\xi+\xi\sqrt{1+\xi^{2}})}
\left|\phi^{\prime\prime}(\xi)\right|^{\frac{1}{2}+i\beta}d\xi,\nonumber\\
&&
I_{2}=\int_{\frac{|\xi_{0}|}{2}\leq |\xi|\leq2|\xi_{0}|}e^{it(A\xi+\xi\sqrt{1+\xi^{2}})}
\left|\phi^{\prime\prime}(\xi)\right|^{\frac{1}{2}+i\beta}d\xi,\nonumber\\
&&
I_{3}=\int_{|\xi|\geq2|\xi_{0}|}e^{it(A\xi+\xi\sqrt{1+\xi^{2}})}
\left|\phi^{\prime\prime}(\xi)\right|^{\frac{1}{2}+i\beta}d\xi.
\end{eqnarray*}
For $I_{1},I_{3},$ by using integration by parts, we have
\begin{eqnarray}
|I_1|\leq Ct^{-1}\leq Ct^{-1/3},|I_3|\leq Ct^{-1}\leq Ct^{-1/3}\label{3.086}.
\end{eqnarray}
For $I_{2}$, by using  Lemma 2.11, since
\begin{eqnarray}
&&(A\xi+\xi\sqrt{1+\xi^{2}})^{\prime\prime\prime}=
\frac{3}{(1+\xi^{2})^{5/2}}>0\label{3.087},\\
&&\left|(A\xi+\xi\sqrt{1+\xi^{2}})^{\prime\prime}\right|=
\left|\frac{(2\xi^{2}+3)\xi}{(1+\xi^{2})^{3/2}}\right|\geq1\label{3.088}
\end{eqnarray}
and  $t>1$,
we have
\begin{eqnarray}
|I_{2}|\leq Ct^{-\frac{1}{2}}\leq Ct^{-\frac{1}{3}}.\label{3.089}
\end{eqnarray}

This completes the proof of Lemma 3.27.

\begin{Lemma} \label{3.28} Let $t\geq1$. Then, we have that
\begin{align}
\left|\int_{\SR}e^{ix\xi\pm it\xi\sqrt{1+\xi^{2}}}d\xi\right|
\leq Ct^{-\frac{1}{3}}.\label{3.090}
\end{align}
\end{Lemma}
\noindent {\bf Proof.} By using a proof similar to Lemma 3.27,
we have that Lemma 3.28 is valid.

This completes the proof of Lemma 3.28.

\begin{Lemma} \label{3.29} Let $t\geq1$ and $\frac{3}{p}+\frac{1}{q}
=\frac{1}{2}$. Then, we have that
\begin{align}
\left\|e^{it\mp\partial_{x}\sqrt{1-\partial_{x}^{2}}}u_{0}\right\|_{L_{t}^{p}L_{x}^{q}}
\leq C\|u_{0}\|_{L^{2}}.\label{3.091}
\end{align}
\end{Lemma}
\noindent {\bf Proof.}Combining $TT^{*}$ method \cite{C} with (\ref{3.090}),
we have that (\ref{3.091}) is valid.

This completes the proof of Lemma 3.29.

\bigskip
\bigskip

\section{Strichartz estimates for orthonormal functions  and Schatten bound with space-time norms}
\medskip

\setcounter{equation}{0}

\setcounter{Theorem}{0}

\setcounter{Lemma}{0}

\setcounter{section}{4}

In this section, by using Lemmas 2.10, 3.8-3.10, 3.12-3.17, 3.20-3.21, we establish
some Strichartz estimates for orthonormal functions
 and Schatten bound with space-time norms
related to  elliptic operator and non-elliptic operator  on $\R^{d},\T^{d}$, for the details,
 we refer the readers to Theorems 4.1-4.7 established in this paper which extends
Theorems  8, 9 of \cite{FS2017} and Proposition 3.3 and  Theorem 1.5 of \cite{N}, respectively.
By using Lemmas 2.10, 3.8-3.10, 3.12-3.17, 3.20-3.21, 3.25, 3.28,  we establish the
Schatten bound with space-time norms and some Strichartz
estimates for orthonormal functions related to Boussinesq operator on $\R$, for the details,
 we refer the readers to Theorems 4.8-4.11.

\begin{Theorem}(Schatten bound with space-time norms related to elliptic operator on $\R^{d}$)\label{4.1}
Let $d\geq1$ and   $p^{\prime}=\frac{p}{p-1},q^{\prime}=\frac{q}{q-1}\geq1$ satisfy
$$\frac{1}{p^{\prime}}+\frac{d}{2q^{\prime}}=1,\ \ q^{\prime}>\frac{d+1}{2}.$$ Then,
 we have the Schatten bound
\begin{eqnarray}
&&\left\|W_{1}U(t)U^{*}(t)W_{2}\right\|_{\mathfrak{S}^{2q^{\prime}}}\leq
C\left\|W_{1}\right\|_{L_{t}^{2p^{\prime}}L_{x}^{2q^{\prime}}(\mathbb{I}\times\mathbf{R}^{d})}
\left\|W_{2}\right\|_{L_{t}^{2p^{\prime}}L_{x}^{2q^{\prime}}(\mathbb{I}\times\mathbf{R}^{d})},\label{4.01}
\end{eqnarray}
with $C>0$ independent of $W_{1},W_{2}$ and $U(t)f=e^{it\Delta}f$. Here
$\mathbb{I}=[a_{1},b_{1}]$ or $\mathbb{I}=[a_{1},b_{1}]\bigcup[c,d_{3}]$ or $\mathbb{I}=\R$ or
 $\mathbb{I}=[a_{1},+\infty)$ or  $\mathbb{I}=(-\infty,b_{1}]$
  or  $\mathbb{I}=[a_{1},+\infty)\bigcup(-\infty,b_{1}].$
  Here $a_{1},b_{1},c,d_{3}\in \R.$
\end{Theorem}
\noindent {\bf Proof.}Inspired by \cite{N,Vega1992}, we present the proof of Theorem 4.1.
Since
\begin{align}
\left\langle U(t) f, g\right\rangle & = \int_{\mathbb{I}\times\mathbb{R}^{d}} e^{it\Delta}f\overline{g}dxdt=
\int_{\SR^{d}} f(x)\left(\overline{\int_{\mathbb{I}}e^{-it\Delta}gdt}\right)dx   \notag\\
&=\left\langle f, \int_{\mathbb{I}} e^{-it\Delta}gdt\right\rangle =\left\langle f, U^{*}(t) g\right\rangle,\label{4.02}
\end{align}
we have
\begin{align}
U^{*}(t)g=\int_{\mathbb{I}} e^{-it\Delta}gdt.\label{4.03}
\end{align}
We define
\begin{align}
&K(t,x)=\frac{1}{(2\pi)^{\frac{d}{2}}}\int_{\SR^{d}} e^{-it|\xi|^{2}}e^{ix\cdot\xi}d\xi\nonumber\\&
=\mathscr{F}_{x}^{-1}(e^{-it|\xi|^{2}})=\frac{1}{(2\pi)^{\frac{d}{2}}}\int_{\SR^{d+1}}e^{ix\cdot\xi}e^{it \tau}
\delta(\tau-|\xi|^{2})d\tau d\xi ,\label{4.04}
\end{align}
then
\begin{align}
U(t)U^{*}(t)g &=e^{it\Delta}\int_{\mathbb{I}} e^{-is\Delta}gds=\int_{\mathbb{I}} e^{i(t-s)\Delta}gds\notag\\
&=\frac{1}{(2\pi)^{\frac{d}{2}}}\int_{\mathbb{I}}\int_{\SR^{d}} e^{-i(t-s)|\xi|^{2}}e^{ix\cdot\xi}
\mathscr{F}_{x}{g}(\xi,s)d\xi ds \notag\\
&=\int_{\mathbb{I}}\int_{\SR^{d}} K(t-s,x-y)g(s,y)dyds.\label{4.05}
\end{align}
Here $ K(t,x)=\int_{\SR^{d}}e^{-it|\xi|^{2}}e^{ix\cdot\xi}d\xi.$
We define
\begin{eqnarray}
&&K_{z}(t,x)=(z+1)t^{z}K(t,x), K_{z,\epsilon}(t,x)=(z+1)t^{z}\chi_{\{t||t|>\epsilon\}\cap \mathbb{I}}(t)K(t,x)\nonumber\\
&&T_{z}f=K_{z}\ast f,
T_{z,\epsilon}f=K_{z,\epsilon}\ast f,
G(s,\xi)=e^{is|\xi|^2}\mathscr{F}_{x}f(\xi,s).\label{4.06}
\end{eqnarray}
Then, we have
\begin{eqnarray}
&&K_{z}(t,x)=(z+1)t^{z}K(t,x)=\lim\limits_{\epsilon\longrightarrow 0}K_{z,\epsilon}(t,x)=
\lim\limits_{\epsilon\longrightarrow 0}(z+1)t^{z}\chi_{\{t||t|>\epsilon\}\cap \mathbb{I}}(t)K(t,x),\nonumber\\
&&T_{z}f=\lim\limits_{\epsilon\longrightarrow 0}T_{z,\epsilon}f=
\lim\limits_{\epsilon\longrightarrow 0}K_{z,\epsilon}\ast f.\label{4.07}
\end{eqnarray}
In particular, $T_{0}=U(t)U^{*}(t).$
Then, by using (\ref{4.07}) and the Plancherel identity, we have
\begin{align}
\left\|T_{z}f\right\|_{L_{x}^2}=\left\|\lim\limits_{\epsilon\longrightarrow 0}
T_{z,\epsilon}f\right\|_{L_{x}^2}
&=\left\|\lim\limits_{\epsilon\longrightarrow 0}K_{z,\epsilon}\ast f\right\|_{L_{x}^2}
=\left\|\lim\limits_{\epsilon\longrightarrow 0}\mathscr{F}_{x}K_{z,\epsilon}(t,\xi)
\ast \mathscr{F}_{x}f(\xi,s)\right\|_{L_{\xi}^2}\notag\\
& =\left\|\lim\limits_{\epsilon\longrightarrow 0}\int_{\mathbb{R}} e^{-it|\xi|^{2}}(z+1)t^{z}\chi_{\{t||t|>\epsilon\}
\cap \mathbb{I}}(t)\mathscr{F}_{x}f(\xi,s-t)dt\right\|_{L_{\xi}^2}\notag\\
&=\left\|\lim\limits_{\epsilon\longrightarrow 0}\int_{\mathbb{R}}(z+1) t^{z}\chi_{\{t||t|>\epsilon\}\cap
\mathbb{I}}(t)e^{i(s-t)|\xi|^{2}}\mathscr{F}_{x}f(\xi,s-t)dt\right\|_{L_{\xi}^2}.\label{4.08}
\end{align}
By using (\ref{4.08}), it follows that
\begin{align}
\left\|T_{z}f\right\|_{L_{s}^2L_{x}^2}=\left\|\lim\limits_{\epsilon\longrightarrow 0}
T_{z,\epsilon}f\right\|_{L_{s}^2L_{x}^2}
&=\left\|\lim\limits_{\epsilon\longrightarrow 0}\int_{\mathbb{R}}(z+1) t^{z}\chi_{\{t||t|>\epsilon\}\cap
\mathbb{I}}(t)G(\xi,s-t)dt\right\|_{L_{s}^2L_{\xi}^2}\notag\\
&=\left\|\mathscr{F}_{t}(\lim\limits_{\epsilon\longrightarrow 0}(z+1)t^{z}\chi_{\{t||t|>\epsilon\}\cap
\mathbb{I}}(t)\mathscr{F}_{t}G(\xi,\tau)\right\|_{L_{\tau}^2L_{\xi}^2}.\label{4.09}
\end{align}
In particular, when $z=-1+bi(b\in \R)$, by using Lemmas 3.8-3.10, 3.12-3.17, 3.20-3.21
 and Remark 4, we have
\begin{eqnarray}
\left|\mathscr{F}_{t}(i\lim\limits_{\epsilon\longrightarrow 0}bt^{-1+bi}\chi_{\{t||t|>\epsilon\}\cap
\mathbb{I}}(t)\right|\leq  CH(b).\label{4.010}
\end{eqnarray}
Here $H(b):=(e^{-\pi b}+e^{\pi b})\left(|b|+1\right)^{2}$.
Combining (\ref{4.09}) with (\ref{4.010}), by using the Plancherel identity with
 respect to $s,x$,  we have
\begin{eqnarray}
&&\left\|T_{-1+bi}f\right\|_{L_{s}^2L_{x}^2}\leq CH(b)\left\|\mathscr{F}_{t}G(\xi,\tau)\right\|_{L_{\tau}^2L_{\xi}^2}\nonumber\\&&
\leq CH(b)
\left\|\mathscr{F}_{t}G(\xi,\tau)\right\|_{L_{\tau}^2L_{\xi}^2}
\nonumber\\&&= CH(b)
\left\|G(\xi,s)\right\|_{L_{s}^2L_{\xi}^2}\nonumber\\&&
=CH(b)\left\|e^{is|\xi|^2}\mathscr{F}_{x}f(\xi,s)\right\|_{L_{s}^2L_{\xi}^2}
\leq CH(b)
\left\|\mathscr{F}_{x}f(\xi,s)\right\|_{L_{s}^2L_{\xi}^2}\nonumber\\&&
=CH(b)
\left\|f\right\|_{L_{s}^2L_{x}^2}.\label{4.011}
\end{eqnarray}
Thus, from  (\ref{4.011}),  we have
\begin{eqnarray}
&&\hspace{-1cm}\left\|W_{1}T_{-1+bi}W_{2}\right\|_{\mathfrak{S}^{\infty}}\leq CH(b)\prod\limits_{j=1}^{2}\|W_{j}\|_{L_{t}^{\infty}
L_{x}^{\infty}(\mathbb{I}\times\mathbf{R}^{d})}.\label{4.012}
\end{eqnarray}
By using the Hardy-Littlewood-Sobolev inequality,   for $Re(z) \in [-1,\frac{d}{2}]$,   since
$$ \left|K(t,x)=\int_{\SR^{d}}e^{-it|\xi|^{2}}e^{ix\cdot\xi}d\xi\right|\leq C|t|^{-d/2},$$   we have
\begin{eqnarray}
&&\hspace{-1cm}\left\|W_{1}T_{z}W_{2}\right\|_{\mathfrak{S}^{2}}\nonumber\\&&
=\left[\int_{\mathbb{I}}\int_{\mathbb{I}}\int_{\SR^{d}}\int_{\SR^{d}}
\left|W_{1}(x,t)\right|^{2}\left|K_{z}(t-\tau,x-y)\right|^{2}\left|W_{2}(y,\tau)\right|^{2}
dxdydtd\tau\right]^{1/2}\nonumber\\
&&\leq C\left[\int_{\mathbb{I}}\int_{\mathbb{I}}
\int_{\SR^{d}}\int_{\SR^{d}}\frac{\left|W_{1}(x,t)\right|^{2}
\left|W_{2}(y,\tau)\right|^{2}}{|t-\tau|^{d-2Rez}}dxdydtd\tau\right]^{1/2}\nonumber\\
&&\leq  C(1+|b|+|Rez|)  \left[\int_{\mathbb{I}}\int_{\mathbb{I}}
\frac{\|W_{1}(\cdot,t)\|_{L_{x}^{2}(\SR^{d})}^{2}
\|W_{2}(\cdot,\tau)\|_{L_{x}^{2}(\SR^{d})}^{2}}{|t-\tau|^{d-2Rez}}dtd\tau\right]^{1/2} \nonumber\\
&&\leq C(1+|b|+|Rez|)\prod_{j=1}^{2}\|W_{j}\|_{L_{t}^{2\tilde{u}}(\mathbb{I})L_{x}^{2}(\SR^{d})} \label{4.013},
\end{eqnarray}
where $0\leq d-2Rez<1, \frac{2}{\tilde{u}}+d-2Rez=2,Rez=\frac{1}{q-1}.$
By using  Theorem 2.9 of \cite{S}, interpolating (\ref{4.012}) with (\ref{4.013}), we derive
\begin{eqnarray}
&&\left\|W_{1}T_{0}W_{2}\right\|_{\mathfrak{S}^{2q^{\prime}}}
=\left\|W_{1}U(t)U^{*}(t)W_{2}\right\|_{\mathfrak{S}^{2q^{\prime}}}\nonumber\\&&\leq C
\prod_{j=1}^{2}\|W_{j}\|_{L_{t}^{2p^{\prime}}(\mathbb{I})L_{x}^{2q^{\prime}}(\SR^{d})},\label{4.014}
\end{eqnarray}
Here
\begin{eqnarray*}
&&\frac{1}{p^{\prime}}+\frac{d}{2q^{\prime}}=1,\frac{2}{d+2}
<\frac{1}{q^{\prime}}<\frac{2}{d+1}.
\end{eqnarray*}
This completes the proof of (\ref{4.01}).

This completes the proof of Theorem 4.1.

\noindent{\bf Remark 6.}Obviously, when $\mathbb{I}=\R$,
we have that $ (\mathscr{F}_{xt}K(x,t))(\tau,\xi)
=\delta(\tau-|\xi|^{2})=\frac{(\tau-|\xi|^{2})^{z}}{\Gamma(z+1)}|_{z=-1}.$
When $\mathbb{I}\neq \R$, $ (\mathscr{F}_{xt}K(x,t))(\tau,\xi)=\delta(\tau-|\xi|^{2})$
may not be valid. Thus,
in proving
Theorem 4.1, we do not completely follow the method of \cite{FS2017}
 since $\mathbb{I}=\R$ or  $\mathbb{I}\neq \R$,
 however,  by combining  Lemmas 3.8-3.20 with  the method of
 \cite{FS2017,Vega1992,N}, we prove Theorem 4.1.

\begin{Theorem}(Strichartz estimates for orthonormal functions related
 to elliptic operator on $\R^{d}$)\label{4.2}
Let  $d\geq1$ and  $p,q\geq1$  satisfy
$$\frac{2}{p}+\frac{d}{q}=d,\ \ 1\leq q<1+\frac{2}{d-1}.$$
Then, for any (possibly infinite) orthonormal system $(f_{j})_{j=1}^{+\infty}$ in $L^{2}(\mathbb{R}^{d})$
 and for all sequence $(\nu_{j})_{j=1}^{+\infty}\subset\mathbf{C}$, we derive
\begin{eqnarray}
&&\left\|\sum\limits_{j=1}^{+\infty}\nu_{j}\left|U(t)f_{j}\right|^{2}\right\|_{L_{t}^{p}L_{x}^{q}
(\mathbb{I}\times\mathbf{R}^{d})}\leq C\left(\sum\limits_{j}
\left|\nu_{j}\right|^{\frac{2q}{q+1}}\right)^{\frac{q+1}{2q}}\label{4.015}.
\end{eqnarray}
with $C>0$ independent of $(\nu_{j})_{j=1}^{+\infty}$ and $(f_{j})_{j=1}^{+\infty}$.
Here
$U(t)f=e^{it\Delta}f$ and
 $\mathbb{I}=[a_{1},b_{1}]$ or $\mathbb{I}=[a_{1},b_{1}]\bigcup[c,d_{3}]$ or $\mathbb{I}=\R$ or
  $\mathbb{I}=[a_{1},+\infty)$ or  $\mathbb{I}=(-\infty,b_{1}]$ or
  $\mathbb{I}=[a_{1},+\infty)\bigcup(-\infty,b_{1}].$ Here $a_{1},b_{1},c,d_{3}\in \R.$

\end{Theorem}
\noindent{\bf Proof.} By using Lemma 2.10 and Theorem 4.1,
 it follows that Theorem 4.2 is valid.

This completes the proof of Theorem 4.2.

\begin{Theorem}(Schatten bound with space-time norms  related to
  non-elliptic  operator on $\R^{d}$)\label{4.3}
Let $\Delta_{\pm}=\sum\limits_{j=1}^{k}\frac{\partial^{2}}{\partial_{x_{j}}^{2}}
-\sum\limits_{j=k+1}^{d}\frac{\partial^{2}}{\partial_{x_{j}}^{2}}(1\leq k\leq d-1$,
 $\Delta_{\pm}=\Delta(k=d)$, $d\geq2$
and $p^{\prime}=\frac{p}{p-1},q^{\prime}=\frac{q}{q-1}\geq 1$ satisfy
$$\frac{1}{p^{\prime}}+\frac{d}{2q^{\prime}}=1,
\ \ q^{\prime}>\frac{d+1}{2}.$$ 
Then, we have the Schatten bound
\begin{eqnarray}
&&\left\|W_{1}U_{\pm}(t)U^{*}_{\pm}(t)W_{2}\right\|_{\mathfrak{S}^{2q^{\prime}}}\leq C
\left\|W_{1}\right\|_{L_{t}^{2p^{\prime}}L_{x}^{2q^{\prime}}(\mathbb{I}\times\mathbf{R}^{d})}
\left\|W_{2}\right\|_{L_{t}^{2p^{\prime}}L_{x}^{2q^{\prime}}(\mathbb{I}\times\mathbf{R}^{d})},\label{4.016}
\end{eqnarray}
with $C>0$ independent of $W_{1},W_{2}$. Here
$U_{\pm}(t)f:=e^{it\Delta_{\pm}}f=\frac{1}{(2\pi)^{\frac{d}{2}}}
\int_{\SR^{d}}e^{it|\xi|_{\pm}^{2}}e^{ix\cdot\xi}d\xi,$
$|\xi|_{\pm}^{2}=\sum\limits_{j=1}^{k}\xi_{j}^{2}-\sum\limits_{j=k+1}^{d}\xi_{j}^{2},
1\leq k\leq d-1,|\xi|^{2}=\sum\limits_{j=1}^{d}\xi_{j}^{2},$
$\mathbb{I}=[a_{1},b_{1}]$ or $\mathbb{I}=[a_{1},b_{1}]\bigcup[c,d_{3}]$
 or $\mathbb{I}=\R$ or
$\mathbb{I}=[a_{1},+\infty)$ or  $\mathbb{I}=(-\infty,b_{1}]$ or
$\mathbb{I}=[a_{1},+\infty)\bigcup(-\infty,b_{1}],$
  $a_{1},b_{1},c,d_{3}\in \R.$
\end{Theorem}
\noindent {\bf Proof.}
Since
\begin{eqnarray}
&&\left|K_{\pm}(t,x)=\int_{\SR^{d}}e^{-it|\xi|_{\pm}^{2}}e^{ix\cdot\xi}d\xi\right|\nonumber\\&&=
\left|\left(\prod\limits_{j=1}^{k}\int_{\SR}e^{-it\xi_{j}}e^{ix_{j}\xi_{j}}d\xi_{j}\right)
\left(\prod\limits_{j=k+1}^{d}\int_{\SR}e^{-it\xi_{j}}e^{ix_{j}\xi_{j}}d\xi_{j}\right)\right|
\leq C|t|^{-d/2}\label{4.017},
\end{eqnarray}
by using a proof similar to Theorem  4.1, we can obtain that Theorem 4.3 is valid.

This completes the proof of Theorem 4.3.

\begin{Theorem}(Strichartz estimates for orthonormal functions
 related to non-elliptic operator on $\R^{d}$)\label{4.4}
Let $\Delta_{\pm}$ and $U_{\pm}(t)$ be defined as in Lemma 4.3, $d\geq2$ and  $p,q\geq1$  satisfy
$$\frac{2}{p}+\frac{d}{q}=d,\ \ 1\leq q<1+\frac{2}{d-1}.$$
Then, for any (possibly infinite) orthonormal system $(f_{j})$ in $L^{2}(\mathbf{R}^{d})$ and for
all sequence $(\nu_{j})_{j=1}^{+\infty}\subset\mathbf{C}$, we derive
\begin{eqnarray}
&&\left\|\sum\limits_{j=1}^{+\infty}\nu_{j}\left|U_{\pm}(t)f_{j}\right|^{2}\right\|_{L_{t}^{p}L_{x}^{q}
(\mathbb{I}\times\mathbf{R}^{d})}\leq C
\left(\sum\limits_{j}\left|\nu_{j}\right|^{\frac{2q}{q+1}}\right)^{\frac{q+1}{2q}}\label{4.018}
\end{eqnarray}
with $C>0$ independent of $(\nu_{j})$ and $(f_{j})$. Here
$\mathbb{I}=[a_{1},b_{1}]$
 or $\mathbb{I}=[a_{1},b_{1}]\bigcup[c,d_{3}]$ or $\mathbb{I}=\R$ or
$\mathbb{I}=[a_{1},+\infty)$ or  $\mathbb{I}=(-\infty,b_{1}]$ or
$\mathbb{I}=[a_{1},+\infty)\bigcup(-\infty,b_{1}].$ Here $a_{1},b_{1},c,d_{3}\in \R.$
\end{Theorem}
\noindent{\bf Proof.} By using Lemma 2.10 and Theorem 4.3,
we derive that Theorem 4.4 is valid.

This completes the proof of Theorem 4.4.

\begin{Theorem} \label{4.5}(Schatten bound with space-time norms related to
non-elliptic operator on $\T^{d}$)
Let $\mathbb{T}=[0,2\pi)$, $S_{d,N}=Z^{d}\cap [-N,N]^{d}$,
$I_{N}=\left[-\frac{1}{2N},\frac{1}{2N}\right].$ Then,
for all $W_{1},W_{2} \in L_{t}^{2p^{\prime}}L_{x}^{2q^{\prime}}(I_{N}\times \mathbb{T}^{d})$, we have
\begin{eqnarray}
\left\|W_{1}\mathscr{E}_{N}\mathscr{E}_{N}^{*}W_{2}\right\|_{\mathfrak{S}^{2q^{\prime}}
(L^{2}(I_{N}\times \mathbf{T}^{d}))}\leq C\left\|W_{1}\right\|_{L_{t}^{2p^{\prime}}L_{x}^{2q^{\prime}}(I_{N}\times \mathbf{T}^{d})}
\left\|W_{2}\right\|_{L_{t}^{2p^{\prime}}L_{x}^{2q^{\prime}}(I_{N}\times \mathbf{T}^{d})}.\label{4.020}
\end{eqnarray}
Here $\frac{1}{p^{\prime}}+\frac{d}{2q^{\prime}}=1$,
$q^{\prime}>\frac{d+1}{2}$ and $
\mathscr{E}_{N}f_{j}=\frac{1}{(2\pi)^{d}}\sum\limits_{n\in S_{d,N}}
\mathscr{F}_{x}{f}_{j}(n)e^{i(x\cdot n+ t|n|_{\pm}^{2})}.$
\end{Theorem}
\noindent{\bf Proof.}
Since
\begin{align}
\mathscr{E}_{N}f_{j}=\frac{1}{(2\pi)^{d}}\sum_{n\in S_{d,N}}
\mathscr{F}_{x}{f}_{j}(n)e^{i(x\cdot n+ t|n|_{\pm}^{2})},\label{4.021}
\end{align}
then
\begin{align}
\left\langle\mathscr{E}_{N}f_{j},g\right\rangle &=\frac{1}{(2\pi)^{d}}\int_{ I_{N}}\left[\sum_{n\in S_{d,N}}
\mathscr{F}_{x}{f}_{j}(n)e^{i(x\cdot n+t|n|_{\pm}^{2}) }\overline{g}dx\right]dt\nonumber\\
&=\frac{1}{(2\pi)^{d}}\sum_{n\in S_{d,N}}\mathscr{F}_{x}{f}_{j}(n)\left[\int_{I_{N}}e^{ it|n|_{\pm}^{2}}
\overline{\int_{\mathbf{T}^{d}}g(x,t)e^{-ix\cdot n }dx}\right]dt \nonumber\\
&=\frac{1}{(2\pi)^{d}}\sum_{n\in S_{d,N}}\mathscr{F}_{x}{f}_{j}(n)\int_{I_{N}}e^{it|n|_{\pm}^{2} }
\overline{\mathscr{F}_{x}{g}}(n,t)dt\nonumber\\
&=\frac{1}{(2\pi)^{d}}\int_{\mathbf{T}^{d}}\left(\sum_{n\in S_{d,N}}e^{-ix\cdot n}\int_{I_{N}}
e^{i t|n|_{\pm}^{2}i}\overline{\mathscr{F}_{x}{g}}(n,t)dt\right)f_{j}dx  \nonumber\\
&=\frac{1}{(2\pi)^{d}}\int_{\mathbf{T}^{d}}\left[\sum_{n\in S_{d,N}}e^{-ix\cdot n}
\overline{\int_{I_{N}}e^{-it|n|_{\pm}^{2}}\mathscr{F}_{x}{g}(n,t)dt}\right]f_{j}dx  \nonumber\\
&=\frac{1}{(2\pi)^{d}}\int_{\mathbf{T}^{d}}\left[\overline{\int_{I_{N}}\sum_{n\in S_{d,N}}
e^{ ix\cdot n}e^{-it|n|_{\pm}^{2}}\mathscr{F}_{x}{g}(n,t)dt}\right]f_{j}dx=
\left\langle f_{j},\mathscr{E}_{N}^{*}g\right\rangle.\label{4.022}
\end{align}
From (\ref{4.022}),  we have that
\begin{align}
\mathscr{E}_{N}^{*}g=\frac{1}{(2\pi)^{d}}\int_{I_{N}}\sum_{n\in S_{d,N}}
e^{ix\cdot n}e^{-it|n|_{\pm}^{2}}\mathscr{F}_{x}{g}(n,t)dt.\label{4.023}
\end{align}
Combining (\ref{4.08}) with (\ref{4.09}), we have
\begin{align}
\mathscr{E}_{N}\mathscr{E}_{N}^{*}g &=\frac{1}{(2\pi)^{2d}}\int_{I_{N}}
\sum_{n\in S_{d,N}}\mathscr{F}_{x}{g}(n,t)e^{i(x\cdot n+(s-t)
|n|_{\pm}^2)}dt
\notag\\&=\frac{1}{(2\pi)^{2d}}\int_{I_{N}}\sum_{n\in S_{d,N}}\int_{\mathbb{T}^{d}}
g(y,t)e^{i((x-y)\cdot n+(s-t)|n|_{\pm}^2)}dydt
\notag\\&=\frac{1}{(2\pi)^{2d}}\int_{\mathbf{T}^{d}\times I_{N}}g(y,t)
\sum_{n\in S_{d,N}}e^{i((x-y)\cdot n+(s-t)|n|_{\pm}^2)}dydt\label{4.024}.
\end{align}
We define
\begin{align}
\int_{\mathbf{T}^{d}\times I_{N}}g(y,t)K_{N}(x-y,s-t)dydt=:g\ast K_{N},\label{4.025}
\end{align}
where
\begin{align}
K_{N}(x,t)=\frac{1}{(2\pi)^{2d}}\sum_{n\in S_{d,\>N}}e^{(x\cdot n+t|n|_{\pm}^2)i}.\label{4.026}
\end{align}
We  define $K_{N,\epsilon}(x,t)=\chi_{\epsilon<|t|<N^{-1}}K_{N}(x,\>t)$ and
$K_{N,\epsilon}^{z}=(z+1)t^{z}K_{N,\>\epsilon}$, from (5.9)  of \cite{KPV}, we have that
\begin{eqnarray}
\left|K_{N,\epsilon}^{z}\right|\leq C(1+|Rez|+|b|)|t|^{Rez-\frac{d}{2}}.\label{4.027}
\end{eqnarray}
By using the idea of Proposition 3.3 of \cite{N} and Lemma 3.15, we have that
Theorem 4.5 is valid.

This completes the proof of Theorem 4.5.

\begin{Theorem}\label{4.6}(Strichartz estimates for orthonormal functions
 related to non-elliptic operator on $\T^{d}$)
Let $\Delta_{\pm}$ be defined as in Lemma 4.3, $N>1$,
 $\lambda=(\lambda_{j})_{j=1}^{+\infty}\in l^{\alpha}
 (\alpha\leq
 \frac{2q}{q+1})$ and orthonormal
 system $(f_{j})_{j=1}^{+\infty}$ in $L^{2}(\mathbb{T}^{d})$,
 $\mathscr{E}_{N}$
 be the dual operator such that  $\langle\mathscr{E}_{N}f_{j},
 F\rangle_{L_{x,t}^2(I_{N}\times  \mathbf{T}^{d})}=\langle f_{j},
 \mathscr{E}_{N}^{*}F\rangle_{I_{N}\times  \mathbf{T}^{d}}$ for $F\in L^2(I_{N}\times
  \mathbb{T}^{d})$.
  Then, we have
\begin{eqnarray}
&&\left\|\sum\limits_{j=1}^{+\infty}\lambda_{j}\left|e^{it\Delta_{\pm}}
P_{\leq N}f_{j}\right|^{2}\right\|_{L_{t}^{p}L_{x}^{q}
(I_{N}\times \mathbf{T}^{d})}\leq C\left\|\lambda\right\|_{l^{\alpha}}\label{4.028}.
\end{eqnarray}
Here
\begin{eqnarray*}
&&\mathscr{E}_{N}f_{j}\definition e^{it\Delta_{\pm}}P_{\leq N}f_{j}\definition\frac{1}{(2\pi)^{d}}
\sum\limits_{n\in S_{d,N}}\mathscr{F}_{x}{f}_{j}(n)e^{i(x\cdot n+ t|n|_{\pm}^{2})},\\
&&S_{d,N}=Z^{d}\cap [-N,N]^{d}, |n|_{\pm}^{2}=\sum\limits_{j=1}^{k}n_{j}^{2}
-\sum\limits_{j=k+1}^{d}n_{j}^{2}(1\leq k\leq d-1),\\
&&|n|_{\pm}^{2}=|n|^{2}(k=d),(\frac{1}{q},\frac{1}{p})\in (A,B],
A=(\frac{d-1}{d+1},\frac{d}{d+1}),B=(1,0).
\end{eqnarray*}
\end{Theorem}
\noindent{\bf Proof.} Combining Lemma 2.10 with Theorem 4.5,
by using a proof similar
 to Proposition 3.3 of \cite{N} and interpolation theorem, we have that Theorem 4.6 is valid.

This completes the proof of Theorem 4.6.

\begin{Theorem}\label{4.7}(Strichartz estimates for orthonormal
functions related to non-elliptic operator on $\T^{d}$)
Let $\Delta_{\pm}$ be defined as in Lemma 4.3, $N>1$,
$\lambda=(\lambda_{j})_{j=1}^{+\infty}\in l^{\alpha}
(\alpha\leq \frac{2q}{q+1})$ and orthonormal
 system $(f_{j})_{j=1}^{+\infty}$ in $L^{2}(\mathbf{T}^{d})$,
 $\mathscr{E}_{N}$
 be the dual operator such that  $\langle\mathscr{E}_{N}f_{j},
 F\rangle_{L_{x,t}^2(\mathbf{T}^{d+1})}=\langle f_{j},
 \mathscr{E}_{N}^{*}F\rangle_{L_{x,t}^2(\mathbf{T}^{d+1})}$ for $F\in L^2(\mathbf{T}^{d+1})$.
 Then, we have
\begin{eqnarray}
&&\left\|\sum\limits_{j=1}^{+\infty}\lambda_{j}
\left|e^{it\Delta_{\pm}}P_{\leq N}f_{j}\right|^{2}\right\|_{L_{t}^{p}L_{x}^{q}
(\mathbf{T}^{d+1})}\leq C
N^{1/p}\left\|\lambda\right\|_{l^{\alpha}}\label{4.029}.
\end{eqnarray}
Here
\begin{eqnarray*}
&&\mathscr{E}_{N}f_{j}\definition e^{it\Delta_{\pm}}P_{\leq N}f_{j}\definition\frac{1}{(2\pi)^{d}}
\sum\limits_{n\in S_{d,N}}\mathscr{F}_{x}{f}_{j}(n)e^{i(x\cdot n+ t|n|_{\pm}^{2})},\\
&&S_{d,N}=Z^{d}\cap [-N,N]^{d}, |n|_{\pm}^{2}=\sum\limits_{j=1}^{k}n_{j}^{2}-
\sum\limits_{j=k+1}^{d}n_{j}^{2}(1\leq k\leq d-1),\\
&&|n|_{\pm}^{2}=|n|^{2}(k=d),(\frac{1}{q},\frac{1}{p})\in (A,B],
A=(\frac{d-1}{d+1},\frac{d}{d+1}),B=(1,0).
\end{eqnarray*}

\end{Theorem}
\noindent{\bf Proof.} Combining Theorem 4.6 with Theorem 1.5 of \cite{N},
we have that Theorem 4.7 is valid.

This completes the proof of Theorem 4.7.

\begin{Theorem}(Schatten bound with space-time norms related to Boussinesq
 operator with small time on $\R$)\label{4.8}
Let  $p^{\prime}=\frac{p}{p-1},q^{\prime}=\frac{q}{q-1}\geq1$ satisfy
$$\frac{1}{p^{\prime}}+\frac{1}{2q^{\prime}}=1, q^{\prime}>1.$$ Then,
 we have the Schatten bound
\begin{eqnarray}
&&\left\|W_{1}U_{1}(t)U^{*}_{1}(t)W_{2}\right\|_{\mathfrak{S}^{2q^{\prime}}}\leq C
\left\|W_{1}\right\|_{L_{t}^{2p^{\prime}}L_{x}^{2q^{\prime}}(\mathbb{I}\times\mathbf{R})}
\left\|W_{2}\right\|_{L_{t}^{2p^{\prime}}L_{x}^{2q^{\prime}}(\mathbb{I}\times\mathbf{R})}.\label{4.035}
\end{eqnarray}
Here  $U_{1}(t)f:=e^{it(-\partial_{x}^{2}+\partial_{x}^{4})^{1/2}}f$,
$I=[0,1]$.

\end{Theorem}
\noindent{\bf Proof.} From Lemma 3.25,  we have
that
\begin{eqnarray}
\left|\int_{\SR}e^{ix\xi+it(\xi^{2}+\xi^{4})^{1/2}}d\xi\right|\leq  C|t|^{-1/2}.\label{4.036}
\end{eqnarray}
Here $ K_{1}(t,x)=\int_{\SR^{d}}e^{it\sqrt{\xi^{2}+\xi^{4}}}e^{ix\cdot\xi}d\xi.$
We define
\begin{eqnarray}
&&K_{1z}(t,x)=(z+1)t^{z}K_{1}(t,x), K_{1z,\epsilon}(t,x)=(z+1)t^{z}\chi_{\{t||t|>\epsilon\}
\cap \mathbb{I}}(t)K_{1}(t,x)\nonumber\\
&&T_{1z}f=K_{1z}\ast f,
T_{1z,\epsilon}f=K_{1z,\epsilon}\ast f,
G(s,\xi)=e^{is|\xi|^2}\mathscr{F}_{x}f(\xi,s).\label{4.037}
\end{eqnarray}
Then, we have
\begin{eqnarray}
&&K_{1z}(t,x)=(z+1)t^{z}K_{1}(t,x)=\lim\limits_{\epsilon\longrightarrow 0}
K_{1z,\epsilon}(t,x)=
\lim\limits_{\epsilon\longrightarrow 0}(z+1)t^{z}\chi_{\{t||t|>\epsilon\}
\cap \mathbb{I}}(t)K_{1}(t,x)\nonumber\\
&&T_{1z}f=\lim\limits_{\epsilon\longrightarrow 0}T_{1z,\epsilon}f=
\lim\limits_{\epsilon\longrightarrow 0}K_{1z,\epsilon}\ast f.\label{4.038}
\end{eqnarray}
In particular, $T_{10}=U(t)U^{*}(t).$
By using (\ref{4.035})-(\ref{4.038}) and a proof similar to Theorem 4.1,
we derive that Theorem 4.8 is valid.

This completes the proof of Theorem 4.8.

\begin{Theorem}(Strichartz estimates for orthonormal functions related
 to Boussinesq operator with small time on $\R$)\label{4.9}
Let   $p,q\geq1$  satisfy
$$\frac{2}{p}+\frac{1}{q}=1,\ \ q\geq1.$$
Then, for any (possibly infinite) orthonormal system
 $(f_{j})_{j=1}^{+\infty}$ in $L^{2}(\mathbb{R}^{d})$
 and for all sequence $(\nu_{j})_{j=1}^{+\infty}\subset\mathbf{C}$,
  we derive that
\begin{eqnarray}
&&\left\|\sum\limits_{j=1}^{+\infty}\nu_{j}\left|U_{1}(t)f_{j}\right|^{2}\right\|_{L_{t}^{p}L_{x}^{q}
(\mathbb{I}\times\mathbf{R})}\leq C
\left(\sum\limits_{j=1}^{+\infty}\left|\nu_{j}\right|^{\frac{2q}{q+1}}\right)^{\frac{q+1}{2q}},\label{4.039}
\end{eqnarray}
where  $(f_{j})_{j=1}^{+\infty}$. Here
  $U_{1}(t)f=e^{it(-\partial_{x}^{2}+\partial_{x}^{4})^{1/2}}f$ and  $I=[0,1]$.
\end{Theorem}
\noindent{\bf Proof.} By using Lemma 2.10 and Theorem 4.8,
 we derive that Theorem 4.9 is valid.

This completes the proof of Theorem 4.9.

\begin{Theorem}(Schatten bound with space-time norms related to
Boussinesq operator with large time on $\R$)\label{4.10}
Let  $p^{\prime}=\frac{p}{p-1},q^{\prime}=\frac{q}{q-1}$ satisfy
$$\frac{1}{p^{\prime}}+\frac{1}{3q^{\prime}}=1, q^{\prime}>1.$$ Then,
 we have the Schatten bound
\begin{eqnarray}
&&\left\|W_{1}U_{1}(t)U^{*}_{1}(t)W_{2}\right\|_{\mathfrak{S}^{2q^{\prime}}}\leq C
\left\|W_{1}\right\|_{L_{t}^{2p^{\prime}}L_{x}^{2q^{\prime}}
(\mathbb{I}\times\mathbf{R})}
\left\|W_{2}\right\|_{L_{t}^{2p^{\prime}}L_{x}^{2q^{\prime}}
(\mathbb{I}\times\mathbf{R})},\label{4.040}
\end{eqnarray}
where $C>0$ is independent of $W_{1},W_{2}$. Here
  $U_{1}(t)f=e^{it(-\partial_{x}^{2}+\partial_{x}^{4})^{1/2}}f$,
   $I=[1,+\infty)$.
\end{Theorem}
\noindent{\bf Proof.} From Lemma 3.28,  we have
that
\begin{eqnarray}
\left|\int_{\SR}e^{ix\xi+it(\xi^{2}+\xi^{4})^{1/2}}d\xi\right|\leq
C|t|^{-1/3}.\label{4.041}
\end{eqnarray}
By using (\ref{4.041}) and a proof similar to Theorem 4.1, we derive
 that Theorem 4.10 is valid.

This completes the proof of Theorem 4.10.

\begin{Theorem}(Strichartz estimates for orthonormal functions
related to Boussinesq operator with large time on $\R$)\label{4.11}
Let   $p,q\geq1$  satisfy
$$\frac{3}{p}+\frac{1}{q}=1,\ \ q>1.$$
Then, for any (possibly infinite) orthonormal system $(f_{j})_{j=1}^{+\infty}$
 in $L^{2}(\mathbf{R}^{d})$
 and for all sequence $(\nu_{j})_{j=1}^{+\infty}\subset\mathbf{C}$, we derive
  that
\begin{eqnarray}
&&\left\|\sum\limits_{j=1}^{+\infty}\nu_{j}\left|U_{1}(t)f_{j}\right|^{2}\right\|_{L_{t}^{p}L_{x}^{q}
(\mathbb{I}\times\mathbf{R})}\leq C\left(\sum\limits_{j=1}^{+\infty}
\left|\nu_{j}\right|^{\frac{2q}{q+1}}\right)^{\frac{q+1}{2q}},\label{4.042}
\end{eqnarray}
where $C>0$  is independent of $(\nu_{j})_{j=1}^{+\infty}$ and
$(f_{j})_{j=1}^{+\infty}$. Here
  $U_{1}(t)f=e^{it(-\partial_{x}^{2}+\partial_{x}^{4})^{1/2}}f$ and
    $I=[1,+\infty)$.
\end{Theorem}
\noindent{\bf Proof.} By using Lemma 2.10 and Theorem
 4.10, we derive that Theorem 4.11 is valid.

This completes the proof of Theorem 4.11.

\bigskip
\bigskip

\section{Convergence  problem of some compact operators  and
some operator equations in Schatten norms}
\medskip

\setcounter{equation}{0}

\setcounter{Theorem}{0}

\setcounter{Lemma}{0}

\setcounter{section}{5}

In this section, we  establish  the convergence of
 some compact operators in Schatten
 norms related to elliptic operator, non-elliptic
  operator and Boussinesq operator and
 the convergence  result related to nonlinear part of
 the solution to
some operator equations in Schatten norms.

\begin{Theorem}\label{5.1}(The convergence of some compact
 operators in Schatten norms related to elliptic operator)
Let $U(t)=e^{it\Delta}$,  $W_{j}\in L_{t}^{2p^{\prime}}L_{x}^{2q^{\prime}}
([0,a]\times\mathbf{R}^{d})(j=1,2),$ $d\geq1$ and
$p^{\prime}=\frac{p}{p-1},q^{\prime}=\frac{q}{q-1}\geq1$ satisfy
$$\frac{1}{p^{\prime}}+\frac{d}{2q^{\prime}}=1,
\ \ q^{\prime}>\frac{d+1}{2}.$$
Then, we have
\begin{eqnarray}
\lim\limits_{t\longrightarrow0}
\left\|W_{1}U(t)U^{*}(t)W_{2}\right\|_{\mathfrak{S}^{2q^{\prime}}}=0.\label{5.01}
\end{eqnarray}
\end{Theorem}
\noindent{\bf Proof.}Combining Lemma 2.9 with (\ref{4.01}),
 we derive that
\begin{eqnarray*}
\lim\limits_{a\longrightarrow0}
\left\|W_{1}U(t)U^{*}(t)W_{2}\right\|_{\mathfrak{S}^{2q^{\prime}}}
\leq C\lim\limits_{a\longrightarrow0}
\left\|W_{1}\right\|_{L_{t}^{2p^{\prime}}L_{x}^{2q^{\prime}}([0,a]\times\mathbf{R}^{d})}
\left\|W_{2}\right\|_{L_{t}^{2p^{\prime}}L_{x}^{2q^{\prime}}([0,a]\times\mathbf{R}^{d})}=0.
\end{eqnarray*}

This completes the proof of Theorem 5.1.

\begin{Theorem}\label{5.2}(The convergence  of some compact
 operators in Schatten norms related to non-elliptic operator)
Let $\Delta_{\pm}=\sum\limits_{j=1}^{k}\frac{\partial^{2}}{\partial_{x_{j}}^{2}}
-\sum\limits_{j=k+1}^{d}\frac{\partial^{2}}{\partial_{x_{j}}^{2}}(1\leq k\leq d-1)$, $\Delta_{\pm}=\Delta(k=d)$, $d\geq2$
and $p^{\prime}=\frac{p}{p-1},q^{\prime}=\frac{q}{q-1}\geq 1$
satisfy
$$\frac{1}{p^{\prime}}+\frac{d}{2q^{\prime}}=1,
\ \ q^{\prime}>\frac{d+1}{2}.$$
Then, we have
\begin{eqnarray}
\lim\limits_{t\longrightarrow0}
\left\|W_{1}U_{\pm}(t)U^{*}_{\pm}(t)W_{2}\right\|_{\mathfrak{S}^{2q^{\prime}}}=0.\label{5.02}
\end{eqnarray}
\end{Theorem}
\noindent{\bf Proof.}Combining Lemma 2.9 with (\ref{4.016}), we derive that
\begin{eqnarray*}
\lim\limits_{a\longrightarrow0}\left\|W_{1}U_{\pm}(t)U_{\pm}^{*}(t)W_{2}\right\|_{\mathfrak{S}^{2q^{\prime}}}
\leq C\lim\limits_{a\longrightarrow0}\left\|W_{1}\right\|_{L_{t}^{2p^{\prime}}L_{x}^{2q^{\prime}}([0,a]\times\mathbf{R}^{d})}
\left\|W_{2}\right\|_{L_{t}^{2p^{\prime}}L_{x}^{2q^{\prime}}([0,a]\times\mathbf{R}^{d})}=0.
\end{eqnarray*}

This completes the proof of Theorem 5.2.

\begin{Theorem}\label{5.3}(The convergence  of
some compact operators in Schatten norms related to Boussinesq operator)
Let $U_{1}(t)f:=e^{it(-\partial_{x}^{2}+\partial_{x}^{4})^{1/2}}f$
 and $p^{\prime}=\frac{p}{p-1},q^{\prime}=\frac{q}{q-1}$ satisfy
$$\frac{1}{p^{\prime}}+\frac{1}{2q^{\prime}}=1, q^{\prime}>1.$$
 Then, we have
\begin{eqnarray}
\lim\limits_{t\longrightarrow0}
\left\|W_{1}U_{1}(t)U_{1}^{*}(t)W_{2}\right\|_{\mathfrak{S}^{2q^{\prime}}}=0.\label{5.03}
\end{eqnarray}
\end{Theorem}
\noindent{\bf Proof.} Combining Lemma 2.9 with (\ref{4.035}),
we derive that
\begin{eqnarray*}
\lim\limits_{a\longrightarrow0}\left\|W_{1}U_{1}(t)U_{1}^{*}(t)W_{2}\right\|_{\mathfrak{S}^{2q^{\prime}}}
\leq C\lim\limits_{a\longrightarrow0}\left\|W_{1}\right\|_{L_{t}^{2p^{\prime}}L_{x}^{2q^{\prime}}([0,a]\times\mathbf{R}^{d})}
\left\|W_{2}\right\|_{L_{t}^{2p^{\prime}}L_{x}^{2q^{\prime}}([0,a]\times\mathbf{R}^{d})}=0.
\end{eqnarray*}

This completes the proof of Theorem 5.3.

\begin{Theorem}\label{5.4}(The convergence   related to nonlinear part of the solution to
some operator equations in Schatten norms)
Assume that  $d\geq1,1\leq q<1+\frac{2}{d-1},p\geq1$ satisfy
 $\frac{2}{p}+\frac{d}{q}=d$ and $\gamma_{0}\in \mathfrak{S}^{\frac{2q}{q+1}}$
and $w\in L_{x}^{q^{\prime}}(\R^{d})$. Then, we have
\begin{eqnarray}
\lim\limits_{T\longrightarrow0}\left\|\gamma(t)-e^{it\Delta}\gamma_{0}e^{-it\Delta}
\right\|_{C_{t}^{0}([0,T],\mathfrak{S}^{\frac{2q}{q+1}})}=0.\label{5.04}
\end{eqnarray}
Here, $\gamma_{0}\in \mathfrak{S}^{\frac{2q}{q+1}}$ and
 $\gamma \in C_{t}^{0}(\R,\mathfrak{S}^{\frac{2q}{q+1}})$  is the solution to
$\left\{
        \begin{array}{ll}
         i\frac{d\gamma(t)}{dt}=[-\Delta+w*\rho_{\gamma},\gamma(t)] \\
          \gamma|_{t=0}=\gamma_{0}.
        \end{array}
      \right.$

\end{Theorem}
\noindent{\bf Proof.} From Theorem 14 of \cite{FS2017}, we have that
 \begin{eqnarray}
\left\|\gamma(t)-e^{it\Delta}\gamma_{0}e^{-it\Delta}
\right\|_{C_{t}^{0}([0,T],\mathfrak{S}^{\frac{2q}{q+1}})}\leq
 8T^{1/p^{\prime}}\|w\|_{L_{x}^{q^{\prime}}}{\rm max}(1,C_{stri}^{2})R^{2}.\label{5.05}
\end{eqnarray}
From (\ref{5.05}), we have that (\ref{5.04}) is valid.

This completes the proof of Theorem 5.4.

\bigskip
\bigskip
 \section{Probabilistic convergence of density functions
of some compact operators with full randomization on $\R^{d}$}

\setcounter{equation}{0}

\setcounter{Theorem}{0}

\setcounter{Lemma}{0}

\setcounter{section}{6}
By using the idea of  Lemma 2.8 of \cite{YZY}  and the
 idea of line 14 of 1908 in \cite{YZDY},
we  present the probabilistic convergence of density functions
of some compact operators with full randomization on $\R^{d}$.

Randomization on a single function is introduced by Lebowitz et al.
\cite{LRS} and  Bourgain
\cite{Bourgain1994, Bourgain1996} and developed in
\cite{ZF2011,ZF2012,BOP2015,BOP,BTL,BT,LMCPDE}.

In this section, by using the full randomization of
 some compact operators
 on $\R^{d}$ \cite{HY} and  Lemma 6.6 which is proved
  with the aid of
  Lemmas 6.1, 6.5, 6.3, 6.4,  we present the
probabilistic convergence of density functions   of
 some compact operators
 with full randomization on $\R^{d}$.

\indent Now we recall the  randomization of a single function on
 $\R^{d}$,
 which originated from
 \cite{BOP2015,BOP,ZF2012}. We define  $B^{d}(0,1)=
 \left\{\xi\in \R^{d}:|\xi|\leq 1\right\}$.
Let  $\psi\in C_{c}^{\infty}(\R^{d})$ be a real-valued,
even, non-negative bump function with $\supp\psi\subset B^{d}(0,1)$
  such that
\begin{eqnarray}
&&\sum\limits_{k\in \z^{d}}\psi(\xi-k)=1,\forall \xi\in\R^{d},\label{6.01}
\end{eqnarray}
which is known as  Wiener decomposition of  the frequency space:
 $\psi(D-k)f:\R^{d}\rightarrow\mathbb{C}$ is defined by
\begin{eqnarray}
(\psi(D-k)f)(x)=\mathscr{F}^{-1}\big(\psi(\xi-k)\mathscr{F}f\big)(x),
x\in \R^{d},\forall k\in \Z^{d}.\label{6.02}
\end{eqnarray}
Let $\{g^{(1)}_{k_{j}}(\omega_{j})\}_{k_{j}\in \mathbf{N}^{+}}$ be  sequences of
 independent, zero-mean, real-valued  random variables
 on the probability spaces
$
(\Omega_{j},\mathcal{A}_{j}, \mathbb{P}_{j})(1\leq j\leq d)
$ endowed with probability
distributions $\mu_{k_{j}}^{1}$ satisfying
\begin{eqnarray}
&&\Big|\int_{-\infty}^{+\infty}e^{\gamma_{k_{j}} x_{j}}d\mu_{k_{j}}^{1}(x_{j})\Big|
\leq e^{C_{k_{j}}\gamma_{k_{j}}^{2}}(1\leq j\leq d,k_{j}\in \N^{+}).\label{6.03}
\end{eqnarray}
Let $(\Omega,\mathcal{A}, \mathbb{P})=
(\prod\limits_{j=1}^{d}\Omega_{j},\prod\limits_{j=1}^{d}\mathcal{A}_{j},
\prod\limits_{j=1}^{d}\mathbb{P}_{j}),\omega=(\omega_{1},\cdot\cdot\cdot,\omega_{d}).
$
Let
$\{g^{(2)}_n(\widetilde{\omega})\}_{n\in \z}$
be a sequence of independent, zero-mean, real-valued
 random variable
 on the probability space  $(\widetilde{\Omega},\widetilde{\mathcal{A}},
\widetilde{\mathbb{P}})$
endowed with probability
distributions $\mu_{n}^{2}$ satisfying
\begin{eqnarray}
&&\Big|\int_{-\infty}^{+\infty}e^{\gamma_{n} x}d\mu_{n}^{2}(x)\Big|
\leq e^{C_{n}\gamma_{n}^2}(n\in \N^{+}).\label{6.04}
\end{eqnarray}
For $f\in L^{2}(\R^{d})$, we define its randomization by
\begin{eqnarray}
&&f^{\omega}:=\sum\limits_{k=(k_{1},\>\cdot\cdot\cdot,\>k_{d})\in\z^{d}}
\prod\limits_{j=1}^{d}g^{(1)}_{k_{j}}(\omega_{j})\psi(D-k)f.\label{6.05}
\end{eqnarray}
We define
\begin{eqnarray*}
\|f\|_{L_{\omega,\>\widetilde{\omega}}^{p}(\Omega\times \widetilde{\Omega} )}
:=\left[\int_{\Omega\times \widetilde{\Omega}}|f(\omega,\widetilde{\omega})|^{p}
d(\mathbb{P}\times\mathbb {\widetilde{P})}\right]^{\frac{1}{p}}
=\left[\int_{\Omega}\int_{\widetilde{\Omega}}|f(\omega,\widetilde{\omega})|^{p}
d\mathbb{P}(\omega)d\mathbb{\widetilde{P}}(\widetilde{\omega})\right]^{\frac{1}{p}}.
\end{eqnarray*}

\noindent(Full randomization of compact operator.)Inspired by
 \cite{HY}, now we present the
full randomization of some compact operators.
Let $\gamma_{0}$ be a compact operator.
 Then we have the singular value decomposition
\begin{align}
\gamma_{0}=\sum_{n=1}^{\infty}\lambda_{n}|f_{n}\rangle\langle f_{n}|,\label{6.06}
\end{align}
where $\lambda_{n}\in \mathbf{C}$, $(f_{n})_{n=1}^{\infty}$ is orthonormal
 system in $L^{2}(\R^{d})$.
For any $(\omega,\widetilde{\omega})\in\Omega \times \widetilde{\Omega}$,
we define the full randomization of compact operator $\gamma_{0}=
\sum\limits_{n=1}^{+\infty}\lambda_{n}|f_{n}\rangle\langle f_{n}|$
by
\begin{eqnarray}
\gamma_{0}^{\omega,\>\widetilde{\omega}}:=
\sum_{n=1}^{\infty}\lambda_{n}g^{(2)}_{n}(\widetilde{\omega})|f_{n}^{\omega}
\rangle\langle f_{n}^{\omega}|.\label{6.07}
\end{eqnarray}
Obviously,
\begin{eqnarray*}
\rho_{\gamma_{0}^{\omega,\>\widetilde{\omega}}}=
\sum_{n=1}^{\infty}\lambda_{n}g^{(2)}_{n}(\widetilde{\omega})|f_{n}^{\omega}|^{2},
\rho_{e^{it\Delta}\gamma_{0}^{\omega,\>\widetilde{\omega}}e^{-it\Delta}}
=\sum_{n=1}^{\infty}\lambda_{n}g^{(2)}_{n}(\widetilde{\omega})|e^{it\Delta}f_{n}^{\omega}|^{2}.
\end{eqnarray*}

\begin{Lemma}(Convergence in measure) \label{Lemma6.1}
Let $F(t,\omega,\widetilde{\omega})$ be
a real-valued measurable function on
$[-1,1]\times\Omega\times \widetilde{\Omega}$.
For all $\epsilon>0$,
\begin{align}
\lim\limits_{t\rightarrow t_{0}}E(A_{\omega,\>\phi(\epsilon)})=0,
\label{6.08}
\end{align}
where $A_{\omega,\widetilde{\omega},\>\phi(\epsilon)}=
\left\{(\omega,\widetilde{\omega})\in(\Omega\times \widetilde{\Omega})|
|f(t,\omega,\widetilde{\omega})-
f(t_{0},\omega,\widetilde{\omega})|>\phi(\epsilon)\right\},$
$\phi>0$ is  monotonically increasing and
\begin{eqnarray*}
\lim\limits_{\epsilon\longrightarrow 0^{+}}\phi(\epsilon)=0.
\end{eqnarray*}
Then, (\ref{6.08}) is equivalent to the following statement: for all
 $\epsilon>0$, $\exists \delta>0$, when $|t-t_{0}|<\delta,$ we have
\begin{eqnarray}
E(A_{\omega,\widetilde{\omega},\>\phi(\epsilon)})<\epsilon.\label{6.09}
\end{eqnarray}
\end{Lemma}
 \noindent{\bf Proof.} For all $\epsilon>0$, since
\begin{align}
\lim\limits_{t\rightarrow t_{0}}E(A_{\omega,\widetilde{\omega},\>\phi(\epsilon)})=0,\label{6.010}
\end{align}
we have that for all $\epsilon>0$, $\exists \delta>0$,
when $|t-t_{0}|<\delta,$ we have that
(\ref{6.09}) is valid.
When (\ref{6.09}) is valid, pick $0<\alpha<\epsilon,$
$\exists \delta>0,$ when $|t-t_{0}|<\delta$, we have
\begin{eqnarray}
E(A_{\omega,\widetilde{\omega},\>\phi(\alpha)})<\alpha.\label{6.011}
\end{eqnarray}
Since $\phi>0$ is  monotonically increasing, by using
(\ref{6.011}),  we have $\phi(\alpha)<\phi(\epsilon)$,
\begin{eqnarray}
E(A_{\omega,\widetilde{\omega},\>\phi(\epsilon)})<
E(A_{\omega,\widetilde{\omega},\>\phi(\alpha)})<\alpha.\label{6.012}
\end{eqnarray}
From (\ref{6.012}), we have
\begin{eqnarray}
\lim\limits_{t\rightarrow t_{0}}
E(A_{\omega,\widetilde{\omega},\>\phi(\epsilon)})\leq\alpha.\label{6.013}
\end{eqnarray}
Let $\alpha$ go to zero, then we have that (\ref{6.08})
is valid.

 This completes the proof of Lemma 6.1.

\noindent{\bf Remark 7.}Lemma 6.1 generalizes the definition
  of Definition 1.1.7 of \cite{G}.

\begin{Lemma}\label{6.2}(Khinchin type inequalities)
Let $\{g^{(1)}_{k_{j}}(\omega_{j})\}_{k_{j}\in \mathbf {N}^{+}},\{g^{(2)}_k(\widetilde{\omega})\}_{k\in \mathbf {N}^{+}}$ be
  sequences of independent, zero-mean, real-valued  random variables
 on the probability spaces $(\Omega_{j},\mathcal{A}_{j}, \mathbb{P}_{j}),$
$(\widetilde{\Omega},\widetilde{\mathcal{A}}, \widetilde{\mathbb{P}})$
  endowed with probability
distributions $\mu^{1}_{k_{j}}(1\leq j\leq d,k_{j}\in \N^{+}),\mu^{2}_n(n\geq 1)$
which satisfy (\ref{6.03}), (\ref{6.04}), respectively.
 Then there exist  constants $C_{k_{j}}, C_{n}$
such that
\begin{eqnarray}
&&\left\|\sum_{k_{j}=1}^{\infty}a_{k_{j}}g_{k_{j}}^{(1)}(\omega_{j})\right\|_{L_{\omega_{j}}^{r}(\Omega_{j})}
\leq C_{k_{j}}r^{\frac{1}{2}}\left\|a_{k_{j}}\right\|_{l_{k_{j}}^{2}}\label{6.014},\\
&&\left\|\sum_{n=1}^{\infty}b_{n}g_{n}^{(2)}(\widetilde{\omega})\right\|_{L_{\widetilde{\omega}}^{r}(\widetilde{\Omega})}
\leq C_{n}r^{\frac{1}{2}}\left\|b_{n}\right\|_{l_{n}^{2}},\label{6.015}
\end{eqnarray}
hold for any $r\in [2,\infty)$ and $(a_{k_{j}})_{k_{j}},(b_{n})_{n}\in l_{n}^{2}$.
\end{Lemma}

For the proof of  Lemma 6.2,  we refer the reader  to Lemma 3.1 of
\cite{BTL}.

\begin{Lemma}\label{lem6.3}(The estimate related to $l_{n}^{2}L_{\omega}^{p}$)
Let  $(\lambda_{n})_{n}\in l^{2}$ and $f_{n}^{\omega}$
 the randomization of $f_{n}$ be
defined as in (\ref{6.05}). Then, for $\forall \epsilon>0$,
$\exists \delta>0$,
 when $|t|<\delta<\frac{\epsilon^{d+1}}{M_{1}^{2}},$  we have
\begin{eqnarray}
&&\left\|\lambda_{n}\left\|U(t)f_{n}^{\omega}-
f_{n}^{\omega}\right\|_{L_{\omega}^{p}}\right\|_{l_{n}^{2}}\leq Cp^{d/2}\epsilon^{d+1}\left[\left(\sum\limits_{n=1}^{\infty}\lambda_{n}^{2}\right)^{1/2}+1\right]\nonumber\\
&&= Cp^{d/2}\epsilon^{d+1}\left[
\left\|\gamma_{0}\right\|_{\mathfrak{S}^{2}}+1\right].\label{6.016}
\end{eqnarray}
Here $M_{1}$ appears in (\ref{6.022}) and
$\left\|\gamma_{0}\right\|_{\mathfrak{S}^{2}}=
\left(\sum\limits_{n=1}^{\infty}\lambda_{n}^{2}\right)^{1/2}$.
\end{Lemma}
\noindent {\bf Proof.} From Lemma 6.2, we have
\begin{eqnarray}
&&\hspace{-1cm}\left\|\lambda_{n}\left\|U(t)f_{n}^{\omega}-
f_{n}^{\omega}\right\|_{L_{\omega}^{p}}\right\|_{l_{n}^{2}}\nonumber\\
&&\hspace{-1cm}\leq Cp^{d/2}\left[\sum_{n=1}^{+\infty}|\lambda_{n}|^{2}
\sum\limits_{k\in \z^{d}}\left|\int_{\SR^{d}}\psi(\xi-k)(e^{-it|\xi|^{2}}-1)
e^{i\sum\limits_{k=1}^{d}x_{k}\xi_{k}}
\mathscr{F}_{x}f_{n}(\xi)d\xi\right|^{2}\right]^{\frac{1}{2}}
.\label{6.017}
\end{eqnarray}
Since $(\lambda_{n})_{n}\in l^{2}$, for $\forall \epsilon>0$,
we know that there
exists $M\in \mathbf{N}^{+}$ such that
\begin{eqnarray}
\left[\sum\limits_{n=M+1}^{\infty}|\lambda_{n}|^{2}\right]^{1/2}<
\frac{\epsilon^{d+1}}{2}.\label{6.018}
\end{eqnarray}
By using (\ref{6.018}), the Cauchy-Schwartz inequality with respect to
 $\xi$ as well as $|e^{-it|\xi|^{2}}-1|\leq2$, since
 $\|f_{n}\|_{L^{2}}=1,$ we know that
\begin{eqnarray}
&&\hspace{-1cm}\left[\sum_{n=M+1}^{+\infty}|\lambda_{n}|^{2}
\sum\limits_{k\in \z^{d}}\left|\int_{\SR^{d}}\psi(\xi-k)(e^{-it|\xi|^{2}}-1)
e^{i\sum\limits_{k=1}^{d}x_{k}\xi_{k}}
\mathscr{F}_{x}f_{n}(\xi)d\xi\right|^{2}\right]^{\frac{1}{2}}\nonumber\\
&&\leq C\left[\sum_{n=M+1}^{+\infty}|\lambda_{n}|^{2}\sum\limits_{k\in \z^{d}}
\int_{\SR^{d}}\left|\psi(\xi-k)
\mathscr{F}_{x}f_{n}(\xi)\right|^{2}d\xi\right]^{\frac{1}{2}}\nonumber\\
&&\leq C\left[\sum_{n=M+1}^{+\infty}|\lambda_{n}|^{2}
\|f_{n}\|_{L^{2}}^{2}\right]^{\frac{1}{2}}\leq
C\left[\sum_{n=M+1}^{+\infty}|\lambda_{n}|^{2}\right]^{\frac{1}{2}}
<C\epsilon^{d+1}
.\label{6.019}
\end{eqnarray}
We  claim that for $1\leq n\leq M$, $\forall \epsilon>0$,
 we have that
\begin{eqnarray}
\left[\sum\limits_{k\in \z^{d}}\left|\int_{\SR^{d}}\psi(\xi-k)
(e^{-it|\xi|^{2}}-1)
e^{i\sum\limits_{k=1}^{d}x_{k}\xi_{k}}\mathscr{F}_{x}f_{n}(\xi)
d\xi\right|^{2}\right]^{1/2}
\leq C(\epsilon^{d+1}+|t||M_{1}^{2}).\label{6.020}
\end{eqnarray}
Here, $M_1$ depends on $M$.
We will use  the idea of    Lemma 2.8 of \cite{YZY}  and the
 idea of line 14 of 1908 in \cite{YZDY} to prove Lemma 6.3.
By using a proof similar to line 14 of 1908 in \cite{YZDY}, we have that
\begin{eqnarray}
\sum\limits_{k\in \z^{d}}\left\|\psi(\xi-k)\mathscr{F}_{x}f_{n}\right\|_{L^{2}}^{2}
\leq \|f_{n}\|_{L^{2}}^{2}\leq  3\sum\limits_{k\in \z^{d}}
\left\|\psi(\xi-k)\mathscr{F}_{x}f_{n}\right\|_{L^{2}}^{2}.\label{6.021}
\end{eqnarray}
Thus, from (\ref{6.021}), $\forall\epsilon>0$, we know
 that there
 exists $M_{1}\geq 100$ such that
\begin{eqnarray}
\left[\sum\limits_{|k|\geq M_{1}}\left\|\psi(\xi-k)
\mathscr{F}_{x}f_{n}\right\|_{L^{2}}^{2}\right]^{1/2}
\leq \frac{\epsilon^{d+1}}{2}.\label{6.022}
\end{eqnarray}
Here we use the fact that $1\leq n\leq M$.
Thus, by using the Cauchy-Schwartz inequality, $\supp\psi\subset B^{d}(0,1)$
 and $|e^{-it|\xi|^{2}}-1|\leq 2$ as well as  (\ref{6.022}),
we have that
\begin{eqnarray}
&&\left[\sum\limits_{|k|\geq M_{1}}\left|\int_{\SR^{d}}\psi(\xi-k)(e^{-it|\xi|^{2}}-1)
e^{i\sum\limits_{k=1}^{d}x_{k}\xi_{k}}\mathscr{F}_{x}f_{n}(\xi)d\xi\right|^{2}\right]^{1/2}\nonumber\\&&
\leq 2\left[\sum\limits_{|k|\geq M_{1}}\int_{\SR^{d}}\left|\psi(\xi-k)
\mathscr{F}_{x}f_{n}(\xi)\right|^{2}d\xi\right]^{1/2}\nonumber\\&&=2\left[\sum\limits_{|k|\geq M_{1}}\left\|\psi(\xi-k)\mathscr{F}_{x}f_{n}\right\|_{L^{2}}^{2}\right]^{1/2}\leq \epsilon^{d+1}.\label{6.023}
\end{eqnarray}
Here we use the fact that $1\leq n\leq M$.
Thus, by using the Cauchy-Schwartz inequality, $\supp\psi\subset B^{d}(0,1)$
 and $|e^{-it|\xi|^{2}}-1|\leq |t||\xi|^{2}$ as well as  (\ref{6.022}),
we have that
\begin{eqnarray}
&&\left[\sum\limits_{|k|\leq M_{1}}\left|\int_{\SR^{d}}\psi(\xi-k)(e^{-it|\xi|^{2}}-1)
e^{i\sum\limits_{k=1}^{d}x_{k}\xi_{k}}\mathscr{F}_{x}f_{n}(\xi)d\xi\right|^{2}\right]^{1/2}\nonumber\\&&
\leq \left[\sum\limits_{|k|\leq M_{1}}|t|^{2}|k|^{4}\left|\int_{\SR^{d}}|\psi(\xi-k)
\mathscr{F}_{x}f_{n}(\xi)|d\xi\right|^{2}\right]^{1/2}\nonumber\\&&
\leq |t|M_{1}^{2}\left[\sum\limits_{|k|\leq M_{1}}\left\|\psi(\xi-k)
\mathscr{F}_{x}f_{n}\right\|_{L^{2}}^{2}\right]^{1/2}\leq C|t|M_{1}^{2}\left[\sum\limits_{k\in \z^{d}}\left\|\psi(\xi-k)\mathscr{F}_{x}f_{n}\right\|_{L^{2}}^{2}\right]^{1/2}\nonumber\\
&&\leq C|t|M_{1}^{2}\|f_{n}\|_{L^{2}}\leq C|t|M_{1}^{2}\label{6.024}
\end{eqnarray}
Combining (\ref{6.023}) with (\ref{6.024}), we derive that (\ref{6.020})
is valid.
From (\ref{6.020})-(\ref{6.022}), by using $(a+b)^{1/2}\leq a^{1/2}+b^{1/2}$
 and $|e^{-it|\xi|^{2}}-1|\leq2$
 as well as $\|f_{n}\|_{L^{2}(\SR^{d})}=1(n\in \mathbf{N}^{+})$, when
 $|t|<\delta<\frac{\epsilon^{d+1}}{M_{1}^{2}},$   we know that
\begin{eqnarray}
&&\left\|\lambda_{n}\left\|U(t)f_{n}^{\omega}-f_{n}^{\omega}\right\|_{L_{\omega}^{p}}\right\|_{l_{n}^{2}}\nonumber\\
&&\leq Cp^{d/2}\left[\sum_{n=1}^{+\infty}|\lambda_{n}|^{2}
\sum\limits_{k\in \z^{d}}\left|\int_{\SR^{d}}\psi(\xi-k)(e^{-it|\xi|^{2}}-1)
e^{i\sum\limits_{k=1}^{d}x_{k}\xi_{k}}
\mathscr{F}_{x}f_{n}(\xi)d\xi\right|^{2}\right]^{\frac{1}{2}}\nonumber\\&&
\leq Cp^{d/2}\left[\sum_{n=1}^{M}|\lambda_{n}|^{2}\sum\limits_{k\in \z^{d}}
\left|\int_{\SR^{d}}\psi(\xi-k)(e^{-it|\xi|^{2}}-1)
e^{i\sum\limits_{k=1}^{d}x_{k}\xi_{k}}
\mathscr{F}_{x}f_{n}(\xi)d\xi\right|^{2}\right]^{\frac{1}{2}}\nonumber\\&&
\qquad + Cp^{d/2}\left[\sum_{n=M+1}^{+\infty}|\lambda_{n}|^{2}\sum\limits_{k\in \z^{d}}
\left|\int_{\SR^{d}}\psi(\xi-k)(e^{-it|\xi|^{2}}-1)
e^{i\sum\limits_{k=1}^{d}x_{k}\xi_{k}}
\mathscr{F}_{x}f_{n}(\xi)d\xi\right|^{2}\right]^{\frac{1}{2}}\nonumber\\
&&\leq Cp^{d/2}\left(|t|M_{1}^{2}+\epsilon^{d+1}\right)
\left[\sum\limits_{n=1}^{M}|\lambda_{n}|^{2}\right]^{1/2}+Cp^{d/2}
\left[\sum_{n=M+1}^{+\infty}|\lambda_{n}|^{2}\sum\limits_{k\in \z^{d}}
\left|\mathscr{F}_{x}f_{n}(k)\right|^{2}\right]^{\frac{1}{2}}\nonumber\\
&&\leq Cp^{d/2}\left(|t|M_{1}^{2}+\epsilon^{d+1}\right)
\left(\sum\limits_{n=1}^{M}|\lambda_{n}|^{2}\right)^{1/2}
+Cp^{d/2}\left[\sum_{n=M+1}^{+\infty}|\lambda_{n}|^{2}\right]^{1/2}\nonumber\\
&&\leq Cp^{d/2}\left(|t|M_{1}^{2}+\epsilon^{d+1}\right)
\left(\sum\limits_{n=1}^{M}|\lambda_{n}|^{2}\right)^{1/2}
+Cp^{d/2}\epsilon^{d+1}\nonumber\\
&&\leq Cp^{d/2}\left(|t|M_{1}^{2}+\epsilon^{d+1}\right)
\left(\sum\limits_{n=1}^{\infty}|\lambda_{n}|^{2}\right)^{1/2}+Cp^{d/2}
\epsilon^{d+1}\nonumber\\
&&\leq Cp^{d/2}\left[\left(|t|M_{1}^{2}+\epsilon^{d+1}\right)
\left(\sum\limits_{n=1}^{\infty}|\lambda_{n}|^{2}\right)^{1/2}+\epsilon^{d+1}\right]\nonumber\\
&&\leq Cp^{d/2}\left[\epsilon^{d+1}\left(\sum\limits_{n=1}^{\infty}|\lambda_{n}|^{2}\right)^{1/2}
+\epsilon^{d+1}\right]\nonumber\\
&&\leq Cp^{d/2}\epsilon^{d+1}
\left[\left(\sum\limits_{n=1}^{\infty}|\lambda_{n}|^{2}\right)^{1/2}+1\right]\leq Cp^{d/2}\epsilon^{d+1}
\left[\left\|\gamma_{0}\right\|_{\mathfrak{S}^{2}}+1\right].\label{6.025}
\end{eqnarray}

This completes the proof of Lemma 6.3.

\noindent{\bf Remark 8.}To establish Lemma 6.3, we pay more attention to the properties
$(\lambda_{n})_{n}\in l^{2}$ and
\begin{eqnarray*}
\sum\limits_{k\in \z^{d}}\left\|\psi(\xi-k)\mathscr{F}_{x}f_{n}\right\|_{L^{2}}^{2}
\leq \|f_{n}\|_{L^{2}}^{2}\leq  3\sum\limits_{k\in \z^{d}}
\left\|\psi(\xi-k)\mathscr{F}_{x}f_{n}\right\|_{L^{2}}^{2}
\end{eqnarray*}
which plays the key role in controlling
\begin{eqnarray*}
\hspace{-1cm}\left[\sum_{n=1}^{+\infty}|\lambda_{n}|^{2}
\sum\limits_{k\in \z^{d}}\left|\int_{\SR^{d}}\psi(\xi-k)(e^{-it|\xi|^{2}}-1)
e^{i\sum\limits_{k=1}^{d}x_{k}\xi_{k}}
\mathscr{F}_{x}f_{n}(\xi)d\xi\right|^{2}\right]^{\frac{1}{2}}.
\end{eqnarray*}

\begin{Lemma}\label{lem6.4}(Some estimates related to $pth$ moment)
Let  $f_{j}^{\omega}$  be
defined as in (\ref{6.05}). Then,  there exists $C>0$ such that
\begin{eqnarray}
&&\|f_{j}^{\omega}\|_{L_{\omega}^{p}}\leq Cp^{d/2}\|f_{j}\|_{L^{2}}=Cp^{d/2},
\|U(t)f_{j}^{\omega}\|_{L_{\omega}^{p}}\leq Cp^{d/2}\|f_{j}\|_{L^{2}}=Cp^{d/2}.\label{6.026}
\end{eqnarray}
\end{Lemma}

For the proof of Lemma 6.4, we refer the readers to Lemma 2.3 of \cite{BOP}.

\begin{Lemma}\label{lem6.5}(Stochastic continuity at zero)Let $F(t,\omega,\widetilde{\omega})$ be
a real-valued measurable function on
$[-1,1]\times\Omega\times \widetilde{\Omega}$.
    For arbitrary $p\geq2$,
$\forall \epsilon>0$, $\forall \alpha>0$, $\exists \delta>0$,  when $|t|<\delta,$
the following inequality holds
\begin{eqnarray}
(\mathbb{P}\times\widetilde{\mathbb P})\left(E_{\alpha,\>t}\right)<
\left[\frac{C\epsilon^{d+1} p^{(2d+1)/2}}{\alpha}\right]^{p},\label{6.027}
\end{eqnarray}
where $E_{\alpha,\>t}\definition \left\{(\omega,\widetilde{\omega})\in
(\Omega\times\widetilde{\Omega})||F(t,\omega,\widetilde{\omega})|>\alpha\right\}$.
Then,  we have
\begin{eqnarray}
\lim\limits_{t\longrightarrow0}(\mathbb{P}\times\widetilde{\mathbb P})
(E_{Ce\epsilon^{1/2}\left(\epsilon ln
\frac{1}{\epsilon}\right)^{(2d+1)/2},\>t})=0.\label{6.028}
\end{eqnarray}
\end{Lemma}
\noindent{\bf Proof.} Take
\begin{eqnarray}
&&p=\left(\frac{\alpha}{Ce\epsilon^{d+1}}\right)^{\frac{2}{2d+1}},
\alpha=Ce\epsilon^{1/2}\left(\epsilon ln
\frac{1}{\epsilon}\right)^{(2d+1)/2}.\label{6.029}
\end{eqnarray}
From  (\ref{6.027}), when $|t|<\epsilon^{d+1},$ we have
\begin{eqnarray}
&&(\mathbb{P}\times\widetilde{\mathbb P})\left(E_{Ce\epsilon^{1/2}\left(\epsilon ln
\frac{1}{\epsilon}\right)^{(2d+1)/2},\>t}\right)
 \leq \epsilon.\label{6.030}
\end{eqnarray}
From Lemma 6.1, we know that Lemma 6.5 is valid since
$Ce\epsilon^{1/2}\left(\epsilon ln
\frac{1}{\epsilon}\right)^{(2d+1)/2}$
is  monotonically increasing with respect to $\epsilon\in (0,\frac{1}{2}]$ and
\begin{eqnarray*}
\lim\limits_{\epsilon\longrightarrow 0^{+}}Ce\epsilon^{1/2}\left(\epsilon ln
\frac{1}{\epsilon}\right)^{(2d+1)/2}=0.
\end{eqnarray*}

This completes the proof of Lemma 6.5.

\begin{Lemma}\label{lem6.6}
Let  $F(t,\omega,\widetilde{\omega})$ be
a real-valued measurable function on
$[-1,1]\times\Omega\times \widetilde{\Omega}$ and
 $p\geq2.$ For $\forall \epsilon>0,$ $\exists\delta>0,$ when $|t|<\delta,$
  the following inequality holds
\begin{eqnarray}
A(t):=\left\|F(t,\omega,\widetilde{\omega})\right\|_{L_{\omega,\>\widetilde{\omega}}^{p}
(\Omega\times \widetilde{\Omega})}\leq Cp^{(2d+1)/2}\epsilon^{d+1}
(\|\gamma_{0}\|_{\mathfrak{S}^{2}}+1).\label{6.031}
\end{eqnarray}
Then, we have
\begin{eqnarray}
&&\hspace{-1.5cm}\lim\limits_{t\longrightarrow 0}(\mathbb{P}\times\widetilde{\mathbb P})
\left(\left\{(\omega, \widetilde{\omega})\in (\Omega\times \widetilde{\Omega})|
|F|>C(\|\gamma_{0}\|_{\mathfrak{S}^{2}}+1)e\epsilon^{1/2}\left(\epsilon ln
\frac{1}{\epsilon}\right)^{(2d+1)/2}\right\}\right)=0.\nonumber\\&&\label{6.032}
\end{eqnarray}

\end{Lemma}
\noindent {\bf Proof.} By using (\ref{6.031}), when $|t|<\delta,$ we have
 \begin{eqnarray}
(\mathbb{P}\times\widetilde{\mathbb P})\left(E_{\alpha,\>t}\right)\leq
\left[\frac{A(t)}{\alpha}\right]^{p}\leq\left[\frac{C\epsilon^{d+1}
(\|\gamma_{0}\|_{\mathfrak{S}^{2}}+1)p^{(2d+1)/2}}{\alpha}\right]^{p}.\label{6.033}
\end{eqnarray}
Here $E_{\alpha,\>t}\definition \left\{(\omega,\widetilde{\omega})\in
(\Omega\times\widetilde{\Omega})||F(t,\omega,\widetilde{\omega})|>\alpha\right\}$.
In particular, we take
$\alpha=C(\|\gamma_{0}\|_{\mathfrak{S}^{2}}+1)e\epsilon^{1/2}\left(\epsilon ln
\frac{1}{\epsilon}\right)^{(2d+1)/2}.$
Following the idea of Lemma 6.5, we obtain that Lemma 6.6 is valid.

This completes the proof of Lemma 6.6.

\begin{Theorem}\label{Theorem6.7} (Stochastic continuity at zero related to Schatten norms)
Let $d\geq 1$ and $r\in [2,\infty)$, $(f_{j})_{j=1}^{\infty}$ is an
 orthonormal system in $L^{2}(\R^{d})$.
Then, for any $\gamma_{0}\in\mathfrak{S}^{2}$ and $\forall \epsilon>0$
 which appears in Lemma 6.4,
 $\exists \delta>0(<\frac{\epsilon^{d+1}}{2M_{1}^{2}}),$ when $|t|<\delta,$
 we have
\begin{align}
\left\|\sum_{n=1}^{\infty}\lambda_{n}g^{(2)}_{n}(\widetilde{\omega})|f_{n}^{\omega}|^{2}-
\sum_{n=1}^{\infty}\lambda_{n}g^{(2)}_{n}(\widetilde{\omega})|e^{it\triangle}f_{n}^{\omega}
|^{2}\right\|_{L_{\omega,\>\widetilde{\omega}}^{r}
(\Omega\times\widetilde{\Omega})}
\leq Cr^{\frac{2d+1}{2}}\epsilon^{d+1}(\left\|\gamma_{0}\right\|_{\mathfrak{S}^{2}}+1).\label{6.034}
\end{align}
Moreover, we have
\begin{eqnarray}
&&\hspace{-1cm}\lim\limits_{t\longrightarrow0}(\mathbb{P}\times\widetilde{\mathbb P})
\left(\left\{(\omega, \widetilde{\omega})\in (\Omega\times \widetilde{\Omega})|
|F(t,\omega,\widetilde{\omega})|>C(\left\|\gamma_{0}\right\|_{\mathfrak{S}^{2}}+1)e\epsilon^{1/2}\left(\epsilon ln
\frac{1}{\epsilon}\right)^{(2d+1)/2}\right\}\right)\nonumber\\&&=0.\label{6.035}
\end{eqnarray}
Here $F(t,\omega,\widetilde{\omega})=\sum\limits_{n=1}^{\infty}
\lambda_{n}g^{(2)}_{n}(\widetilde{\omega})|f_{n}^{\omega}|^{(2)}-
\sum\limits_{n=1}^{\infty}\lambda_{n}g^{(2)}_{n}(\widetilde{\omega})
|e^{it\triangle}f_{n}^{\omega}|^{2}$ and
\begin{eqnarray*}
\rho_{\gamma_{0}^{\omega,\>\widetilde{\omega}}}=
\sum_{n=1}^{\infty}\lambda_{n}g^{(2)}_{n}(\widetilde{\omega})|f_{n}^{\omega}|^{2},
\rho_{e^{it\Delta}\gamma_{0}^{\omega,\>\widetilde{\omega}}e^{-it\Delta}}
=\sum_{n=1}^{\infty}\lambda_{n}g^{(2)}_{n}(\widetilde{\omega})|e^{it\Delta}f_{n}^{\omega}|^{2}
\end{eqnarray*}
and
$\rho_{\gamma_{0}^{\omega,\>\widetilde{\omega}}}, \rho_{e^{it\Delta}
   \gamma_{0}^{\omega,\>\widetilde{\omega}}e^{-it\Delta}}$ denote the
 density function of $\gamma_{0}^{\omega,\>\widetilde{\omega}}, e^{it\Delta}
   \gamma_{0}^{\omega,\>\widetilde{\omega}}e^{-it\Delta}$, respectively.
\end{Theorem}
\noindent{\bf Proof.} Obviously,
\begin{eqnarray}
|e^{it\triangle}f_{n}^{\omega}|^{2}-|f_{n}^{\omega}|^{2}=I_{1}+I_{2},\label{6.036}
\end{eqnarray}
where
\begin{eqnarray}
I_{1}=\left(e^{it\triangle}f_{n}^{\omega}-f_{n}^{\omega}\right)
e^{-it\triangle}\bar{f}_{n}^{\omega},
I_{2}=\left(e^{-it\triangle}
\bar{f}_{n}^{\omega}-\bar{f}_{n}^{\omega}\right)f_{n}^{\omega}.\label{6.037}
\end{eqnarray}
By using (\ref{6.036}), (\ref{6.037}), Minkowski inequality
 and the H\"older inequality as well as Lemmas 6.4, 6.5,
$\forall \epsilon>0$ which appears in Lemma 6.4,
 $\exists \delta>0(<\frac{\epsilon^{d+1}}{2M_{1}^{2}}),$ when $|t|<\delta,$
we have
\begin{eqnarray}
&&\left\|\sum_{n=1}^{\infty}\lambda_{n}g^{2}_{n}(\widetilde{\omega})|f_{n}^{\omega}|^{2}-
\sum_{n=1}^{\infty}\lambda_{n}g^{2}_{n}(\widetilde{\omega})|e^{it\triangle}f_{n}^{\omega}
|^{2}\right\|_{L_{\omega,\>\widetilde{\omega}}^{r}
(\Omega\times\widetilde{\Omega})}\nonumber\\
&&=\left\|\sum_{n=1}^{\infty}\lambda_{n}g^{2}_{n}(\widetilde{\omega})
(I_{1}+I_{2})\right\|_{L_{\omega,\>\widetilde{\omega}}^{r}
(\Omega\times\widetilde{\Omega})}\nonumber\\
&&\leq \left\|\sum_{n=1}^{\infty}\lambda_{n}g^{2}_{n}(\widetilde{\omega})
I_{1}\right\|_{L_{\omega,\>\widetilde{\omega}}^{r}
(\Omega\times\widetilde{\Omega})}+\left\|\sum_{n=1}^{\infty}\lambda_{n}g^{2}_{n}
(\widetilde{\omega})I_{2}\right\|_{L_{\omega,\>\widetilde{\omega}}^{r}
(\Omega\times\widetilde{\Omega})}\nonumber\\
&&\leq Cr^{1/2}\left[\left\|\lambda_{n}I_{1}\right\|_{L_{\omega}^{r}l_{n}^{2}}+
\left\|\lambda_{n}I_{2}\right\|_{L_{\omega}^{r}l_{n}^{2}}\right]\nonumber\\
&&\leq Cr^{1/2}\left[\left\|\lambda_{n}I_{1}\right\|_{l_{n}^{2}L_{\omega}^{r}}+
\left\|\lambda_{n}I_{2}\right\|_{l_{n}^{2}L_{\omega}^{r}}\right]\nonumber\\
&&\leq Cr^{1/2}\left[\left\|\lambda_{n}
\left\|\left(e^{it\triangle}f_{n}^{\omega}-f_{n}^{\omega}\right)
\right\|_{L_{\omega}^{2r}}\|e^{-it\triangle}
\bar{f}_{n}^{\omega}\|_{L_{\omega}^{2r}}\right\|_{l_{n}^{2}}
\right]\nonumber\\&&\qquad\qquad+Cr^{1/2}\left[\left\|\lambda_{n}
\left\|\left(e^{-it\triangle}\bar{f}_{n}^{\omega}-\bar{f}_{n}^{\omega}\right)
\right\|_{L_{\omega}^{2r}}\|f_{n}^{\omega}\|_{L_{\omega}^{2r}}
\right\|_{l_{n}^{2}}
\right]\nonumber\\
&&\leq Cr^{\frac{2d+1}{2}}(|t|M_{1}^{2}+\epsilon^{d+1})
(\left\|\gamma_{0}\right\|_{\mathfrak{S}^{2}}+1)\leq Cr^{\frac{2d+1}{2}}\epsilon^{d+1}
(\left\|\gamma_{0}\right\|_{\mathfrak{S}^{2}}+1).\label{6.038}
\end{eqnarray}
Combining (\ref{6.038}) with Lemma 6.6,
since
$Ce\epsilon^{1/2}\left(\epsilon ln
\frac{1}{\epsilon}\right)^{3/2}$
is  monotonically increasing with respect to $\epsilon\in (0,\frac{1}{2}]$ and
\begin{eqnarray*}
\lim\limits_{\epsilon\longrightarrow 0^{+}}Ce\epsilon^{1/2}\left(\epsilon ln
\frac{1}{\epsilon}\right)^{d+\frac{1}{2}}=0,
\end{eqnarray*}
we have that (\ref{6.035}) is valid.

This completes the proof of Theorem 6.1.

\noindent{\bf Remark 9.} For $\gamma_{0}\in \mathfrak{S}^{\beta}(\beta<2)$,
Bez et al. \cite{BLN}
\begin{eqnarray}
\lim\limits_{t\longrightarrow0}\rho_{\gamma(t)}(x)=\lim\limits_{t\longrightarrow0}
\sum\limits_{j=1}^{\infty}\lambda_{j}|e^{it\partial_{x}^{2}}f_{j}|^{2}
=\rho_{\gamma_{0}}(x)=\sum\limits_{j=1}^{\infty}\lambda_{j}|f_{j}|^{2},
\qquad  a.e. x\in \R,\label{6.039}
\end{eqnarray}
where $\rho_{\gamma(t)}$ is the density function of
 $\gamma(t)$ and
$\gamma(t)=\sum\limits_{j=1}^{+\infty}
\lambda_{j}|e^{it\partial_{x}^{2}}f_{j}\rangle\langle e^{it\partial_{x}^{2}}f_{j}|$
 is the solution to
the operator-valued equation
to
$\left\{
        \begin{array}{ll}
         i\frac{d\gamma(t)}{dt}=[-\partial_{x}^{2},\gamma(t)] \\
          \gamma(0)=\sum\limits_{j=1}^{\infty}\lambda_{j}|f_{j}\rangle\langle f_{j}|.
        \end{array}
      \right.$
Here $(f_{j})_{j\in \mathbf{N}^{+}}$ are orthonormal functions in $\dot{H}^{1/4}(\R)$.
Thus, when $d=1$, Theorem 6.1 improves the result of Corollary 1.2 of \cite{BLN}
in the probabilistic sense since $\gamma_{0}\in \mathfrak{S}^{2}$.
(\ref{6.035}) is called {\bf the probabilistic convergence of density functions
   of some compact operators with full randomization on $\R^{d}.$}

\bigskip
\bigskip

\section{Probabilistic convergence of density functions
   of some compact operators with full randomization on  $\T^{d}$}

\setcounter{equation}{0}

\setcounter{Theorem}{0}

\setcounter{Lemma}{0}

\setcounter{section}{7}
By using the idea of \cite{HY} and idea of    Lemma 3.2 of \cite{WYY},
we  present the probabilistic convergence of density functions
of some compact operators with full randomization on $\T^{d}$.

Combining   the  randomization of  compact operator  \cite{HY} with
  the randomization of  single function on  $\T^{d}$ \cite{CLS,YZY},
we firstly present the full randomization of some compact operators on $\T^{d}$.
Then, by using  the full randomization of some compact operators on $\T^{d}$,
we present the probabilistic convergence of density functions
   of some  compact operators with full randomization on torus.

Let $\{g^{(1)}_{k_{j}}(\omega)\}_{k_{j}\in \mathbf{N}^{+}},\{g^{(2)}_k(\widetilde{\omega})\}_{k\in \mathbf{N}^{+}}$ be
  sequences of independent, zero-mean, real-valued  random variables
 on the probability spaces $(\Omega_{j},\mathcal{A}_{j}, \mathbb{P}_{j}),$
$(\widetilde{\Omega},\widetilde{\mathcal{A}}, \widetilde{\mathbb{P}})$
  endowed with probability
distributions $\mu^{1}_{k_{j}}(1\leq j\leq d,k_{j}\in \N^{+}),\mu^{2}_n(n\in \N^{+})$
 which satisfy (\ref{6.03}), (\ref{6.04}), respectively.
Let $(\Omega,\mathcal{A}, \mathbb{P})=
(\prod\limits_{j=1}^{d}\Omega_{j},\prod\limits_{j=1}^{d}\mathcal{A}_{j},
\prod\limits_{j=1}^{d}\mathbb{P}_{j}),\omega=(\omega_{1},\cdot\cdot\cdot,\omega_{d}).
$

Now we recall the randomization of a single function on $\T^{n}$.
Let $f_{n}(n\in \mathbf{N}^{+})$ be an orthonormal system in
 $L^{2}(\T^{d}).$
For $f_{n}\in L^{2}(\T^{d})(n\in \mathbf{N}^{+})$, inspired by \cite{CLS,WYY,YZY},
 we define its randomization by
\begin{eqnarray}
&&f_{n}^{\omega}:=\sum\limits_{k=(k_{1},\>\cdot\cdot\cdot,\>k_{d})\in\z^{d}}
\prod\limits_{j=1}^{d}g^{(1)}_{k_{j}}(\omega_{j})
e^{i\sum\limits_{j=1}^{k}x_{j}k_{j}}\mathscr{F}_{x}f_{n}(k).\label{7.01}
\end{eqnarray}

\noindent(Full randomization of some compact operators.)Inspired by \cite{HY},
we present the full randomization of some compact operators.
Let $\gamma_{0}\in L^{2}(\T^{d})$ be a compact operator.
 Then we have the singular value decomposition
\begin{align}
\gamma_{0}=\sum_{n=1}^{\infty}\lambda_{n}|f_{n}\rangle\langle f_{n}|,\label{7.02}
\end{align}
where $\lambda_{n}\in \mathbf {C}$, $(f_{n})_{n=1}^{\infty}$
 is an orthonormal system in $L^{2}(\T^{d})$.
For any $(\omega,\widetilde{\omega})\in\Omega \times \widetilde{\Omega}$,
 inspired by \cite{HY} and
(\ref{7.01}),
we define the full randomization of  $\gamma_{0}
=\sum\limits_{n=1}^{+\infty}\lambda_{n}|f_{n}\rangle\langle f_{n}|$
by
\begin{eqnarray}
\gamma_{0}^{\omega,\>\widetilde{\omega}}:=
\sum_{n=1}^{\infty}a_{n}g^{(2)}_{n}(\widetilde{\omega})|f_{n}^{\omega}
\rangle\langle f_{n}^{\omega}|.\label{7.03}
\end{eqnarray}

\begin{Lemma}\label{lem7.1}(The estimate related to $l_{n}^{2}L_{\omega}^{p}$)
Let $\|f_{n}\|_{L^{2}(\ST^{d})}=1=\left[\sum\limits_{k\in \z^{d}}
\left|\mathscr{F}_{x}f_{n}(k)\right|^{2}\right]^{1/2}(n\in \mathbf{N}^{+})$,
 $(\lambda_{n})_{n}\in l^{2}$ and $f_{n}^{\omega}$
 the randomization of $f_{n}$ be
defined as in (\ref{7.01}). Then, for $\forall \epsilon>0$,
 $\exists \delta>0$,
 when $|t|<\delta<\frac{\epsilon^{d+1}}{M_{1}^{2}},$  we have
\begin{eqnarray}
&&\left\|\lambda_{n}\left\|U(t)f_{n}^{\omega}-f_{n}^{\omega}
\right\|_{L_{\omega}^{p}}\right\|_{l_{n}^{2}}\leq Cp^{d/2}\epsilon^{d+1}
\left[\left(\sum\limits_{n=1}^{\infty}\lambda_{n}^{2}\right)^{1/2}+1\right]\nonumber\\
&&= Cp^{d/2}\epsilon^{d+1}
\left[\left\|\gamma_{0}\right\|_{\mathfrak{S}^{2}}+1\right].
\label{7.04}
\end{eqnarray}
Here $M_{1}$ appears in (\ref{7.08}) and
$\left\|\gamma_{0}\right\|_{\mathfrak{S}^{2}}=
\left(\sum\limits_{n=1}^{\infty}\lambda_{n}^{2}\right)^{1/2}$.
\end{Lemma}
\noindent {\bf Proof.} From Lemma 6.2, we have
\begin{eqnarray}
&&\left\|\lambda_{n}\left\|U(t)f_{n}^{\omega}-f_{n}^{\omega}\right\|_{L_{\omega}^{p}}\right\|_{l_{n}^{2}}\leq Cp^{d/2}\left[\sum_{n=1}^{+\infty}|\lambda_{n}|^{2}\sum\limits_{k\in \z^{d}}\left|(e^{-itk^{2}}-1)\mathscr{F}_{x}f_{n}(k)\right|^{2}\right]^{\frac{1}{2}}
.\label{7.05}
\end{eqnarray}
Since $(\lambda_{n})_{n}\in l^{2}$, for $\forall \epsilon>0$,
we know that there exists $M\in \mathbf{N}^{+}$ such that
\begin{eqnarray}
\left[\sum\limits_{n=M+1}^{\infty}|\lambda_{n}|^{2}\right]^{1/2}<\epsilon^{d+1}.\label{7.06}
\end{eqnarray}
We  claim that for $1\leq n\leq M$ and $\forall \epsilon>0$,
 we have that
\begin{eqnarray}
\left[\sum\limits_{k\in \z^{d}}\left|(e^{-it|k|^{2}}-1)
\mathscr{F}_{x}f_{n}(k)\right|^{2}\right]^{1/2}\leq C
(\epsilon^{d+1}+|t||M_{1}^{2}).\label{7.07}
\end{eqnarray}
Here, $M_1$ depends on $M$.
We will use the idea of    Lemma 3.2 of \cite{WYY} to prove Lemma 7.1.
Since $1\leq n\leq M,$ for $\forall \epsilon>0$, we know that there
 exists $M_{1}$ which depends on $M$ such that
\begin{eqnarray}
\left[\sum\limits_{|k|\geq M_{1}+1}
\left|\mathscr{F}_{x}f_{n}(k)\right|^{2}\right]^{1/2}<\frac{\epsilon^{d+1}}{2}.\label{7.08}
\end{eqnarray}
By using the fact that $|e^{-it|k|^{2}}-1|\leq |t|k^{2}$, we have that
\begin{eqnarray}
\left[\sum\limits_{|k|\leq M_{1}}\left|(e^{-it|k|^{2}}-1)
\mathscr{F}_{x}f_{n}(k)\right|^{2}\right]^{1/2}\leq C|t|M_{1}^{2}.\label{7.09}
\end{eqnarray}
Combining (\ref{7.08}) with (\ref{7.09}), we derive that (\ref{7.07}) is valid.
From (\ref{7.05})-(\ref{7.07}), by using $(a+b)^{1/2}\leq a^{1/2}+b^{1/2}$
 and $|e^{-it|k|^{2}}-1|\leq2$ as well as $\|f_{n}\|_{L^{2}(\ST^{d})}=\left[\sum\limits_{k\in \z^{d}}
\left|\mathscr{F}_{x}f_{n}(k)\right|^{2}\right]^{1/2}=1(n\in \mathbf{N}^{+})$,
 when $|t|<\delta<\frac{\epsilon^{d+1}}{M_{1}^{2}},$   we know that
\begin{eqnarray}
&&\left\|\lambda_{n}\left\|U(t)f_{n}^{\omega}-f_{n}^{\omega}\right\|_{L_{\omega}^{p}}\right\|_{l_{n}^{2}}\nonumber\\
&&\leq Cp^{d/2}\left[\sum_{n=1}^{+\infty}|\lambda_{n}|^{2}\sum\limits_{k\in \z^{d}}\left|(e^{-itk^{2}}-1)\mathscr{F}_{x}f_{n}(k)\right|^{2}\right]^{\frac{1}{2}}\nonumber\\&&
\leq Cp^{d/2}\left[\sum_{n=1}^{M}|\lambda_{n}|^{2}\sum\limits_{k\in \z^{d}}\left|(e^{-itk^{2}}-1)\mathscr{F}_{x}f_{n}(k)\right|^{2}\right]^{\frac{1}{2}}\nonumber\\&&\qquad + Cp^{d/2}\left[\sum_{n=M+1}^{+\infty}|\lambda_{n}|^{2}\sum\limits_{k\in \z^{d}}\left|(e^{-itk^{2}}-1)\mathscr{F}_{x}f_{n}(k)\right|^{2}\right]^{\frac{1}{2}}\nonumber\\
&&\leq Cp^{d/2}\left(|t|M_{1}^{2}+\epsilon^{d+1}\right)
\left[\sum\limits_{n=1}^{M}|\lambda_{n}|^{2}\right]^{1/2}+Cp^{d/2}\left[\sum_{n=M+1}^{+\infty}|\lambda_{n}|^{2}\sum\limits_{k\in \z^{d}}\left|\mathscr{F}_{x}f_{n}(k)\right|^{2}\right]^{\frac{1}{2}}\nonumber\\
&&\leq Cp^{d/2}\left(|t|M_{1}^{2}+\epsilon^{d+1}\right)
\left(\sum\limits_{n=1}^{M}|\lambda_{n}|^{2}\right)^{1/2}+Cp^{d/2}
\left[\sum_{n=M+1}^{+\infty}|\lambda_{n}|^{2}\right]^{1/2}\nonumber\\
&&\leq Cp^{d/2}\left(|t|M_{1}^{2}+\epsilon^{d+1}\right)
\left(\sum\limits_{n=1}^{M}|\lambda_{n}|^{2}\right)^{1/2}+Cp^{d/2}\epsilon^{d+1}\nonumber\\
&&\leq Cp^{d/2}\left(|t|M_{1}^{2}+\epsilon^{d+1}\right)
\left(\sum\limits_{n=1}^{\infty}|\lambda_{n}|^{2}\right)^{1/2}+Cp^{d/2}\epsilon^{d+1}\nonumber\\
&&\leq Cp^{d/2}\left[\left(|t|M_{1}^{2}+\epsilon^{d+1}\right)
\left(\sum\limits_{n=1}^{\infty}|\lambda_{n}|^{2}\right)^{1/2}+\epsilon^{d+1}\right]\nonumber\\
&&\leq Cp^{d/2}\left[\epsilon^{d+1}
\left(\sum\limits_{n=1}^{\infty}|\lambda_{n}|^{2}\right)^{1/2}+\epsilon^{d+1}\right]\nonumber\\
&&\leq Cp^{d/2}\epsilon^{d+1}
\left[\left(\sum\limits_{n=1}^{\infty}|\lambda_{n}|^{2}\right)^{1/2}+1\right]\leq Cp^{d/2}\epsilon^{d+1}
\left[\left\|\gamma_{0}\right\|_{\mathfrak{S}^{2}}+1\right].\label{7.010}
\end{eqnarray}

This completes the proof of Lemma 7.1.

\noindent{\bf Remark 10.}
To establish Lemma 7.1, we pay more attention to the properties
$(\lambda_{n})_{n}\in l^{2}$ and
\begin{eqnarray}
\left[\sum\limits_{k\in \z^{d}}\left|(e^{-it|k|^{2}}-1)
\mathscr{F}_{x}f_{n}(k)\right|^{2}\right]^{1/2}\leq 2\left[\sum\limits_{k\in \z^{d}}\left|
\mathscr{F}_{x}f_{n}(k)\right|^{2}\right]^{1/2}=2.\label{7.011}
\end{eqnarray}
which play the key   role in controlling
\begin{eqnarray*}
\left[\sum_{n=1}^{+\infty}|\lambda_{n}|^{2}
\sum\limits_{k\in \z^{d}}\left|(e^{-itk^{2}}-1)
\mathscr{F}_{x}f_{n}(k)\right|^{2}\right]^{\frac{1}{2}}.
\end{eqnarray*}

\begin{Lemma}\label{lem7.2}(Some estimates related to $pth$ moment)
Let $\|f_{n}\|_{L^{2}(\ST^{d})}=1(n\in \mathbf{N}^{+})$ and $f_{n}^{\omega}$
 the randomization of $f_{n}$ be
defined as in (\ref{7.01}). Then,  there exists $C>0$ such that
\begin{eqnarray}
&&\|f_{n}^{\omega}\|_{L_{\omega}^{p}}\leq Cp^{d/2}\|f_{n}\|_{L^{2}}=Cp^{d/2},
\|U(t)f_{n}^{\omega}\|_{L_{\omega}^{p}}\leq Cp^{d/2}\|f_{n}\|_{L^{2}}=Cp^{d/2}.\label{7.012}
\end{eqnarray}
\end{Lemma}
\noindent{\bf Proof.}Since
\begin{eqnarray*}
&&f_{n}^{\omega}:=\sum\limits_{k=(k_{1},\>\cdot\cdot\cdot,\>k_{d})\in\z^{d}}
\prod\limits_{j=1}^{d}g^{(1)}_{k_{j}}(\omega_{j})
e^{i\sum\limits_{j=1}^{k}x_{j}k_{j}}\mathscr{F}_{x}f_{n}(k),
\end{eqnarray*}
by using Lemma 6.2, we have
\begin{eqnarray*}
\left\|f_{n}^{\omega}\right\|_{L_{\omega}^{p}}\leq Cp^{d/2}
\left[\sum\limits_{k=(k_{1},\>\cdot\cdot\cdot,\>k_{d})\in\z^{d}}
|\mathscr{F}_{x}f_{n}(k)|^{2}\right]^{1/2}=Cp^{d/2}.
\end{eqnarray*}
Since
\begin{eqnarray*}
&&U(t)f_{n}^{\omega}:=\sum\limits_{k=(k_{1},\>\cdot\cdot\cdot,\>k_{d})\in\z^{d}}e^{-it|k|^{2}}
\prod\limits_{j=1}^{d}g^{(1)}_{k_{j}}(\omega_{j})
e^{i\sum\limits_{j=1}^{k}x_{j}k_{j}}\mathscr{F}_{x}f_{n}(k),
\end{eqnarray*}
by using Lemma 6.2, we have
\begin{eqnarray*}
\left\|U(t)f_{n}^{\omega}\right\|_{L_{\omega}^{p}}\leq
 Cp^{d/2}\left[\sum\limits_{k=(k_{1},\>\cdot\cdot\cdot,\>k_{d})\in\z^{d}}
|\mathscr{F}_{x}f_{n}(k)|^{2}\right]^{1/2}=Cp^{d/2}.
\end{eqnarray*}

This completes the proof of Lemma 7.2.

\begin{Theorem}\label{7.3}(Stochastic continuity at zero related to Schatten norm)
 Let $\|f_{n}\|_{L^{2}(\ST^{d})}=1(n\in \mathbf{N}^{+})$  and $r\in [2,\infty)$.
Then, for any $\gamma_{0}\in\mathfrak{S}^{2}$ and $\forall \epsilon>0$
 which appears in Lemma 7.1,
 $\exists \delta>0,$ when $|t|<\delta<\frac{\epsilon^{d+1}}{M_{1}^{2}},$
 we have
\begin{eqnarray}
&&\left\|\sum_{n=1}^{\infty}\lambda_{n}g^{(2)}_{n}(\widetilde{\omega})|f_{n}^{\omega}|^{2}-
\sum_{n=1}^{\infty}\lambda_{n}g^{(2)}_{n}(\widetilde{\omega})|e^{it\triangle}f_{n}^{\omega}
|^{2}\right\|_{L_{\omega,\>\widetilde{\omega}}^{r}
(\Omega\times\widetilde{\Omega})}
\nonumber\\&&\leq Cr^{\frac{(2d+1)}{2}}\epsilon^{d+1}
\left[\left\|\gamma_{0}\right\|_{\mathfrak{S}^{2}}+1\right].\label{7.013}
\end{eqnarray}
Moreover, we have
\begin{eqnarray}
&&\hspace{-1cm}\lim\limits_{t\longrightarrow0}(\mathbb{P}\times\widetilde{\mathbb P})
\left(\left\{(\omega, \widetilde{\omega})\in (\Omega\times \widetilde{\Omega})|
|F(t,\omega,\widetilde{\omega})|>C
\left[\left\|\gamma_{0}\right\|_{\mathfrak{S}^{2}}+1\right]e\epsilon^{1/2}\left(\epsilon ln
\frac{1}{\epsilon}\right)^{(2d+1)/2}\right\}\right)\nonumber\\&&=0.
\label{7.014}
\end{eqnarray}
Here $F(t,\omega,\widetilde{\omega})=\sum
\limits_{n=1}^{\infty}\lambda_{n}g^{(2)}_{n}(\widetilde{\omega})|f_{n}^{\omega}|^{2}-
\sum\limits_{n=1}^{\infty}\lambda_{n}g^{(2)}_{n}(\widetilde{\omega})|e^{it\triangle}f_{n}^{\omega}|^{2}$
and $M_{1}$ appears in (\ref{7.08}) and
\begin{eqnarray*}
\rho_{\gamma_{0}^{\omega,\>\widetilde{\omega}}}=
\sum_{n=1}^{\infty}\lambda_{n}g^{(2)}_{n}(\widetilde{\omega})|f_{n}^{\omega}|^{2},
\rho_{e^{it\Delta}\gamma_{0}^{\omega,\>\widetilde{\omega}}e^{-it\Delta}}
=\sum_{n=1}^{\infty}\lambda_{n}g^{(2)}_{n}(\widetilde{\omega})|e^{it\Delta}f_{n}^{\omega}|^{2}.
\end{eqnarray*}
and
  $\rho_{\gamma_{0}^{\omega,\>\widetilde{\omega}}}, \rho_{e^{it\Delta}
   \gamma_{0}^{\omega,\>\widetilde{\omega}}e^{-it\Delta}}$ denote the
 density function of $\gamma_{0}^{\omega,\>\widetilde{\omega}}, e^{it\Delta}
   \gamma_{0}^{\omega,\>\widetilde{\omega}}e^{-it\Delta}$, respectively.

\end{Theorem}
\noindent{\bf Proof.} Obviously,
\begin{eqnarray}
|e^{it\triangle}f_{n}^{\omega}|^{2}-|f_{n}^{\omega}|^{2}=I_{1}+I_{2},\label{7.015}
\end{eqnarray}
where
\begin{eqnarray}
I_{1}=\left(e^{it\triangle}f_{n}^{\omega}-f_{n}^{\omega}\right)e^{-it\triangle}\bar{f}_{n}^{\omega},
I_{2}=\left(e^{-it\triangle}\bar{f}_{n}^{\omega}-\bar{f}_{n}^{\omega}\right)f_{n}^{\omega}.\label{7.016}
\end{eqnarray}
By using (\ref{7.015}), (\ref{7.016}),   the Minkowski inequality and the
 H\"older inequality as well as Lemmas 7.1, 7.2,
for $\forall \epsilon>0$ which appears in Lemma 7.1,
 $\exists \delta>0(<\frac{\epsilon^{d+1}}{M_{1}^{2}}),$ when $|t|<\delta,$
we have
\begin{eqnarray}
&&\left\|\sum_{n=1}^{\infty}\lambda_{n}g^{(2)}_{n}(\widetilde{\omega})|f_{n}^{\omega}|^{2}-
\sum_{n=1}^{\infty}\lambda_{n}g^{(2)}_{n}(\widetilde{\omega})|e^{it\triangle}f_{n}^{\omega}
|^{2}\right\|_{L_{\omega,\>\widetilde{\omega}}^{r}
(\Omega\times\widetilde{\Omega})}\nonumber\\
&&=\left\|\sum_{n=1}^{\infty}\lambda_{n}g^{(2)}_{n}(\widetilde{\omega})
(I_{1}+I_{2})\right\|_{L_{\omega,\>\widetilde{\omega}}^{r}
(\Omega\times\widetilde{\Omega})}\nonumber\\
&&\leq \left\|\sum_{n=1}^{\infty}\lambda_{n}g^{(2)}_{n}(\widetilde{\omega})
I_{1}\right\|_{L_{\omega,\>\widetilde{\omega}}^{r}
(\Omega\times\widetilde{\Omega})}+
\left\|\sum_{n=1}^{\infty}\lambda_{n}g^{(2)}_{n}(\widetilde{\omega})
I_{2}\right\|_{L_{\omega,\>\widetilde{\omega}}^{r}
(\Omega\times\widetilde{\Omega})}\nonumber\\
&&\leq Cr^{1/2}\left[\left\|\lambda_{n}I_{1}\right\|_{L_{\omega}^{r}l_{n}^{2}}+
\left\|\lambda_{n}I_{2}\right\|_{L_{\omega}^{r}l_{n}^{2}}\right]\nonumber\\
&&\leq Cr^{1/2}\left[\left\|\lambda_{n}I_{1}\right\|_{l_{n}^{2}L_{\omega}^{r}}+
\left\|\lambda_{n}I_{2}\right\|_{l_{n}^{2}L_{\omega}^{r}}\right]\nonumber\\
&&\leq Cr^{1/2}\left[\left\|\lambda_{n}
\left\|\left(e^{it\triangle}f_{n}^{\omega}-f_{n}^{\omega}\right)
\right\|_{L_{\omega}^{2r}}\|e^{-it\triangle}\bar{f}_{n}^{\omega}\|_{L_{\omega}^{2r}}\right\|_{l_{n}^{2}}
\right]\nonumber\\&&\qquad\qquad+Cr^{1/2}\left[\left\|\lambda_{n}
\left\|\left(e^{-it\triangle}\bar{f}_{n}^{\omega}-\bar{f}_{n}^{\omega}\right)
\right\|_{L_{\omega}^{2r}}\|f_{n}^{\omega}\|_{L_{\omega}^{2r}}\right\|_{l_{n}^{2}}
\right]\nonumber\\
&&\leq Cr^{\frac{(2d+1)}{2}}\left[\left(|t|M_{1}^{2}+\epsilon^{d+1}\right)
\left(\sum\limits_{n=1}^{\infty}|\lambda_{n}|^{2}\right)^{1/2}+\epsilon^{d+1}\right]
\nonumber\\
&&\leq Cr^{\frac{(2d+1)}{2}}\left[\left(|t|M_{1}^{2}+\epsilon^{d+1}\right)
\left\|\gamma_{0}\right\|_{\mathfrak{S}^{2}}+\epsilon^{d+1}\right]\nonumber\\
&&\leq Cr^{\frac{(2d+1)}{2}}\epsilon^{d+1}
\left[\left\|\gamma_{0}\right\|_{\mathfrak{S}^{2}}+1\right].\label{7.017}
\end{eqnarray}
Combining  (\ref{7.017}) with Lemma 6.6,  since
$Ce\epsilon^{1/2}\left(\epsilon ln
\frac{1}{\epsilon}\right)^{(2d+1)/2}$
is  monotonically increasing with respect to $\epsilon\in (0,\frac{1}{2}]$ and
\begin{eqnarray*}
\lim\limits_{\epsilon\longrightarrow 0^{+}}Ce\epsilon^{1/2}\left(\epsilon ln
\frac{1}{\epsilon}\right)^{d+\frac{1}{2}}=0,
\end{eqnarray*}
we have that (\ref{7.014}) is valid.

This completes the proof of Theorem 7.1.

\noindent{\bf Remark 11.} (\ref{7.014}) is called {\bf the probabilistic
 convergence of density functions
   of some compact operators with full randomization on $\T^{d}.$}

\bigskip
\bigskip

\section{Probabilistic convergence of density functions
  of some compact operators with full randomization on $\Theta=\left\{x\in \R^{3}|
  \sum\limits_{j=1}^{3}x_{j}^{2}<1\right\}$}

\setcounter{equation}{0}

\setcounter{Theorem}{0}

\setcounter{Lemma}{0}

\setcounter{section}{8}
By using the idea of  \cite{HY,BT} and  the idea of    Lemmas 3.2, 3.3 of \cite{WYY},
we  present the probabilistic convergence of density functions
of some compact operators with full randomization on $\Theta=\left\{x\in \R^{3}:|x|<1\right\}.$

Combining   the  randomization of some  compact operators  \cite{HY} with
  the randomization of  single function on  $\Theta=\left\{x\in \R^{3}:|x|<1\right\}$ \cite{BT},
we firstly present the full randomization of  compact operator on
 the three dimensional unit ball $\Theta=\left\{x\in \R^{3}:|x|<1\right\}$.
Then, by using  the full randomization of some compact operators on the
 three  dimensional unit ball $\Theta=\left\{x\in \R^{3}:|x|<1\right\}$,
we present the probabilistic convergence of density functions
   of some compact operators with full randomization on the three
 dimensional unit ball $\Theta=\left\{x\in \R^{3}:|x|<1\right\}$.

Now we recall the randomization of a single function on $\Theta$.
Let $f_{n}(n\in N^{+})$ be an orthonormal system in
 $L^{2}(\Theta),\Theta=\left\{x\in \R^{3}:|x|<1\right\}$
  and be radial. Then, we have
$f_{n}=\sum\limits_{m=1}^{+\infty}c_{n,\>m}e_{m}$ and
$e^{it\Delta}f_{n}=\sum\limits_{m=1}^{+\infty}e^{-it(m\pi)^{2}}c_{n,\>m}e_{m},$
where $e_{k}=\frac{sin k\pi |x|}{(2\pi)^{1/2}|x|},c_{n,\>m}=\int_{\Theta}f_{n}e_{m}dx$.
Moreover $e_n$ are the radial eigenfunctions of the Laplace operator
 $-\Delta$ with Dirichlet
boundary conditions, associated to eigenvalues $\lambda_{n}^{2} =(n\pi)^2$
 \cite{BT2007,WYY}.
 For any $\omega_{1}\in\Omega_{1}$, inspired by \cite{BT},
we define its randomization by
\begin{eqnarray}
f_{n}^{\omega_{1}}=
\sum\limits_{m=1}^{+\infty}g^{(1)}_m(\omega_{1})\frac{c_{n,\>m}}{m\pi}e_{m}.\label{8.01}
\end{eqnarray}

\noindent(Full randomization of some compact operators.)Inspired by \cite{HY}, now
 we present the full randomization of some compact operators.
Let $\gamma_{0}$ be  a compact operator on $L^{2}(\Theta)$.
 Then we have the singular value decomposition
\begin{align}
\gamma_{0}=\sum\limits_{n=1}^{\infty}\lambda_{n}|f_{n}\rangle\langle f_{n}|,\label{8.02}
\end{align}
where $\lambda_{n}\in \mathbf{C}$, $(f_{n})_{n=1}^{\infty}$
 is orthonormal system in $L^{2}(\Theta)$.
For any $(\omega_{1},\widetilde{\omega})\in\Omega_{1} \times \widetilde{\Omega}$,
inspired by \cite{HY} and
(\ref{8.01}),
we define the full randomization of  $\gamma_{0}
=\sum\limits_{n=1}^{+\infty}\lambda_{n}|f_{n}\rangle\langle f_{n}|$
by
\begin{eqnarray}
\gamma_{0}^{\omega_{1},\>\widetilde{\omega}}:=
\sum_{n=1}^{\infty}\lambda_{n}g^{(2)}_{n}(\widetilde{\omega})|f_{n}^{\omega_{1}}
\rangle\langle f_{n}^{\omega_{1}}|.\label{8.03}
\end{eqnarray}

\begin{Lemma}(Convergence in measure) \label{Lemma8.1}
Let $F(t,\omega_{1},\widetilde{\omega})$ be
a real-valued measurable function on
$[-1,1]\times\Omega_{1}\times \widetilde{\Omega}$.
For all $\epsilon>0$,
\begin{align}
\lim\limits_{t\rightarrow t_{0}}E(A_{\omega_{1},\widetilde{\omega},\>\phi(\epsilon)})=0,
\label{8.04}
\end{align}
where $A_{\omega_{1},\widetilde{\omega},\>\phi(\epsilon)}=
\left\{(\omega_{1},\widetilde{\omega})\in (\Omega_{1}\times\widetilde{\Omega})||f(t,\omega_{1},\widetilde{\omega})-
f(t_{0},\omega_{1},\widetilde{\omega})|>\phi(\epsilon)\right\},$
$\phi>0$ is  monotonically increasing and
\begin{eqnarray*}
\lim\limits_{\epsilon\longrightarrow 0^{+}}\phi(\epsilon)=0.
\end{eqnarray*}
Then, (\ref{8.04}) is equivalent to the following statement: for all
 $\epsilon>0$, $\exists \delta>0$, when $|t-t_{0}|<\delta,$ we have
\begin{eqnarray}
E(A_{\omega_{1},\widetilde{\omega},\>\phi(\epsilon)})<\epsilon.\label{8.05}
\end{eqnarray}
\end{Lemma}

Lemma 8.1 can be proved similarly to Lemma 6.1.

\begin{Lemma}\label{lem8.2}(The estimate related to $l_{n}^{2}L_{\omega_{1}}^{p}$)
Let  $f_{n}^{\omega_{1}}$
  be
defined as in (\ref{8.01}). Then, for $\forall \epsilon>0,$
  $\exists \delta>0$, when $|t|<\delta(<\frac{\epsilon^{2}}{2M_{1}^{2}}),$  we have
\begin{eqnarray}
\left\|\lambda_{n}\left\|U(t)f_{n}^{\omega_{1}}-f_{n}^{\omega_{1}}\right\|_{L_{\omega_{1}}^{p}}\right\|_{l_{n}^{2}}\leq Cp^{1/2}\epsilon^{2}\left[\left\|\gamma_{0}\right\|_{\mathfrak{S}^{2}}+1\right].\label{8.06}
\end{eqnarray}
Here $M_{1}$ appears in (\ref{8.010}).
\end{Lemma}
\noindent {\bf Proof.} From Lemma 6.2, we have
\begin{eqnarray}
&&\left\|\lambda_{n}\left\|U(t)f_{n}^{\omega_{1}}-f_{n}^{\omega_{1}}\right\|_{L_{\omega_{1}}^{p}}\right\|_{l_{n}^{2}}\nonumber\\
&&\leq Cp^{1/2}\left[\sum_{n=1}^{+\infty}|\lambda_{n}|^{2}\sum\limits_{k\in \z}\left|(e^{-it(k\pi)^{2}}-1)\frac{c_{n,k}e_{k}}{k\pi}\right|^{2}\right]^{\frac{1}{2}}
.\label{8.07}
\end{eqnarray}
Since $(\lambda_{n})_{n}\in l^{2}$, for $\forall \epsilon>0$, we know that
 there exists $M\in \mathbf{N}^{+}$ such that
\begin{eqnarray}
\left[\sum\limits_{n=M+1}^{\infty}|\lambda_{n}|^{2}\right]^{1/2}<\epsilon^{2}.\label{8.08}
\end{eqnarray}
We  claim that for $1\leq n\leq M$, $\forall \epsilon>0$,  we have that
\begin{eqnarray}
\left[\sum\limits_{k\in \z}\left|(e^{-it(k\pi)^{2}}-1)
\frac{c_{n,k}e_{k}}{k\pi}\right|^{2}\right]^{1/2}\leq C(\epsilon^{2}+|t||M_{1}^{2}).\label{8.09}
\end{eqnarray}
Here, $M_1$ depends on $M$.
We will use the idea of    Lemmas 3.2, 3.3 of \cite{WYY} to prove Lemma 8.1.
Since $1\leq n\leq M,$ for $\forall \epsilon>0$, we know that there exists $M_{1}$
 which depends on $M$ such that
\begin{eqnarray}
\left[\sum\limits_{|k|\geq M_{1}+1}\left|c_{n,k}\right|^{2}\right]^{1/2}<
\frac{\epsilon^{2}}{2}.\label{8.010}
\end{eqnarray}
By using the fact that $|\frac{e_{k}}{k\pi}|\leq C$ and $|e^{-it(k\pi)^{2}}-1|\leq 2$,
\begin{eqnarray}
\left[\sum\limits_{|k|\geq M_{1}+1}\left|(e^{-it(k\pi)^{2}}-1)
\frac{c_{k}e_{k}}{k\pi}\right|^{2}\right]^{1/2}\leq
\left[\sum\limits_{|k|\geq M_{1}+1}\left|c_{n,k}\right|^{2}\right]^{1/2}
\leq \frac{\epsilon^{2}}{2}.\label{8.011}
\end{eqnarray}
By using the fact that $|\frac{e_{k}}{k\pi}|\leq C$
 and $|e^{-it(k\pi)^{2}}-1|\leq |t||k|^{2}$, we have that
\begin{eqnarray}
\left[\sum\limits_{|k|\leq M_{1}}\left|(e^{-it|k\pi|^{2}}-1)
\frac{c_{n,k}e_{k}}{k\pi}\right|^{2}\right]^{1/2}\leq C|t|M_{1}^{2}.\label{8.012}
\end{eqnarray}
Combining (\ref{8.011}) with (\ref{8.012}), we derive that (\ref{8.010}) is valid.
From (\ref{8.08})-(\ref{8.010}), by using $(a+b)^{1/2}\leq a^{1/2}+b^{1/2}$,
 $|\frac{e_{k}}{k\pi}|\leq C$ and $|e^{-it|k\pi|^{2}}-1|\leq2$ as well as
  $\|f_{n}\|_{L^{2}(\Theta)}=1=\left[\sum\limits_{k\in \z}
  \left|c_{n,k}\right|^{2}\right]^{1/2}(n\in \mathbf{N}^{+})$,  $\exists \delta>0$,  when
   $|t|<\delta<\frac{\epsilon^{2}}{M_{1}^{2}},$   we know that
\begin{eqnarray}
&&\left\|\lambda_{n}\left\|U(t)f_{n}^{\omega_{1}}-f_{n}^{\omega_{1}}\right\|_{L_{\omega_{1}}^{p}}
\right\|_{l_{n}^{2}}\nonumber\\
&&\leq Cp^{1/2}\left[\sum_{n=1}^{+\infty}|\lambda_{n}|^{2}\sum\limits_{k\in \z}\left|(e^{-it|k\pi|^{2}}-1)\frac{c_{n,k}e_{k}}{k\pi}\right|^{2}\right]^{\frac{1}{2}}\nonumber\\&&
\leq Cp^{1/2}\left[\sum_{n=1}^{M}|\lambda_{n}|^{2}\sum\limits_{k\in \z}\left|(e^{-it|k\pi|^{2}}-1)\frac{c_{n,k}e_{k}}{k\pi}\right|^{2}\right]^{\frac{1}{2}}\nonumber\\&&\qquad + Cp^{1/2}\left[\sum_{n=M+1}^{+\infty}|\lambda_{n}|^{2}\sum\limits_{k\in \z}\left|(e^{-it|k\pi|^{2}}-1)\frac{c_{n,k}e_{k}}{k\pi}\right|^{2}\right]^{\frac{1}{2}}\nonumber\\
&&\leq Cp^{1/2}\left(|t|M_{1}^{2}+\epsilon^{2}\right)\left[\sum\limits_{n=1}^{M}
|\lambda_{n}|^{2}\right]^{1/2}
+Cp^{1/2}\left[\sum_{n=M+1}^{+\infty}|\lambda_{n}|^{2}\sum\limits_{k\in \z}
\left|c_{n,k}\right|^{2}\right]^{\frac{1}{2}}\nonumber\\
&&\leq Cp^{1/2}\left(|t|M_{1}^{2}+\epsilon^{2}\right)\left(\sum\limits_{n=1}^{M}
|\lambda_{n}|^{2}\right)^{1/2}
+Cp^{1/2}\left[\sum_{n=M+1}^{+\infty}|\lambda_{n}|^{2}\right]^{1/2}\nonumber\\
&&\leq Cp^{1/2}\left(|t|M_{1}^{2}+\epsilon^{2}\right)\left(\sum\limits_{n=1}^{M}
|\lambda_{n}|^{2}\right)^{1/2}+Cp^{1/2}\epsilon^{2}\nonumber\\
&&\leq Cp^{1/2}\left(|t|M_{1}^{2}+\epsilon^{2}\right)\left(\sum\limits_{n=1}^{\infty}
|\lambda_{n}|^{2}\right)^{1/2}+Cp^{1/2}\epsilon^{2}\nonumber\\
&&\leq Cp^{1/2}\left[\left(|t|M_{1}^{2}+\epsilon^{2}\right)\left(\sum\limits_{n=1}^{\infty}
|\lambda_{n}|^{2}\right)^{1/2}+\epsilon^{2}\right]\nonumber\\
&&\leq Cp^{1/2}\left[\epsilon^{2}\left(\sum\limits_{n=1}^{\infty}|\lambda_{n}|^{2}\right)^{1/2}+\epsilon^{2}\right]\nonumber\\
&&\leq Cp^{1/2}\epsilon^{2}\left[\left(\sum\limits_{n=1}^{\infty}|\lambda_{n}|^{2}\right)^{1/2}+1\right]\leq Cp^{1/2}\epsilon^{2}\left[\left\|\gamma_{0}\right\|_{\mathfrak{S}^{2}}+1\right].\label{8.013}
\end{eqnarray}

This completes the proof of Lemma 8.2.

\noindent{\bf Remark 12.}
To establish Lemma 8.1, we pay more attention to the properties
$(\lambda_{n})_{n}\in l^{2}$ and
\begin{eqnarray}
\left[\sum\limits_{k\in \z}\left|(e^{-it(k\pi)^{2}}-1)
\frac{c_{n,k}e_{k}}{k\pi}\right|^{2}\right]^{1/2}\leq 2\left[\sum\limits_{k\in \z}
\left|c_{n,k}\right|^{2}\right]^{1/2}=2.\label{8.014}
\end{eqnarray}
which play the key   role in controlling
\begin{eqnarray*}
\left[\sum_{n=1}^{+\infty}|\lambda_{n}|^{2}\sum\limits_{k\in \z}\left|(e^{-it(k\pi)^{2}}-1)
\frac{c_{n,k}e_{k}}{k\pi}\right|^{2}\right]^{\frac{1}{2}}.
\end{eqnarray*}

\begin{Lemma}\label{lem8.3}(Some estimates related to $pth$ moment)
Let  $f_{n}^{\omega_{1}}$ be
defined as in (\ref{8.01}). Then,  there exists $C>0$ such that
\begin{eqnarray}
&&\|f_{n}^{\omega_{1}}\|_{L_{\omega_{1}}^{p}}\leq Cp^{1/2}\|f_{n}\|_{L^{2}}=Cp^{1/2},
\|U(t)f_{n}^{\omega_{1}}\|_{L_{\omega_{1}}^{p}}\leq Cp^{1/2}\|f_{n}\|_{L^{2}}=Cp^{1/2}.\label{8.015}
\end{eqnarray}
\end{Lemma}
\noindent{\bf Proof.} Since
\begin{eqnarray*}
f_{n}^{\omega_{1}}=\sum\limits_{m=1}^{+\infty}g^{(1)}_m(\omega_{1})\frac{c_{n,\>m}}{m\pi}e_{m},
\end{eqnarray*}
by using Lemma 6.2 and $|\frac{e_{m}}{m\pi}|\leq C$ and $\|f_{n}\|_{L^{2}(\Theta)}
=\left[\sum\limits_{m=1}^{\infty}
\left|c_{n,\>m}\right|^{2}\right]^{1/2}=1(n\in \N^{+})$, we have
\begin{eqnarray*}
&&\|f_{n}^{\omega_{1}}\|_{L_{\omega_{1}}^{p}}\leq Cp^{1/2}
\left[\sum\limits_{m=1}^{\infty}|\frac{c_{n,\>m}}{m\pi}e_{m}|^{2}\right]^{1/2}
\leq Cp^{1/2}\left[\sum\limits_{m=1}^{\infty}|c_{n,\>m}|^{2}\right]^{1/2}=Cp^{1/2}.
\end{eqnarray*}
Since
$e^{it\Delta}f_{n}^{\omega_{1}}=\sum\limits_{m=1}^{+\infty}g^{(1)}_m(\omega_{1})e^{-it(m\pi)^{2}}c_{n,\>m}e_{m},$
by using Lemma 6.2 and $|\frac{e_{m}}{m\pi}|\leq C$
and $\|f_{n}\|_{L^{2}(\Theta)}=\left[\sum\limits_{m=1}^{\infty}|c_{n,\>m}|^{2}\right]^{1/2}
=1(n\in \mathbf{N}^{+})$ as well as $|e^{-it(m\pi)^{2}}|=1$, we have
\begin{eqnarray*}
&&\|f_{n}^{\omega_{1}}\|_{L_{\omega_{1}}^{p}}\leq Cp^{1/2}
\left[\sum\limits_{m=1}^{\infty}\left|\frac{c_{n,\>m}}{m\pi}e_{m}\right|^{2}\right]^{1/2}
\leq Cp^{1/2}\left[\sum\limits_{m=1}^{\infty}|c_{n,\>m}|^{2}\right]^{1/2}=Cp^{1/2}.
\end{eqnarray*}

This completes the proof of Lemma 8.3.

\begin{Lemma}\label{lem8.4}(Stochastic continuity at zero)Let $F(t,\omega_{1},\widetilde{\omega})$ be
a real-valued measurable function on
$[-1,1]\times\Omega_{1}\times \widetilde{\Omega}$.
    For arbitrary $p\geq2$,
$\forall \epsilon>0$, $\forall \alpha>0$, $\exists \delta>0$,  when $|t|<\delta,$
the following inequality holds
\begin{eqnarray}
(\mathbb{P}\times\widetilde{\mathbb P})\left(E_{\alpha,\>t}\right)<
\left[\frac{C\epsilon^{d+1} p^{(2d+1)/2}}{\alpha}\right]^{p},\label{8.016}
\end{eqnarray}
where $E_{\alpha,\>t}\definition \left\{(\omega_{1},\widetilde{\omega})\in
(\Omega_{1}\times\widetilde{\Omega})||F(t,\omega_{1},\widetilde{\omega})|>\alpha\right\}$.
Then,  we have
\begin{eqnarray}
\lim\limits_{t\longrightarrow0}(\mathbb{P}\times\widetilde{\mathbb P})
(E_{Ce\epsilon^{1/2}\left(\epsilon ln
\frac{1}{\epsilon}\right)^{(2d+1)/2},\>t})=0.\label{8.017}
\end{eqnarray}
\end{Lemma}

Lemma 8.4 can be proved similarly to Lemma 6.5.

\begin{Lemma}\label{lem8.5}
Let  $F(t,\omega_{1},\widetilde{\omega})$ be
a real-valued measurable function on
$[-1,1]\times\Omega_{1}\times \widetilde{\Omega}$ and
 $p\geq2.$ For $\forall \epsilon>0,$ $\exists\delta>0,$ when $|t|<\delta,$
  the following inequality holds
\begin{eqnarray}
A(t):=\left\|F(t,\omega_{1},\widetilde{\omega})\right\|_{L_{\omega_{1},\>\widetilde{\omega}}^{p}
(\Omega_{1}\times \widetilde{\Omega})}\leq Cp^{(2d+1)/2}\epsilon^{d+1}
(\|\gamma_{0}\|_{\mathfrak{S}^{2}}+1).\label{8.018}
\end{eqnarray}
Then, we have
\begin{eqnarray}
&&\hspace{-1.5cm}\lim\limits_{t\longrightarrow 0}(\mathbb{P}\times\widetilde{\mathbb P})
\left(\left\{(\omega_{1}, \widetilde{\omega})\in (\Omega_{1}\times \widetilde{\Omega})|
|F|>C(\|\gamma_{0}\|_{\mathfrak{S}^{2}}+1)e\epsilon^{1/2}\left(\epsilon ln
\frac{1}{\epsilon}\right)^{(2d+1)/2}\right\}\right)=0.\nonumber\\&&\label{8.019}
\end{eqnarray}

\end{Lemma}

 Lemma 8.5 can be proved similarly to Lemma 6.6.

\begin{Theorem}\label{Theorem8.1}(Stochastic continuity at
zero related to Schatten norms) Let  $r\in [2,\infty)$.
Then, for any $\gamma_{0}\in\mathfrak{S}^{2}$ and $\forall \epsilon>0$
 which appears in Lemma 8.1,
 $\exists \delta>0(<\epsilon^{2} M_{1}^{-2}),$ when $|t|<\delta,$
 we have
\begin{align}
\left\|\sum_{n=1}^{\infty}\lambda_{n}g^{(2)}_{n}(\widetilde{\omega})|f_{n}^{\omega_{1}}|^{2}-
\sum_{n=1}^{\infty}\lambda_{n}g^{(2)}_{n}(\widetilde{\omega})|e^{it\triangle}f_{n}^{\omega_{1}}
|\right\|_{L_{\omega_{1},\>\widetilde{\omega}}^{r}
(\Omega_{1}\times\widetilde{\Omega})}
\leq Cr^{\frac{3}{2}}\epsilon^{2}
\left[\left\|\gamma_{0}\right\|_{\mathfrak{S}^{2}}+1\right].\label{8.020}
\end{align}
Moreover, we have
\begin{eqnarray}
&&\lim\limits_{t\longrightarrow0}(\mathbb{P}\times\widetilde{\mathbb P})
\left(\left\{(\omega_{1}, \widetilde{\omega})\in
 (\Omega_{1}\times \widetilde{\Omega})|
|F(t,\omega_{1},\widetilde{\omega})|>Ce
\left[\left\|\gamma_{0}\right\|_{\mathfrak{S}^{2}}+1\right]
\epsilon^{\frac{1}{2}}\left(\epsilon ln \frac{1}{\epsilon}\right)^{3/2}\right\}\right)
\nonumber\\&&=0.\label{8.021}
\end{eqnarray}
Here $F(t,\omega_{1},\widetilde{\omega})=
\sum\limits_{n=1}^{\infty}\lambda_{n}g^{(2)}_{n}(\widetilde{\omega})
|f_{n}^{\omega_{1}}|^{2}-
\sum\limits_{n=1}^{\infty}\lambda_{n}g^{(2)}_{n}(\widetilde{\omega})
|e^{it\triangle}f_{n}^{\omega_{1}}|^{2}$ and $M_{1}$ appears in (\ref{8.010}) and
\begin{eqnarray*}
\rho_{\gamma_{0}^{\omega_{1},\>\widetilde{\omega}}}=
\sum_{n=1}^{\infty}\lambda_{n}g^{(2)}_{n}(\widetilde{\omega})|f_{n}^{\omega_{1}}|^{2},
\rho_{e^{it\Delta}\gamma_{0}^{\omega_{1},\>\widetilde{\omega}}e^{-it\Delta}}
=\sum_{n=1}^{\infty}\lambda_{n}g^{(2)}_{n}(\widetilde{\omega})|e^{it\Delta}f_{n}^{\omega_{1}}|^{2}.
\end{eqnarray*}
and  $\rho_{\gamma_{0}^{\omega_{1},\>\widetilde{\omega}}}, \rho_{e^{it\Delta}
   \gamma_{0}^{\omega_{1},\>\widetilde{\omega}}e^{-it\Delta}}$ denote the
 density function of $\gamma_{0}^{\omega_{1},\>\widetilde{\omega}}, e^{it\Delta}
   \gamma_{0}^{\omega_{1},\>\widetilde{\omega}}e^{-it\Delta}$, respectively.

\end{Theorem}
\noindent{\bf Proof.} Obviously,
\begin{eqnarray}
|e^{it\triangle}f_{n}^{\omega_{1}}|^{2}-|f_{n}^{\omega_{1}}|^{2}=I_{1}+I_{2},\label{8.022}
\end{eqnarray}
where
\begin{eqnarray}
I_{1}=\left(e^{it\triangle}f_{n}^{\omega_{1}}-f_{n}^{\omega_{1}}\right)e^{-it\triangle}\bar{f}_{n}^{\omega_{1}},
I_{2}=\left(e^{-it\triangle}\bar{f}_{n}^{\omega_{1}}-\bar{f}_{n}^{\omega_{1}}\right)f_{n}^{\omega_{1}}.\label{8.023}
\end{eqnarray}
By using (\ref{8.022}), (\ref{8.023}), the Minkowski inequality and the H\"older
 inequality as well as Lemmas 8.2, 8.3, for
$\forall \epsilon>0$ which appears in Lemma 8.2,
 $\exists \delta>0(<\epsilon^{2} M_{1}^{-2}),$ when $|t|<\delta,$
we have
\begin{eqnarray}
&&\left\|\sum_{n=1}^{\infty}\lambda_{n}g^{(2)}_{n}(\widetilde{\omega})|f_{n}^{\omega_{1}}|^{2}-
\sum_{n=1}^{\infty}\lambda_{n}g^{(2)}_{n}(\widetilde{\omega})|e^{it\triangle}f_{n}^{\omega_{1}}
|\right\|_{L_{\omega_{1},\>\widetilde{\omega}}^{r}
(\Omega\times\widetilde{\Omega})}\nonumber\\
&&=\left\|\sum_{n=1}^{\infty}\lambda_{n}g^{(2)}_{n}(\widetilde{\omega})(I_{1}+I_{2})\right\|_{L_{\omega,\>\widetilde{\omega}}^{r}
(\Omega_{1}\times\widetilde{\Omega})}\nonumber\\
&&\leq \left\|\sum_{n=1}^{\infty}\lambda_{n}g^{(2)}_{n}(\widetilde{\omega})I_{1}\right\|_{L_{\omega,\>\widetilde{\omega}}^{r}
(\Omega_{1}\times\widetilde{\Omega})}+\left\|\sum_{n=1}^{\infty}\lambda_{n}g^{(2)}_{n}(\widetilde{\omega})
I_{2}\right\|_{L_{\omega_{1},\>\widetilde{\omega}}^{r}
(\Omega_{1}\times\widetilde{\Omega})}\nonumber\\
&&\leq Cr^{1/2}\left[\left\|\lambda_{n}I_{1}\right\|_{L_{\omega_{1}}^{r}l_{n}^{2}}
+\left\|\lambda_{n}I_{2}\right\|_{L_{\omega_{1}}^{r}l_{n}^{2}}\right]\nonumber\\
&&\leq Cr^{1/2}\left[\left\|\lambda_{n}I_{1}\right\|_{l_{n}^{2}L_{\omega_{1}}^{r}}
+\left\|\lambda_{n}I_{2}\right\|_{l_{n}^{2}L_{\omega_{1}}^{r}}\right]\nonumber\\
&&\leq Cr^{1/2}\left[\left\|\lambda_{n}
\left\|\left(e^{it\triangle}f_{n}^{\omega_{1}}-f_{n}^{\omega_{1}}\right)
\right\|_{L_{\omega_{1}}^{2r}}\|e^{-it\triangle}\bar{f}_{n}^{\omega_{1}}\|_{L_{\omega_{1}}^{2r}}\right\|_{l_{n}^{2}}
\right]\nonumber\\&&\qquad\qquad+Cr^{1/2}\left[\left\|\lambda_{n}
\left\|\left(e^{-it\triangle}\bar{f}_{n}^{\omega_{1}}-\bar{f}_{n}^{\omega_{1}}\right)
\right\|_{L_{\omega_{1}}^{2r}}\|f_{n}^{\omega_{1}}\|_{L_{\omega_{1}}^{2r}}\right\|_{l_{n}^{2}}
\right]\nonumber\\
&&\leq Cr^{\frac{3}{2}}(|t|M_{1}^{2}+\epsilon^{2})\left[\left\|\gamma_{0}\right\|_{\mathfrak{S}^{2}}+1\right]\nonumber\\
&&\leq Cr^{\frac{3}{2}}\epsilon^{2}\left[\left\|\gamma_{0}\right\|_{\mathfrak{S}^{2}}+1\right].\label{8.024}
\end{eqnarray}
Combining (\ref{8.024}) with Lemma 8.5,
we have that (\ref{8.021}) is valid since
$Ce\epsilon^{1/2}\left(\epsilon ln
\frac{1}{\epsilon}\right)^{3/2}$
is  monotonically increasing with respect to $\epsilon\in (0,\frac{1}{2}]$ and
\begin{eqnarray*}
\lim\limits_{\epsilon\longrightarrow 0^{+}}Ce\epsilon^{1/2}\left(\epsilon ln
\frac{1}{\epsilon}\right)^{3/2}=0.
\end{eqnarray*}

This completes the proof of Theorem 8.1.

\noindent{\bf Remark 13.} (\ref{8.021}) is called {\bf the probabilistic
 convergence of density functions
   of some compact operators with full randomization on $\Theta.$}

\bigskip
\bigskip

\leftline{\large \bf Acknowledgments}

This work is supported by National Natural Science Foundation of China under the grant number  12371242.

\end{document}